\documentclass[a4paper,11pt]{article}
\usepackage{amsmath,amssymb,amsthm,a4wide,color}
\usepackage[utf8]{inputenc}
\usepackage[T1]{fontenc}
\usepackage{authblk,mathrsfs,empheq}
\usepackage{mathtools}
\usepackage{algorithm,algpseudocode}
\usepackage{graphicx,subcaption}
\usepackage{esint}

\usepackage{multirow}

\usepackage[colorlinks=true,citecolor=blue]{hyperref}

\usepackage{bm}

\usepackage{booktabs}
\usepackage{ragged2e}

\usepackage{caption}
\usepackage{subcaption}
\DeclareCaptionLabelFormat{boldfig}{\textbf{Fig.~\textbf{#2}}}
\usepackage{pdfsync}
\usepackage{todonotes}
\usepackage[capitalise,noabbrev]{cleveref}

\input{macro}

\title{\bf Discrete FEM-BEM coupling with the\\
  Generalized Optimized Schwarz Method}
\author[1,2]{Antonin Boisneault\orcidID{0000-0001-7986-9048} \thanks{Corresponding author: antonin.boisneault@inria.fr}}
\author[2]{Marcella Bonazzoli\orcidID{0000-0002-0284-5643}}
\author[1]{Xavier Claeys\orcidID{0000-0003-0826-6244}}
\author[1]{Pierre Marchand\orcidID{0000-0002-2522-6837}}

\affil[1]{{\small POEMS, CNRS, Inria, ENSTA, Institut Polytechnique de
    Paris, 91120 Palaiseau, France} }
\affil[2]{{\small Inria, Unité de Mathématiques Appliquées, ENSTA, Institut
    Polytechnique de Paris, 91120 Palaiseau, France}}
\date{}

\begin{document}

\maketitle

\begin{abstract}
The present contribution aims at developing a non-overlapping Domain
Decomposition (DD) approach to the solution of acoustic wave propagation
boundary value problems based on the Helmholtz equation, on both bounded
and unbounded domains.  This DD solver, called Generalized Optimized
Schwarz Method (GOSM), is a substructuring method, that is, the unknowns of
an iteration are associated with the subdomains interfaces.  We extend the
analysis presented in~\cite{MR4665035} to a fully discrete setting. We do
not consider only a specific set of boundary conditions, but a whole class
including, e.g., Dirichlet, Neumann, and Robin conditions.  Our analysis
will also cover interface conditions corresponding to a Finite Element
Method - Boundary Element Method (FEM-BEM) coupling. In particular, we
shall focus on three classical FEM-BEM couplings, namely the Costabel,
Johnson-N\'ed\'elec and Bielak-MacCamy couplings.  As a remarkable outcome,
the present contribution yields well-posed substructured formulations of
these classical FEM-BEM couplings for wavenumbers different from classical
spurious resonances. We also establish an explicit relation between the
dimensions of the kernels of the initial variational formulation, the local
problems and the substructured formulation. That relation especially holds
for any wavenumber for the substructured formulation of Costabel FEM-BEM
coupling, which allows us to prove that the latter formulation is
well-posed even at spurious resonances.  Besides, we introduce a
systematically geometrically convergent iterative method for the Costabel
FEM-BEM coupling, with estimates on the convergence speed.
\end{abstract}

\noindent\textbf{Acknowledgements:}{ \small This work is funded by the Inria program “Actions exploratoires” (OptiGPR3D)}

\noindent\textbf{Mathematics Subject Classification:}{ 65R20 \(\cdot\) 65N22 \(\cdot\) 65N38 \(\cdot\) 35J05 \(\cdot\) 65N55}

\section*{Introduction}

In this work, our aim is to design an efficient solver for acoustic wave
propagation modelled by the Helmholtz equation. The objective is to develop
a numerical method that is versatile enough to accommodate heterogeneous
and complex environments, while remaining efficient and tailored to high
performance computing.  Two classes of methods are often used for the
numerical simulation of time-harmonic wave propagation: the finite element
method (FEM), which is versatile and suited to heterogeneous materials, and
the boundary element method (BEM), which performs best on piecewise
homogeneous materials and/or unbounded domains. In order to optimally handle
different types of materials, we seek a numerical algorithm to
systematically couple and benefit from these two methods.

To this end, we introduce a new FEM-BEM coupling strategy based on a substructuring domain decomposition method called \emph{Generalized Optimized Schwarz Method} (GOSM). Substructuring DD methods involve
\begin{itemize}
  \item a non-overlapping partition of the domain of interest,
  \item resolution of local problems with optimized boundary conditions in each subdomain that can be performed in parallel,
  \item an iterative resolution of a substructured problem, that is, a problem on the skeleton (the union of the boundaries of the subdomains).
\end{itemize}
The novelty of our approach lies in the extension to FEM-BEM coupling of recent developments from~\cite{MR4433119,MR4665035}. One of the strengths of this current work is a clear well-posedness and convergence theory.

Research on domain decomposition methods for time-harmonic wave propagation started with a proper choice of transmission condition at the interface between subdomains. In~\cite{Despres1991DDM}, a first convergent DD method was introduced for the Helmholtz equation using impedance transmission conditions, where the impedance coefficient is \(\imath \kappa\) and \(\kappa\) the wavenumber. To accelerate the possibly slow convergence, several approaches have been explored to replace the impedance coefficient by a local transmission operator, see e.g.~\cite{GanderMagoulesEtAl2002OSM,BoubendirAntoineEtAl2012QON}.

Another approach studied in \cite{CollinoGhanemiEtAl2000DDM} is to replace the impedance coefficient with a non-local impedance operator, which ensures provable geometric convergence. A strong requirement for such a fast convergence is the absence of cross-points, i.e.\ points shared by more than two subdomains in the domain partition. All the previous methods are usually called Optimized Schwarz Methods (OSM). To treat cross-points,~\cite{MR4433119,MR4665035} recently introduced the GOSM, which uses a non-local impedance operator, and also a non-local exchange operator \(\Pi\) instead of the usual swap operator between subdomains, used in the OSMs.

Classical FEM-BEM couplings, such as the Costabel~\cite{Costabel1987SMC},
Johnson-Nédélec~\cite{JohnsonNedelec1980CBI}, and
Bielak-MacCamy~\cite{MR700668} couplings, rely on Dirichlet and Neumann
transmission conditions between the FEM and BEM subdomains. Meanwhile, new
FEM-BEM couplings have been developped building on the progress of DD
methods for the Helmholtz equation. For example, the Després impedance
condition has been used in~\cite{BendaliBoubendirEtAl2007FDD},
while~\cite{Caudron2018CFB,CaudronAntoineEtAl2020OWC} employ local
transmission operators as in~\cite{BoubendirAntoineEtAl2012QON}. All the
previous works consider a partition into two subdomains, one bounded FEM
subdomain and one unbounded BEM subdomain. With more than two subdomains,
the additional challenge related to cross-points has not been often studied
in the literature. In~\cite{BoubendirBendaliEtAl2008CNO}, one additional
unknown per cross-point is considered, which results in solving, per
iteration, one additional (sparse) linear system in each FEM subdomain, and
finally a dense linear system (whose size equals the number of
cross-points). Citing~\cite{BoubendirBendaliEtAl2008CNO}, ``the local
problems remain coupled at the cross-points''. In~\cite{MR4779883}, a
proper functional framework dealing with cross-points has been designed,
and new multidomain FEM-BEM formulations have been proposed and analyzed in
a continuous setting, but their practical implementation is not
straightforward.

We propose here a discrete counterpart of the analysis presented in~\cite{MR4665035} and extend it to FEM-BEM coupling. Note that in~\cite{MR4665035}, only bounded domains are considered, and physical conditions on their boundary are weakly imposed via additional boundary unknowns. The classical FEM-BEM couplings are reformulated as~\eqref{IntialBVP6}, that is, a substructured equation of the form \(\Id+\Pi\mS\) where 
\begin{itemize}
  \item \(\Pi\) is the non-local exchange operator mentioned above,
  \item \(\mS\) is a scattering operator, which takes ingoing impedance traces, solves the local problems and returns the outgoing impedance traces. 
\end{itemize}
We prove that the substructured reformulations are well-posed\footnote{\label{note:well_posed}Well-posedness is intended in the classical Hadamard sense (and not in the weaker Fredholm sense).}, when the classical FEM-BEM couplings are. 
Besides, when the considered FEM-BEM coupling satisfies a sign assumption related to physical dissipation (see Equation~\eqref{Assumption2}), \(\Pi \mS\) is a contraction\footnote{\label{note:contraction}The terms \emph{contractive} and \emph{contraction} are employed here as in~\cite{Brezis2011FAS}: the inequality is not necessarily strict.} (see Lemma~\ref{CaracExchangeOp} and Proposition~\ref{ProposistionContractivity}). Thus, the substructured system can be solved using a Richardson solver with guaranteed geometric convergence (see Theorem~\ref{thm:richardson_upper_bound}), even when cross-points are present. The sign assumption~\eqref{Assumption2} is satisfied by the Costabel coupling (Proposition~\ref{prop:costabel_sign_assumption}), but not by the Johnson-Nédélec and Bielak-MacCamy couplings (Proposition~\ref{prop:JN_sign_assumption}), for which a GMRes~\cite{Saad2003IMS} iterative solver can be used instead (when these couplings are well-posed).

Additionally, when the sign assumption~\eqref{Assumption2} is satisfied, an
explicit relation is established between the dimension of the kernels of
the classical variational formulation, the local problems, and the
substructured formulation (see
Theorem~\ref{EquivalenceDimensionResonance}). Note that classical FEM-BEM
couplings can have a non-trivial kernel even when the initial boundary
value problem is well-posed, a phenomenon commonly referred to as
\emph{spurious resonances}, see e.g.~\cite{SchulzHiptmair2022}.

Finally, we explain the implementation details of the GOSM, and illustrate its properties with thorough numerical experiments. Remark that subdomains can be associated either with volume unknowns for FEM subdomains, or with boundary unknowns for BEM subdomains and physical boundary conditions. Thus, we discuss the choice of the impedance operators, and we observe that convergence is improved by picking them according to the nature of the subdomains (FEM vs BEM vs boundary condition). Moreover, we show that \(h\)-robust geometric convergence can be achieved even with cross-points.

The article is structured as follows. After introducing notations in Section~\ref{NotationConventions} and the considered boundary value problems in Section~\ref{sec:pbs}, we recall the classical FEM-BEM couplings and discuss the validity of the sign assumption~\eqref{Assumption2} for their formulations. Section~\ref{sec:discr} presents the discretization framework used throughout the paper, while Section~\ref{sec:ddm} introduces the domain decomposition building blocks, such as the impedance and exchange operators. The GOSM formulation is derived in Section~\ref{sec:problem_reformulation} based on the scattering operator \(\mS\), and its well-posedness is treated in Section~\ref{sec:resonances}. A relation between the infsup constants of the initial variational problem and its substructured reformulation is given in Section~\ref{sec:infsup_estimates}. Section~\ref{sec:general_problems} is about the extension to more general settings, and Section~\ref{sec:numerical_experiments} presents the numerical investigation.

\section{Notation conventions}\label{NotationConventions}

To proceed further into the description and the analysis of the
boundary value problem and the domain decomposition method we wish to
study, we need now to fix a few notational conventions. The present
section is mainly here for reference, as most of the conventions we
adopt are classical. All vector spaces that we are going to consider have $\mathbb{C}$ as a 
scalar field.

\paragraph{Dual spaces} Assuming that $\mH$ is a Banach space equipped with the
norm $\Vert \cdot\Vert_{\mH}$, its topological dual, denoted $\mH^{*}$, 
will systematically be equipped with the norm
\begin{equation*}
  \Vert \varphi\Vert_{\mH^{*}} \coloneq
  \sup_{v\in \mH\setminus\{0\}} \frac{\vert \varphi(v) \vert}{\Vert v\Vert_{\mH}}.
\end{equation*}
The canonical duality pairing will be denoted
$\langle \cdot,\cdot\rangle\colon \mH^{*}\times\mH\to \mathbb{C}$ and
defined by $\langle \varphi, v\rangle\coloneq \varphi(v)$. Although
the space $\mH$ does not appear explicitly in the notation
``$\langle \varphi, v\rangle$'', when such pairing angle brackets
are used, it shall be clear from the context which pair of spaces
$(\mH,\mH^{*})$ is under consideration. We emphasize
that the duality pairings we consider do \emph{not} involve any
complex conjugation. We shall write $\langle v,\varphi\rangle =
\langle \varphi, v\rangle,\ \forall v\in\mH, \varphi\in \mH^*$ indifferently.

\paragraph{Inf-sup constants and adjoints}
Assuming that $\mV$ is another Banach space
equipped with the norm $\Vert \cdot\Vert_{\mV}$,
and $\mL\colon\mH\to \mV$ is a bounded linear map, we shall
refer to its \emph{inf-sup constant} denoted and defined as follows
\begin{equation}\label{DefInfSupCst}
  \infsup_{\mH\to \mV}(\mL)\coloneq\inf_{u\in \mH\setminus\{0\}}
  \frac{\Vert \mL(u)\Vert_{\mV}}{\Vert u\Vert_{\mH}}.
\end{equation}
In the case where $\mL$ is invertible, this inf-sup constant
equals the inverse of the continuity modulus of $\mL^{-1}$.
The inf-sup constant is well-defined even if $\mL$ is not
invertible though. The \emph{adjoint}\footnote{This adjoint operator, which does not involve complex conjugation, is sometimes referred to as \emph{quasi-adjoint}, see e.g.~\cite{ChandlerWildeGrahamEtAl2012NAB}. We follow here the notational convention of~\cite[\S2.6]{Brezis2011FAS} and~\cite[\S4.10]{Rudin19991book}.} of the map $\mL\colon\mH\to \mV$ shall be
defined as the unique bounded linear map $\mL^*\colon\mV^{*}\to \mH^{*}$
satisfying
\begin{equation}\label{DefAdjoint}
  \langle \mL^{*}(p),u\rangle \coloneq 
  \langle p,\mL(u)\rangle 
\end{equation}
for all $p\in \mV^*$ and all $u\in\mH$. When \(V=H\), the map \(\mL\) is said \emph{symmetric} if \(\mL^* = \mL\). Once again, we insist that
no complex conjugation comes into play in~\eqref{DefAdjoint}. The bounded
linear map $\mL$ induces another bounded linear map
$\overline{\mL}\colon\mH\to \mV$ defined by $\overline{\mL}(\overline{u})\coloneq\overline{\mL(u)}$
for all $u\in\mH$.

\paragraph{Annihilator}

For any subset $\mX\subset \mH$, we denote its \emph{annihilator} (or \emph{polar set}) by
\begin{equation}\label{DefPolarSet}
  \mX^{\circ} \coloneq\{ \varphi\in \mH^* \mid \langle \varphi,v\rangle = 0,\ \forall v\in \mX\},
\end{equation}
see, e.g.~\cite[§4.6]{Rudin19991book},~\cite[§1.3]{Brezis2011FAS}
and~\cite[§4.1.4--4.1.5]{BoffiBrezziFortin2013book}. We point out three
important properties.  First, using~\cite[Theorem 2.16]{Brezis2011FAS}, let
\(\mY\) and \(\mZ\) be two closed subspaces in \(\mH\). If \(\mZ+\mY\) is
closed in \(\mH\), then
\begin{align}\label{PropSumOfAnnihilator}
   \mZ^{\circ} + \mY^{\circ} = \left(\mZ \cap \mY\right)^{\circ}.
\end{align}
By applying Definition~\eqref{DefPolarSet} to \(\mrm{Im}(\mL) \subset \mV\), note that for \emph{any} bounded linear map \(\mL:\mH\to\mV\),
\begin{align}\label{eq:polarOfImage}
  \mrm{Im}(\mL)^{\circ} = \mrm{Ker} (\mL^*).
\end{align}
Finally, if \(\mL:\mH\to\mV\) is a bounded linear map with \emph{closed range}, then
\begin{align}\label{ClosedRangeTheorem}
  \mrm{Im}(\mL) = \mrm{Ker} (\mL^*)^{\circ}.
\end{align}

\paragraph{Scalar products} A bounded linear operator $\mT\colon\mH\to \mH^*$
is called \emph{self-adjoint} if $\overline{\mT} = \mT^{*}$ and, in this case
we have $\langle \mT(u),\overline{u}\rangle\in \RR$ for all $u\in\mH$.
It is called \emph{positive definite} if $\langle \mT(u),\overline{u}\rangle
\in (0,+\infty)$ for all $u\in\mH\setminus\{0\}$. If $\mT$ is
both self-adjoint and positive definite, the sesquilinear form
$u,v\mapsto \langle \mT(u),\overline{v}\rangle$ induces a scalar
product over $\mH$ and the associated norm is denoted
\begin{equation*}
  \Vert u\Vert_{\mT}\coloneq\sqrt{\langle \mT(u),\overline{u}\rangle}.
\end{equation*}

\paragraph{Cartesian products} We shall also consider Cartesian products
$\mH_{1}\times \dots\times \mH_{\mJ}$ where each $\mH_{j}$ is a Banach space
equipped with the norm $\Vert\cdot\Vert_{\mH_{j}}$. Then the Cartesian product
shall be equipped with the following canonical norm and duality pairings
\begin{equation*}
  \begin{aligned}
    & \Vert \bv\Vert_{\mH_{1}\times \dots\times \mH_{\mJ}}^{2} \coloneq
    \Vert v_1\Vert_{\mH_{1}}^{2}+\dots +\Vert v_{\mJ}\Vert_{\mH_{\mJ}}^{2}\\
    & \langle \bv,\bq \rangle \coloneq
    \langle v_1,q_1\rangle + \dots + \langle v_\mJ,q_\mJ\rangle
  \end{aligned}
\end{equation*}
for $\bv = (v_1,\dots,v_\mJ), v_j\in\mH_{j}$, and
$\bq = (q_1,\dots,q_{\mJ}), q_{j}\in\mH^{*}_{j}$.
If $\mV_{j},j=1,\dots,\mJ$, is another collection of Banach spaces
and $\mL_{j}\colon\mH_{j}\to \mV_{j}$ are bounded linear maps, we
shall also consider the block-diagonal operator
$\mrm{diag}(\mL_{1},\dots,\mL_{\mJ})$,
mapping $\mH_{1}\times \dots\times \mH_{\mJ}$ onto
$\mV_{1}\times \dots\times \mV_{\mJ}$ and defined,
for $\bv = (v_1,\dots,v_\mJ), v_j\in\mH_{j}$, and
$\bq = (q_1,\dots,q_{\mJ}), q_j\in\mV_{j}^*$, by
\begin{equation*}
  \langle \bq, \mrm{diag}(\mL_{1},\dots,\mL_{\mJ})\,\bv\rangle \coloneq
  \langle q_1,\mL_{1}(v_1)\rangle + \dots +
  \langle q_{\mJ},\mL_{\mJ}(v_{\mJ}) \rangle.
\end{equation*}

\paragraph{Pseudo-inverses} Finally, we shall also need to refer to the Moore-Penrose pseudo-inverse
(or generalized inverse). We recall briefly its definition, and refer
our reader to e.g.~\cite{MR1987382,OriginalMooreRef,MR69793} for more
details. Assume that $(\mH,\Vert\cdot\Vert_{\mH})$ and
$(\mV,\Vert\cdot\Vert_{\mV})$ are two Hilbert spaces and $\mL\colon
\mH\to \mV$ is a bounded linear map with closed range. By definition,
its \emph{Moore-Penrose pseudo-inverse} is the unique bounded linear
operator $\mL^{\dagger}\colon\mV\to \mH$ satisfying the four
conditions
\begin{equation}\label{MoorePenroseProperties}
  \begin{aligned}
    (i) \quad & \mL\cdot\mL^{\dagger}\cdot\mL =\mL\\
    (ii) \quad & \mL^{\dagger}\cdot\mL\cdot\mL^{\dagger} =\mL^{\dagger}\\
    (iii) \quad & \mL^{\dagger}\cdot\mL\colon \mH\to \mH\;\text{is an orthogonal projector}\\
    (iv) \quad & \mL\cdot\mL^{\dagger}\colon \mV\to \mV\;\text{is an orthogonal projector}\\
  \end{aligned}
\end{equation}
The properties above imply in particular that
\begin{subequations}\label{MoorePenroseConsequences}
  \begin{align}
    \mrm{Im}(\mL\cdot\mL^{\dagger}) &= \mrm{Im}(\mL), \label{MoorePenroseRange}\\
    \mrm{Ker}(\mL^{\dagger}\cdot\mL) &= \mrm{Ker}(\mL). \label{MoorePenroseKernel}
  \end{align}
\end{subequations}
In the case where
$\mL\colon\mH\to \mV$ is bounded and surjective, its pseudo-inverse is
characterized by
\begin{equation}\label{CaracPseudoInverse}
  \mL\cdot\mL^{\dagger} = \Id\quad
  \text{and}\quad
  \Vert \mL^{\dagger}(u)\Vert_{\mH} = \min\{\Vert v\Vert_{\mH},\; \mL(v) = u\}.
\end{equation}
Similarly, in the case where $\mL\colon\mH\to \mV$ is bounded and
one-to-one, the pseudo-inverse is characterized as a left inverse
$\mL^{\dagger}\cdot\mL = \Id$ satisfying the optimality property
$\Vert u-\mL\cdot\mL^{\dagger}(u)\Vert_{\mV} = \min\{\Vert
u-\mL(v)\Vert_{\mV},\;v\in \mH\}$.

\section{Problems under consideration}\label{sec:pbs}
In dimension $d=2$ or $3$, let $\Omega\subset \RR^d$ be a polyhedral
bounded open set, and denote $\Gamma \coloneq \partial\Omega$ its
boundary. We shall consider the space of square integrable measurable
functions $\mL^{2}(\Omega)\coloneq\{ v\colon\Omega\to \CC, \Vert
v\Vert_{\mL^{2}(\Omega)}^{2}\coloneq \int_{\Omega}\vert
v(\bx)\vert^{2} d\bx <+\infty\}$, and the Sobolev space
$\mH^{1}(\Omega)\coloneq\{ v\in \mL^{2}(\Omega), \nabla v\in
\mL^{2}(\Omega)\}$ equipped with $\Vert
v\Vert_{\mH^{1}(\Omega)}^{2}\coloneq \Vert \nabla
v\Vert_{\mL^{2}(\Omega)}^{2}+ \tilde{k}^{2}\Vert
v\Vert_{\mL^{2}(\Omega)}^2$, where $\tilde{k}>0$ is some fixed
parameter. We shall also consider $\mH^{1}_0(\Omega)$, which is the
closure of
$\mathscr{C}^{\infty}_{0,c}(\Omega)\coloneq\{\varphi\vert_{\Omega},
\varphi\in \mathscr{C}^{\infty}(\RR^d), \text{with bounded support
}\mrm{supp}(\varphi)\subset \Omega\}$ for $\Vert
\cdot\Vert_{\mH^{1}(\Omega)}$, and the space
$\mH^{1}_{\text{loc}}(\RR^d) \coloneq \{v, \varphi v \in
\mH^{1}(\RR^d) \,\forall \varphi \in
\mathscr{C}^{\infty}_{c}(\RR^d)\}$, where
$\mathscr{C}^{\infty}_{c}(\RR^d)$ is the set of compactly supported
$\mathscr{C}^{\infty}$ functions. Lastly, we consider a source term $f\in
\mL^{2}(\RR^d)$ supported in $\Omega$, and a measurable function 
$\kappa\colon\RR^d\to \CC$ modelling a wavenumber and satisfying
\begin{equation}\label{Assumption1}
  \begin{aligned}
    & \textbf{Assumption:}\\
    & \Im\{\kappa(\bx)^2\}\geq 0 \quad \forall \bx\in\Omega,
  \end{aligned}
\end{equation} 
as well as $\kappa(\bx) = \tilde{k}>0$ for $\bx\in
\RR^d\setminus\Omega$, and $\mathop{\sup\text{ess}}_{\RR^d}\vert \kappa\vert <+\infty$.
We wish to develop a solution strategy for boundary value problems
modelling acoustic wave propagation and taking one of the following
generic forms
\begin{subequations}\label{IntialBVP}
  \begin{align}
  & \begin{aligned}
    & u\in \mH^{1}(\Omega)\;\;\text{and}\;\; -\Delta u-\kappa^{2}u = f\;\;\text{in}\;\Omega\\
    & +\;\text{boundary condition on }\Gamma = \partial\Omega,
  \end{aligned} \label{IntialBVPa}\\
  \notag\\
  & \begin{aligned}
    & u\in \mH^{1}_{\text{loc}}(\RR^d)\;\;\text{and}\;\; -\Delta u-\kappa^{2}u = f\;\;\text{in}\;\RR^d\\
    & +\;\text{Sommerfeld's radiation condition}.
  \end{aligned} \label{IntialBVPb}
  \end{align}
\end{subequations}
Note that in~\eqref{IntialBVPb} $\Gamma = \partial \Omega$ is seen as a fictitious interface on which are imposed Dirichlet and Neumann transmission conditions. Sommerfeld's radiation condition at infinity can be expressed as (see, e.g.~\cite[§2.6.5]{MR1822275}) 
\begin{equation}\label{SommerfeldRadiationCondition}
  \begin{aligned}
    &\textbf{Sommerfeld's condition:}\\
    &\lim_{r\to \infty}\int_{\partial\mathcal{B}_r}\vert \partial_{r}u-\imath k u\vert^{2} \diff s = 0
  \end{aligned}
\end{equation}
where $\mathcal{B}_r$ is the ball centered at $0$ with radius $r$,
$\partial_r$ denotes the radial derivative, and $\diff s$ refers to the
Lebesgue surface measure on $\partial\mathcal{B}_r$.

Let us examine how to write a variational formulation of Problems~\eqref{IntialBVP}. Taking the cue from~\cite{MR4665035} to deal with
boundary value problems of the form~\eqref{IntialBVPa}, we first
 introduce $\mathcal{A}_{\Omega}\colon \mH^{1}(\Omega)\to
\mH^{1}(\Omega)^*$ and $\ell_{\Omega}\in \mH^{1}(\Omega)^*$ that, for
all $u,v\in \mH^{1}(\Omega)$, are defined by
\begin{equation}\label{HelmholtzOperator}
  \begin{aligned}
    & \langle \mathcal{A}_{\Omega}(u),v\rangle \coloneq
    \int_{\Omega}(\nabla u\cdot\nabla v  -\kappa^{2} u v)\,d\bx,\\
    & \langle \ell_{\Omega},v\rangle\coloneq\int_{\Omega}f v \,d\bx.
  \end{aligned}
\end{equation}
To \emph{weakly} impose the boundary condition of \eqref{IntialBVPa}, first we consider the trace space
$\mH^{1/2}(\Gamma)\coloneq\{v\vert_{\Gamma}, v\in \mH^{1}(\Omega)\}$,
equipped with the norm $\Vert
v\vert_{\Gamma}\Vert_{\mH^{1/2}(\Gamma)}\coloneq \inf_{v_0\in
  \mH^{1}_0(\Omega)}\Vert v+v_0\Vert_{\mH^{1}(\Omega)}$. Its dual
shall be denoted $\mH^{-1/2}(\Gamma)\coloneq
\mH^{1/2}(\Gamma)^*$, with $\Vert
p\Vert_{\mH^{-1/2}(\Gamma)}\coloneq \sup_{v\in
  \mH^{1/2}(\Gamma)\setminus\{0\}} \vert \langle p,v\rangle\vert/
\Vert v\Vert_{\mH^{1/2}(\Gamma)}$. Then
we introduce a bounded operator
\begin{equation*}
  \mathcal{A}_\Gamma\colon \mH^{+\frac{1}{2}}(\Gamma)\times \mH^{-\frac{1}{2}}(\Gamma)\to
    \mH^{-\frac{1}{2}}(\Gamma)\times \mH^{+\frac{1}{2}}(\Gamma)
\end{equation*}
whose precise definition depends on the boundary/interface condition we
wish to take into account in~\eqref{IntialBVP}. Several boundary
conditions for~\eqref{IntialBVPa} have been discussed in~\cite{MR4665035}. For example
a Dirichlet boundary condition can be weakly imposed by choosing (see~\cite[Example 3.1]{MR4665035})
\begin{align}\label{DirichletBC}
  \textbf{Dirichlet:}\hspace{1cm} &
  \mathcal{A}_{\Gamma} =
  \left\lbr\begin{array}{cc}
  0 & \Id\\
  \Id & 0
  \end{array}\right\rbr.
\end{align}
A Neumann boundary condition can be weakly imposed by taking (see~\cite[Example 3.2]{MR4665035})
\begin{align}\label{NeumannBC}
  \textbf{Neumann:}\hspace{1cm} &
  \mathcal{A}_{\Gamma} =
  \left\lbr\begin{array}{cc}
  0 & 0\\
  0 & \mathcal{T}_{\Gamma}^{-1}
  \end{array}\right\rbr, 
\end{align}
where $\mathcal{T}_\Gamma\colon \mH^{1/2}(\Gamma)\to
\mH^{-1/2}(\Gamma)$ is the bounded linear map defined
by \(\langle\mathcal{T}_\Gamma(v),\overline{v}\rangle \coloneq \Vert
v\Vert_{\mH^{1/2}(\Gamma)}^{2}\), \(\forall v\in \mH^{1/2}(\Gamma)\). 
Due to~\eqref{Assumption1}, $\Im\{\langle \mathcal{A}_{\Omega}(v),
\overline{v}\rangle\}\leq 0,\ \forall v\in \mH^{1}(\Omega)$ and, as shown 
in~\cite{MR4665035}, a similar property holds for the boundary
operators $\mathcal{A}_{\Gamma}$ in~\eqref{DirichletBC} and~\eqref{NeumannBC}, and for those associated with other classical boundary
conditions (Robin, mixed Dirichlet-Neumann). We state this property as
an assumption that shall be further discussed in the subsequent
analysis
\begin{equation}\label{Assumption2}
  \begin{aligned}
    & \textbf{Assumption:}\\
    & \Im\{\langle \mathcal{A}_{\Gamma}(v,q),(\overline{v},\overline{q})\rangle\}\leq 0\\
    & \forall v\in\mH^{+\frac{1}{2}}(\Gamma),\;\;\forall q\in\mH^{-\frac{1}{2}}(\Gamma).
  \end{aligned}
\end{equation}
As discussed in~\cite{MR4665035}, once the operator
$\mathcal{A}_{\Gamma}$ has been defined in accordance with the
boundary condition to be imposed, the boundary value problem~\eqref{IntialBVPa} can be reformulated into the following variational form
\begin{equation}\label{IntialBVP2}
  \begin{aligned}
    & \text{Find}\;u\in \mH^{1}(\Omega), \;p\in
    \mH^{-\frac{1}{2}}(\Gamma)\;\text{such that}\\
    & \langle \mathcal{A}_{\Omega}(u),v\rangle+\langle
    \mathcal{A}_{\Gamma}(u\vert_{\Gamma},p),
    (v\vert_{\Gamma},q)\rangle = \langle \ell_{\Omega}, v\rangle +
    \langle \ell_{\Gamma},(v\vert_{\Gamma},q)\rangle\\
    & \hspace{5.25cm} \forall v\in \mH^{1}(\Omega), \;
    \forall q\in \mH^{-\frac{1}{2}}(\Gamma),
  \end{aligned}
\end{equation}
where $\ell_{\Omega}\in \mH^1(\Omega)^*$ and $\ell_{\Gamma}\in
\mH^{-\frac{1}{2}}(\Gamma)\times \mH^{+\frac{1}{2}}(\Gamma)$ are known
source terms related to the function $f$ and the right-hand side of the
boundary condition in Problem~\eqref{IntialBVPa}. 

We are going to develop a discrete counterpart of the theory proposed in~\cite{MR4665035}. Choosing $\mathcal{A}_{\Gamma}$ as~\eqref{DirichletBC} or~\eqref{NeumannBC} for~\eqref{IntialBVPa}, two situations can occur: either~\eqref{IntialBVPa}, and thus~\eqref{IntialBVP2}, is well-posed\textsuperscript{\ref{note:well_posed}}, or it is not. The latter can happen for a cavity problem when \(\kappa\) is a resonance. In both cases,~\eqref{IntialBVP2} will be rewritten as a substructured form after decomposing the domain \(\Omega\).

\medskip
In the next section we exhibit other choices of $\mathcal{A}_{\Gamma}$
to reformulate also Problem~\eqref{IntialBVPb} into the generic form~\eqref{IntialBVP2}. This can be achieved by means of FEM-BEM coupling
schemes, which model the Dirichlet and Neumann transmission conditions
on $\Gamma = \partial \Omega$, the Helmholtz equation with constant
wavenumber in $\RR^d\setminus\Omega$, and Sommerfeld's radiation
condition at infinity.

\section{FEM-BEM coupling}\label{FEMBEMCoupling}

We are particularly interested in FEM-BEM coupling and one of the goal
of the present contribution is to design a numerical strategy that
systematically converges, with proven convergence bounds, toward the
solution to a FEM-BEM formulation. In the present section, we 
introduce a few notations related to layer potentials and boundary
integral operators, in order to then define $\mathcal{A}_{\Gamma}$ for
FEM-BEM coupling formulations.

\medskip
Denote $\bx\mapsto \mathscr{G}_{k}(\bx)$ the outgoing \emph{Green kernel} for the Helmholtz operator with constant wavenumber $k>0$. For $d = 3$, this is
$\mathscr{G}_k(\bx)\coloneq \exp(\imath k\vert \bx\vert)/(4\pi\vert \bx\vert)$
and, for $d = 2$, this is $\mathscr{G}_k(\bx)\coloneq i H^{(1)}_{0}(k\vert
\bx\vert)/(4\pi)$, where $z\mapsto H^{(1)}_{0}(z)$ is the $0$-th
order Hankel function of the first kind, see e.g.~\cite[Chap.10]{MR2723248}. For any sufficiently
smooth pair of Dirichlet and Neumann traces $(v,p)$, we define the
\emph{total layer potential operator} by the formula
\begin{equation}\label{LayerPOtentialOperator}
  \begin{aligned}
  \mathcal{G}_{\Gamma}(v,p)(\bx)\coloneq
  \int_{\Gamma} \bigl(\bn(\by)\cdot(\nabla\mathscr{G}_k)(\bx-\by)\,v(\by) & \\
  +\;\mathscr{G}_k(\bx-\by)\,p(\by) \bigr) & \; \diff s(\by),\quad \bx\in \RR^d\setminus\Gamma,
  \end{aligned}
\end{equation}
where $\diff s$ refers to the Lebesgue surface measure on $\Gamma$ and
$\bn\colon\Gamma\to \RR^d$ is the unit vector field normal to $\Gamma$
directed toward the exterior of $\Omega$. The map $(v,p)\mapsto
\mathcal{G}_{\Gamma}(v,p)\vert_{\Omega}$ can be extended by density as
a bounded linear operator $\mH^{1/2}(\Gamma)\times
\mH^{-1/2}(\Gamma)\to \mH^{1}(\Delta,\overline{\Omega})\coloneq\{v\in
\mH^{1}(\Omega),\;\Delta v\in \mL^{2}(\Omega)\}$, see
e.g.~\cite[Thm.3.1.16]{MR2743235} or~\cite[Thm.6.11]{MR1742312}.
Similarly, $(v,p)\mapsto
\mathcal{G}_{\Gamma}(v,p)\vert_{\RR^d\setminus\Omega}$ can be extended
by density as a bounded linear operator $\mH^{1/2}(\Gamma)\times
\mH^{-1/2}(\Gamma)\to
\mH^{1}_{\mrm{loc}}(\Delta,\RR^d\setminus\Omega)$.

\medskip
We stress that, because of the convolutional form of the potential
operator~\eqref{LayerPOtentialOperator} and the properties of the outgoing
Green kernel, for any pair $(v,p)\in \mH^{1/2}(\Gamma)\times\mH^{-1/2}(\Gamma)$,
the function $u(\bx)\coloneq \mathcal{G}_{\Gamma}(v,p)(\bx)$ satisfies
Sommerfeld's condition at infinity~\eqref{SommerfeldRadiationCondition}, see~\cite{Sommerfeld1912,MR29463,MR103705,MR1822275}, as well as the Helmholtz equation with wavenumber $k>0$ in $\RR^d\setminus\Omega$, see e.g.~\cite[§2.4]{CoKr:bookIEM:1983}.

\medskip
Next, we define the \emph{Dirichlet-Neumann
interior trace map} $\gamma_{\Gamma}\colon \mH^{1}(\Delta,\overline{\Omega})\to
\mH^{1/2}(\Gamma)\times \mH^{-1/2}(\Gamma)$ as the unique bounded linear
operator satisfying, for all $\varphi\in \mathscr{C}^{\infty}(\overline{\Omega})
\coloneq\{\varphi\vert_{\Omega},\;\varphi\in\mathscr{C}^{\infty}(\RR^d)\}$,
the formula 
\begin{equation*}
  \gamma_{\Gamma}(\varphi)\coloneq
  (\varphi\vert_{\Gamma}, \bn\cdot\nabla
  \varphi\vert_{\Gamma}).
\end{equation*}
We emphasize that traces are taken from the interior of $\Omega$ here.
Similarly, we define the \emph{Dirichlet-Neumann exterior trace map}
$\gamma_{\Gamma,c}\colon \mH^{1}_{\mrm{loc}}(\Delta,\RR^d\setminus{\Omega})
\to\mH^{1/2}(\Gamma)\times \mH^{-1/2}(\Gamma)$ as the unique bounded
linear map satisfying $ \gamma_{\Gamma,c}(\varphi)\coloneq
(\varphi\vert_{\Gamma}, \bn\cdot\nabla\varphi\vert_{\Gamma})$ for all
$\varphi\in \mathscr{C}^{\infty}(\RR^d\setminus\Omega)$ where, this
time, traces are taken from the exterior of $\Omega$ (with $\bn$ still
directed toward the exterior of $\Omega$).

With Dirichlet-Neumann trace maps, we can form
$\frac{1}{2}(\gamma_{\Gamma}+\gamma_{\Gamma,c})\cdot\mathcal{G}_{\Gamma}\colon
\mH^{1/2}(\Gamma)\times \mH^{-1/2}(\Gamma)\to \mH^{1/2}(\Gamma)\times
\mH^{-1/2}(\Gamma)$, which is commonly decomposed into a $2\times 2$ matrix
of \emph{boundary integral operators}
\begin{equation}\label{BIOP}
  \frac{1}{2}(\gamma_{\Gamma}+\gamma_{\Gamma,c})\cdot\mathcal{G}_{\Gamma} =
  \left\lbr\begin{array}{ll}
    \mathcal{K}_{\Gamma} & \mathcal{V}_{\Gamma}\\
    \mathcal{W}_{\Gamma} & \tilde{\mathcal{K}}_{\Gamma}
  \end{array}\right\rbr
\end{equation}
where the bounded operators $\mathcal{V}_{\Gamma}\colon \mH^{-1/2}(\Gamma)\to \mH^{+1/2}(\Gamma)$ (\emph{single-layer operator}),
$\mathcal{W}_{\Gamma}\colon \mH^{+1/2}(\Gamma)\to \mH^{-1/2}(\Gamma)$ (\emph{hypersingular operator}), 
$\mathcal{K}_{\Gamma}\colon \mH^{+1/2}(\Gamma)\to \mH^{+1/2}(\Gamma)$ (\emph{double layer operator}), and
$\tilde{\mathcal{K}}_{\Gamma}\colon \mH^{-1/2}(\Gamma)\to
\mH^{-1/2}(\Gamma)$ (\emph{adjoint double layer operator}) can simply
be defined so as to comply with~\eqref{BIOP}. Both single-layer and
hypersingular operators are symmetric, while \(\mathcal{K}_{\Gamma}^*
= - \Kprime\).

\medskip
FEM-BEM coupling can be achieved by certain choices of
$\mathcal{A}_{\Gamma}$ involving the operators in~\eqref{BIOP}. These
choices model the Dirichlet and Neumann transmission conditions on
$\Gamma$, the Helmholtz equation with wavenumber $k>0$ in
$\RR^d\setminus\Omega$, and Sommerfeld's
condition~\eqref{SommerfeldRadiationCondition}. Several FEM-BEM couplings exist, and here we shall concentrate on
three of them, which were reviewed in e.g.~\cite{MR3038165}.

\begin{example}[Symmetric Costabel coupling~\cite{Costabel1987SMC,MR1052136}]
  It is classically written as two equations: 
  find $u\in \mH^{1}(\Omega)$, $p\in \mH^{-1/2}(\Gamma)$ such that
  \[
    \begin{aligned}
      \langle \mathcal{A}_{\Omega}(u),v\rangle 
      + \int_{\Gamma} \mathcal{W}_{\Gamma}(u\vert_{\Gamma})\,v\vert_{\Gamma} \,\diff s
      + \int_{\Gamma} (\Id/2+\tilde{\mathcal{K}}_{\Gamma})(p)\,v\vert_{\Gamma} \,\diff s  
      & = \langle \ell_{\Omega},v\rangle \\
      \int_{\Gamma} (\Id/2-\mathcal{K}_{\Gamma})(u\vert_{\Gamma})\,q \,\diff s
      - \int_{\Gamma} \mathcal{V}_{\Gamma}(p)\,q \,\diff s & = 0
    \end{aligned}
  \]
  $\forall v\in \mH^{1}(\Omega)$, $\forall q\in \mH^{-1/2}(\Gamma)$,  
  with $\mathcal{A}_{\Omega}$ and $\ell_{\Omega}$ defined in~\eqref{HelmholtzOperator}. 
  After summing the two equations, it fits formulation~\eqref{IntialBVP2} by taking $\ell_{\Gamma}=0$ and 
  \begin{equation}\label{CostabelCoupling}
    \textbf{Costabel:} \qquad
    \mathcal{A}_{\Gamma} =
    \left\lbr\begin{array}{cc}
    \mathcal{W}_{\Gamma} & \Id/2+\tilde{\mathcal{K}}_{\Gamma}\\
    \Id/2-\mathcal{K}_{\Gamma} & -\mathcal{V}_{\Gamma}
    \end{array}\right\rbr.
  \end{equation}
  Note that \(\mathcal{A}_{\Gamma}^* = \mathcal{A}_{\Gamma}\), hence the name \emph{symmetric} for the Costabel coupling.
\end{example}

\begin{example}[Johnson-N\'ed\'elec coupling~\cite{JohnsonNedelec1980CBI,MR3089444,MR2831059}]
  It is classically written as two equations: 
  find $u\in \mH^{1}(\Omega)$, $p\in \mH^{-1/2}(\Gamma)$ such that
  \[
    \begin{aligned}
      \langle \mathcal{A}_{\Omega}(u),v\rangle 
      + \int_{\Gamma} p\,v\vert_{\Gamma} \, \diff s  & = \langle \ell_{\Omega},v\rangle \\
      \int_{\Gamma} (\Id/2-\mathcal{K}_{\Gamma})(u\vert_{\Gamma})\,q \,\diff s
      - \int_{\Gamma} \mathcal{V}_{\Gamma}(p)\,q \,\diff s & = 0
    \end{aligned}
  \]
  $\forall v\in \mH^{1}(\Omega)$, $\forall q\in \mH^{-1/2}(\Gamma)$,  
  with $\mathcal{A}_{\Omega}$ and $\ell_{\Omega}$ defined in~\eqref{HelmholtzOperator}. 
  After summing the two equations, it fits formulation~\eqref{IntialBVP2} by taking $\ell_{\Gamma}=0$ and 
  \begin{equation}\label{JohnsonNedelecCoupling}
    \textbf{Johnson-N\'ed\'elec:} \qquad 
    \mathcal{A}_{\Gamma} =
    \left\lbr\begin{array}{cc}
    0 & \Id\\
    \Id/2-\mathcal{K}_{\Gamma} & -\mathcal{V}_{\Gamma}
    \end{array}\right\rbr. 
  \end{equation}
\end{example}

\begin{example}[Bielak-MacCamy coupling~\cite{MR700668}] 
  It is classically written as two equations: 
  find $u\in \mH^{1}(\Omega)$, $p\in \mH^{-1/2}(\Gamma)$ such that
  \[
    \begin{aligned}
      \langle \mathcal{A}_{\Omega}(u),v\rangle 
      + \int_{\Gamma} (\Id/2+\tilde{\mathcal{K}}_{\Gamma})(p)\,v\vert_{\Gamma} \,\diff s  
      & = \langle \ell_{\Omega},v\rangle \\
      \int_{\Gamma} u\vert_{\Gamma}\,q \,\diff s
      - \int_{\Gamma} \mathcal{V}_{\Gamma}(p)\,q \,\diff s & = 0
    \end{aligned}
  \]
  $\forall v\in \mH^{1}(\Omega)$, $\forall q\in \mH^{-1/2}(\Gamma)$,  
  with $\mathcal{A}_{\Omega}$ and $\ell_{\Omega}$ defined in~\eqref{HelmholtzOperator}. 
  After summing the two equations, it fits formulation~\eqref{IntialBVP2} by taking $\ell_{\Gamma}=0$ and 
  \begin{equation}\label{BielakMacCamyCoupling}
    \textbf{Bielak-MacCamy:} \qquad
    \mathcal{A}_{\Gamma} =
    \left\lbr\begin{array}{cc}
    0   & \Id/2+\tilde{\mathcal{K}}_{\Gamma} \\
    \Id & -\mathcal{V}_{\Gamma}
    \end{array}\right\rbr.
  \end{equation}
\end{example}

It is natural to ask whether the FEM-BEM couplings above comply with the imaginary part sign assumption~\eqref{Assumption2}. As a matter of fact, only
the Costabel coupling fits this hypothesis.
\begin{prop}\label{prop:costabel_sign_assumption}\quad\\
  Operator~\eqref{CostabelCoupling} for the Costabel coupling 
  fulfills Assumption~\eqref{Assumption2}.
\end{prop}
\begin{proof}
For the sake of brevity, set $\lbr(u,p),(v,q)\rbr_{\Gamma}\coloneq \langle u,q\rangle-\langle
v,p\rangle$ and $\mathfrak{A}_{\Gamma}\coloneq
\frac{1}{2}(\gamma_{\Gamma}+\gamma_{\Gamma,c})\cdot\mathcal{G}_{\Gamma}$,
see~\eqref{BIOP}.
Then, with $\mathcal{A}_{\Gamma}$ defined by~\eqref{CostabelCoupling},
observe that $\langle \mathcal{A}_{\Gamma}(u,p),(v,q)\rangle =
-\lbr \mathfrak{A}_{\Gamma}(u,p),(v,q)\rbr_{\Gamma} +
(\langle u,q\rangle+\langle v,p\rangle)/2$. This implies 
$\Im\{\langle \mathcal{A}_{\Gamma}(u,p),(\overline{u},\overline{p})\rangle\} =
-\Im\{\lbr \mathfrak{A}_{\Gamma}(u,p),(\overline{u},\overline{p})\rbr_{\Gamma}\}$ because $\langle u, \overline{p} \rangle + \langle \overline{u}, p \rangle = \langle u, \overline{p} \rangle + \overline{\langle u, \overline{p} \rangle} = 2 \Re \{ \langle u, \overline{p} \rangle\}$. 
Finally, according to~\cite[Proposition A.2]{MR4779883} or~\cite[Lemma 3.2.1]{MR1822275},
$\Im\{\lbr \mathfrak{A}_{\Gamma}(u,p),(\overline{u},\overline{p})
\rbr_{\Gamma}\}\geq 0$ for all $(u,p)\in \mH^{1/2}(\Gamma)\times \mH^{-1/2}(\Gamma)$.
This concludes the proof. 
\end{proof}

\begin{prop}\label{prop:JN_sign_assumption}\quad\\
  Operators~\eqref{JohnsonNedelecCoupling} and~\eqref{BielakMacCamyCoupling} for the Johnson-N\'ed\'elec and
  Bielak-MacCamy couplings do \emph{not} fulfill
  Assumption~\eqref{Assumption2}.
\end{prop}
\begin{proof}
  We construct a pair \((\Phi, P) \in \Hdemi \times \mH^{-1/2}(\Gamma)\)
  such that \( \Im\{\langle\mathcal{A}_{\Gamma}(\Phi,P),
  (\overline{\Phi,P})\rangle\} > 0\) in the case of the 
  Johnson-Nédélec coupling. A similar construction can be made for the 
  Bielak-MacCamy coupling. For all \(\phi \in \Hdemi\), \(p \in
  \mH^{-1/2}(\Gamma)\)
  \begin{align*}
      \langle \mathcal{A}_{\Gamma}(\phi,p), \overline{(\phi,p)} \rangle 
      &= \langle p, \overline{\phi} \rangle +\langle (\Id/2 - \K) \phi - \V p, \overline{p} \rangle \\
      &= \langle p, \overline{\phi} \rangle + \langle \phi, \overline{p} \rangle -
      \langle (\Id/2 + \K) \phi, \overline{p} \rangle - \langle \V p, \overline{p} \rangle \\
      &= \langle p, \overline{\phi} \rangle + \overline{\langle \overline{\phi}, p \rangle} -
      \langle (\Id/2 + \K) \phi, \overline{p} \rangle - \langle \V p, \overline{p} \rangle \\
      &= 2\Re\{\langle p, \overline{\phi} \rangle\} - \langle (\Id/2 + \K) \phi, \overline{p} \rangle - \langle \V p, \overline{p} \rangle.
  \end{align*}
  According to~\cite[Lemma~3.1]{EnglederSteinbach2007MBI},
  \(\Im\{\langle \V p, \overline{p} \rangle\} \geq 0\). Our goal is
  then to find a pair $(\phi,p)$ such that \(-\Im \{\langle (\Id/2 + \K) \phi, \overline{p}
  \rangle\} > \Im\{\langle \V p, \overline{p} \rangle\} \).
  Consider any invertible real positive definite operator \(\T\colon\Hdemi
  \to \mH^{-1/2}(\Gamma)\).  According to the Fredholm alternative, \(\mrm{dim} \, \mrm{Ker}(\Id/2 + \K)<+\infty\), so $\mrm{Ker}(\Id/2 +
  \mathcal{K}_{\Gamma})\neq \Hdemi$ and there exists $\phi\in \Hdemi$ satisfying
  $\lVert (\Id/2 + \mathcal{K}_{\Gamma}) \phi \rVert_{\mT}>0$.
  Let $\alpha>0$ be some real parameter whose value will be
  chosen appropriately below, and set $p_{\alpha}\coloneqq \alpha p$
  with \( p\coloneqq\imath \T (\Id/2 + \mathcal{K}_{\Gamma}) \phi\). We obtain
  \begin{equation*}
      \langle\mathcal{A}_{\Gamma}(\phi,p_{\alpha}),( \overline{\phi, p_{\alpha}}) \rangle 
      = 2\Re\{\langle p_{\alpha}, \overline{\phi} \rangle\} +
      \imath\, \alpha \lVert (\Id/2 + \mathcal{K}_{\Gamma}) \phi \rVert_{\mT}^2 - \alpha^{2}\langle  \V p,\overline{p} \rangle.
  \end{equation*}
  Hence, \(\Im\{\langle\mathcal{A}_{\Gamma}(\phi,p_{\alpha}),
  (\overline{\phi, p_{\alpha}}) \rangle\} = \alpha \lVert (\Id/2 +
  \mathcal{K}_{\Gamma}) \phi \rVert_{\mT}^2 - \alpha^{2}\Im\{\langle
  \overline{p}, \V p \rangle\}\). If $\Im\{\langle \overline{p}, \V p
  \rangle\} = 0$, we can choose $\alpha = 1$. Otherwise, $\Im\{\langle
  \overline{p}, \V p \rangle\} > 0$ and we can choose
  $\alpha>0$ satisfying $\alpha < \lVert (\Id/2 +
  \mathcal{K}_{\Gamma}) \phi \rVert_{\mT}^2/\Im\{\langle \overline{p},
  \V p \rangle\}$. Then \((\Phi, P) \coloneqq (\phi, p_{\alpha})\)
  yields a counterexample of~Assumption~\eqref{Assumption2}.
\end{proof}

In the following, we are going to develop a discrete counterpart of the
theory in~\cite{MR4665035} that will cover in particular the case where
\(\mathcal{A}_{\Gamma}\) yields one of the previous FEM-BEM couplings, to
reformulate Problem~\eqref{IntialBVPb}. However, even
when~\eqref{IntialBVPb} is well-posed, the corresponding variational
formulation~\eqref{IntialBVP2} can be ill-posed at some wavenumbers, called
\emph{spurious resonances}, due to the introduction of boundary integral
operators, see e.g.~\cite{SchulzHiptmair2022}. Thus, several situations can
occur:
\begin{itemize}
  \item \eqref{IntialBVP2} is well-posed, then~\eqref{IntialBVP2} will be
    rewritten as a substructured form with a trivial kernel. It can be the
    case of the three
    couplings~\eqref{CostabelCoupling},~\eqref{JohnsonNedelecCoupling}
    and~\eqref{BielakMacCamyCoupling} if the wavenumber \(\tilde{k}\) does
    not match a spurious resonance.
  \item Assuming~\eqref{Assumption2} and using Moore-Penrose
    pseudo-inverses,~\eqref{IntialBVP2} will be rewritten as a
    substructured form with a trivial kernel, even when the operator
    attached to~\eqref{IntialBVP2} admits a non-trivial kernel because of a
    spurious resonance. In this situation, the substructured formulation
    takes the form of ``identity + contraction''. This is the case of the
    Costabel coupling~\eqref{CostabelCoupling}.
\end{itemize}
These two situations are not mutually exclusive: for instance, the Costabel coupling~\eqref{CostabelCoupling} falls into both when the wavenumber is not a spurious resonance. 

\begin{rem}\quad\\
  If the Costabel FEM-BEM coupling is applied, then for any value of
  $\tilde{k}$ we can reconstruct a solution of Problem~\eqref{IntialBVPb}
  from a solution of the variational form~\eqref{IntialBVP2}.  If the
  Johnson-Nédélec or Bielak-MacCamy FEM-BEM couplings are applied, then we
  can reconstruct a solution of Problem~\eqref{IntialBVPb} from a solution
  of the variational form~\eqref{IntialBVP2} only when \(\tilde{k}\) does
  not match a spurious resonance.
\end{rem}

\section{Discretization}
\label{sec:discr}

Now we introduce notations related to the discretization of problems
of the form~\eqref{IntialBVP2}. We shall assume that the polyhedral
domain $\Omega\subset \RR^d$ is covered by a $d$-dimensional regular
simplicial mesh $\mathcal{T}_h^{d}(\Omega)$, i.e.~a
collection of $d$-dimensional simplices:
\begin{equation*}
  \overline{\Omega} = \cup\{\;\overline{\tau},\; \tau\in
  \mathcal{T}_h^{d}(\Omega)\;\}.
\end{equation*}
Any polyhedral $d$-dimensional subdomain $\Theta\subset \Omega$ will
be said to \emph{conform} to this mesh if there exists a
sub-collection $\mathcal{T}_h^{d}(\Theta)\subset
\mathcal{T}_h^{d}(\Omega)$ such that $\overline{\Theta} = \cup_{\tau\in
  \mathcal{T}_h^{d}(\Theta)}\overline{\tau}$. Let us denote
$\mathcal{T}_h^{d-1}(\Omega)$ the collection of $(d-1)$-dimensional
simplices that are faces of elements of $\mathcal{T}_h^{d}(\Omega)$.
Any $(d-1)$-dimensional set $\Upsilon\subset \overline{\Omega}$ is
said to \emph{conform} to $\mathcal{T}_h^{d}(\Omega)$ if there exists a
subcollection of faces $\mathcal{T}_h^{d-1}(\Upsilon)\subset
\mathcal{T}_h^{d-1}(\Omega)$ such that $\overline{\Upsilon} =
\cup_{\tau\in \mathcal{T}_h^{d-1}(\Upsilon)}\overline{\tau}$. In
particular, the boundary $\partial\Theta$ of any conforming
$d$-dimensional subset $\Theta\subset \Omega$ is conforming. Next we
consider the space of $\mathbb{P}_k$-Lagrange finite element functions
\begin{equation*}
  \Vh(\Omega) \coloneq \{ v\in \mathscr{C}^{0}(\overline{\Omega}),\;
  v\vert_{\tau}\in \mathbb{P}_k(\tau),
  \;\forall\tau\in\mathcal{T}_h^d(\Omega)\}.
\end{equation*}
We have $\Vh(\Omega)\subset \mH^{1}(\Omega)$. For any
$d$-dimensional or $(d-1)$-dimensional conforming subset
$\Theta\subset \Omega$, we consider the corresponding
$\mathbb{P}_k$-Lagrange finite element space $\Vh(\Theta) \coloneq \{
v\vert_{\Theta},\; v\in \Vh(\Omega)\}$. We will also need to consider
the spaces
\begin{equation*}
  \begin{aligned}
    & \Qh(\Gamma)\subset \mH^{-1/2}(\Gamma),\;\; \mrm{dim}\, \mQ_h(\Gamma) <+\infty,\\
    & \Vh(\Omega,\Gamma)\coloneq \Vh(\Omega)\times\Qh(\Gamma).
  \end{aligned}
\end{equation*}
The subsequent analysis does not need any further assumption on
$\mQ_h(\Gamma)$. The finite dimensional space $\Vh(\Omega,\Gamma)$ yields a Galerkin approximation of
$\mH^{1}(\Omega)\times\mH^{-1/2}(\Gamma)$, which leads to the operator
$\mA_{\Omega\times \Gamma}\colon \Vh(\Omega,\Gamma)\to\Vh(\Omega,\Gamma)^*$ defined by 
\begin{equation*}
  \begin{aligned}
  \langle \mA_{\Omega\times \Gamma}(u,p),(v,q)\rangle \coloneq
  \langle \mathcal{A}_{\Omega}(u),v\rangle+\langle
  \mathcal{A}_{\Gamma}(u\vert_{\Gamma},p),
  (v\vert_{\Gamma},q)\rangle,&\\
  \forall (u,p), (v,q) \in \Vh(\Omega,\Gamma).&\\
  \end{aligned}
\end{equation*}
The operator $\mA_{\Omega\times\Gamma}$ is simply the matrix stemming
from the Galerkin discretization of~\eqref{IntialBVP2}. Similarly, we
consider $\ell_{\Omega\times\Gamma}\in \Vh(\Omega,\Gamma)^*$ defined
by $\langle \ell_{\Omega\times\Gamma},(v,q)\rangle \coloneq \langle
\ell_{\Omega},v\rangle + \langle
\ell_\Gamma,(v\vert_{\Gamma},q)\rangle$ for all $v\in \Vh(\Omega)$,
$q\in\Qh(\Gamma)$. Leaving~\eqref{IntialBVP2} apart, from now on we
shall focus on the discrete problem
\begin{equation}\label{IntialBVP3}
  \begin{aligned}
    & (u,p)\in \Vh(\Omega,\Gamma)\;\text{and}\\
    & \mA_{\Omega\times \Gamma}(u,p) = \ell_{\Omega\times\Gamma}.
  \end{aligned}
\end{equation}
One of our goals is to develop a discrete
counterpart of the theory presented in~\cite{MR4665035}, and the
starting point of this analysis will be~\eqref{IntialBVP3}.

\medskip
We do not discard the possibility of a non-trivial kernel
$\mrm{Ker}(\mA_{\Omega\times\Gamma})\neq \{0\}$. If the wavenumber
$\kappa(\bx)$ is real-valued, this may occur for example when
considering~\eqref{HelmholtzOperator}-\eqref{DirichletBC}-\eqref{IntialBVP3}
or~\eqref{HelmholtzOperator}-\eqref{NeumannBC}-\eqref{IntialBVP3}. This can
also occur with FEM-BEM coupling,
i.e.~\eqref{HelmholtzOperator}-\eqref{CostabelCoupling}-\eqref{IntialBVP3}
and~\eqref{HelmholtzOperator}-\eqref{JohnsonNedelecCoupling}-\eqref{IntialBVP3},
when the wavenumber matches a spurious resonance. See the discussion at the end
of Sections~\ref{sec:pbs} and~\ref{FEMBEMCoupling}.

\section{Domain decomposition}\label{sec:ddm}

We now subdivide the domain $\Omega$ into a finite
collection of non-overlapping polyhedra $\overline{\Omega} =
\overline{\Omega}_1\cup \dots\cup \overline{\Omega}_{J}$ with
$\Omega_j\cap \Omega_k = \emptyset$ for $j \neq k$, and assume that
each $\Omega_j$ conforms with the mesh
$\mathcal{T}_h(\Omega)$. 
In practice, the
generation of such a partition can easily be automated with the help
of a graph partitioner.
In the context of FEM-BEM coupling, already
the case $J=1$ appears relevant. 
We shall denote $\Gamma_j\coloneq\partial\Omega_j$
and $\Sigma \coloneq \Gamma \cup \Gamma_1\cup \dots\cup \Gamma_{J}$ the \emph{skeleton} of the subdomain partition. 
In line with the notation of Section~\ref{sec:discr}, we recall that $\Vh(\Gamma) \coloneq \{v\vert_{\Gamma}, \,v \in \Vh(\Omega)\}$, $\Vh(\Sigma) \coloneq \{v\vert_{\Sigma}, \,v \in \Vh(\Omega)\}$, $\Vh(\Gamma_j) \coloneq \{v\vert_{\Gamma_j}, \,v \in \Vh(\Omega)\}$, and $\Vh(\Omega_j) \coloneq \{v\vert_{\Omega_j}, \,v \in \Vh(\Omega)\}$. 
Next, we introduce function spaces
adapted to the multi-domain setting:
\begin{equation*}
  \begin{aligned}
    & \Vh^{\textsc{b}}(\Gamma)\coloneq \Vh(\Gamma)\times \Qh(\Gamma)\\
    & \mbV_h(\Omega) \coloneq \Vh^{\textsc{b}}(\Gamma)\times  \Vh(\Omega_1)\times\cdots\times \Vh(\Omega_J)\\
    & \mbV_h(\Sigma) \coloneq \Vh(\Gamma)\times  \Vh(\Gamma_1)\times\cdots\times \Vh(\Gamma_J).
  \end{aligned}
\end{equation*}
The space $\mbV_h(\Sigma)$ consists of independent tuples of Dirichlet traces on the boundary
of each subdomain, and we shall refer to it as \textit{Dirichlet multi-trace space}.
Then, we introduce trace operators, beginning with
$\mathcal{B}\colon \Vh(\Omega,\Gamma)\to \Vh(\Sigma)$ defined by $\mathcal{B}(v,q)\coloneq
v\vert_{\Sigma}$. 
We also consider a multi-domain boundary trace
operator $\mB\colon \mbV_h(\Omega)\to \mbV_h(\Sigma)$ that acts
subdomain-wise and is block-diagonal:
\begin{equation*}
  \begin{aligned}
    & \mB \coloneq \mrm{diag}(\mB_\Gamma,\mB_{\Omega_1},\dots,\mB_{\Omega_J})\\
    & \mB_\Gamma(v,q)\coloneq v\quad\text{and}\quad \mB_{\Omega_j}(u)\coloneq u\vert_{\Gamma_j}.
  \end{aligned}
\end{equation*}
In practice, the trace operators $\mathcal{B},\mB$ are made of Boolean matrices that
keep track of the correspondence between volume degrees of freedom and
boundary degrees of freedom. We also need restriction operators
$\mathcal{R}\colon \Vh(\Omega,\Gamma)\to \mbV_h(\Omega)$ and 
$\mR\colon \Vh(\Sigma)\to \mbV_h(\Sigma)$ defined by 
\begin{equation}\label{DomainRestrictionOperators}
  \begin{aligned}
    & \mathcal{R}(v,q) \coloneq ((v\vert_{\Gamma},q),v\vert_{\Omega_1},\dots, v\vert_{\Omega_J})
    & v\in\Vh(\Omega), q\in \Qh(\Gamma) \\
    & \mR(u)  \coloneq (u\vert_{\Gamma},u\vert_{\Gamma_1},\dots, u\vert_{\Gamma_J})
    & u\in \Vh(\Sigma).
  \end{aligned}
\end{equation}
Here again, the restriction operators $\mathcal{R},\mR$ reduce to
Boolean matrices encoding adjacency relation between subdomains or
boundaries. They play a pivotal role in the expression of
transmission conditions. Note the commutativity relation
$\mR\cdot\mathcal{B} = \mB\cdot\mathcal{R}$. Next we need to introduce
a notation for the range of the restriction operators
\begin{equation}\label{DefinitionSingleTraceSpace}
  \begin{aligned}
    & \mbX_h(\Omega)\coloneq \mathcal{R}(\Vh(\Omega,\Gamma)),\\
    & \mbX_h(\Sigma)\coloneq \mR(\Vh(\Sigma)).
  \end{aligned}
\end{equation}
By construction, $\mbX_h(\Omega)\subset\mbV_h(\Omega)$ and
$\mbX_h(\Sigma)\subset\mbV_h(\Sigma)$. We call $\mbX_h(\Sigma)$ the
\textit{Dirichlet single trace space}: it consists of tuples of Dirichlet
traces that match across subdomain interfaces. Because this space encodes
Dirichlet transmission conditions, it will play an important role in the
subsequent analysis.  Neumann transmission conditions are encoded by its
annihilator $\mbX_h(\Sigma)^{\circ} \subset \mbV_h(\Sigma)^*$.  The
following two properties hold.
\begin{lem}\label{ReformulationTransmission}\quad\\[-15pt]
  \begin{itemize}
  \item[i)]  $\mbX_h(\Omega) = \mB^{-1}(\mbX_h(\Sigma))$,\\[-20pt]
  \item[ii)] $\mbX_h(\Omega)^{\circ} =\mB^{*}(\mbX_h(\Sigma)^{\circ})$.
  \end{itemize}
\end{lem}
\noindent 
Here \(\mB^{-1}\) refers to a preimage, as \(\mB\) is not invertible. We do not provide the proof of this lemma as it follows exactly the
same path as the one of~\cite[Lemma 4.1]{MR4665035}. 

\medskip
Let us now discuss in detail how we impose transmission
conditions through interfaces. Our strategy requires a
scalar product on the trace space, induced by an operator $\mT$ called \emph{impedance} (or \emph{transmission}) \emph{operator}:
\begin{equation}\label{ChoiceImpedance}
  \begin{aligned}
    & \mT\colon\mbV_h(\Sigma)\to \mbV_h(\Sigma)^*\quad \text{such that}\\
    & \langle \mT(\bu),\overline{\bu}\rangle>0, \ \forall \bu\in \mbV_h(\Sigma)\setminus\{0\}
    \quad \text{and}\quad \mT^* = \mT = \overline{\mT}.
  \end{aligned}
\end{equation}
The above properties transfer to the inverse map
$\mT^{-1}\colon \mbV_h(\Sigma)^*\to \mbV_h(\Sigma)$ which itself induces a
scalar product on $\mbV_h(\Sigma)^*$. We shall abundantly use the
corresponding norms defined as follows:
\begin{equation*}
  \Vert \bv\Vert_{\mT}^{2}\coloneq \langle \mT(\bv),\overline{\bv}\rangle\quad \text{and}\quad
  \Vert \bp\Vert_{\mT^{-1}}^{2}\coloneq \langle \mT^{-1}(\bp),\overline{\bp}\rangle.
\end{equation*}
The following result, whose proof can be found in~\cite[Lemma 3.4]{MR4648528}, establishes a connection between the
orthogonal complement of the single trace space and its annihilator.
\begin{lem}\label{OrthogonalDecomposition}\quad\\
  For any impedance operator $\mT\colon \mbV_h(\Sigma)\to \mbV_h(\Sigma)^*$ we have
  the two orthogonal decompositions
  \begin{equation*}
    \begin{aligned}
     & \mbV_h(\Sigma)^{\phantom{*}} = \mbX_h(\Sigma)^{\phantom{\circ}}\oplus \mT^{-1}(\mbX_h(\Sigma)^{\circ}),\\
     & \mbV_h(\Sigma)^{*} = \mbX_h(\Sigma)^{\circ}\oplus \mT(\mbX_h(\Sigma)).
    \end{aligned}
  \end{equation*}
\end{lem} 
\noindent 
The space $\mathscr{X}_h(\Sigma) \coloneq \mbX_h(\Sigma)\times \mbX_h(\Sigma)^{\circ}$ is a subspace
of $\mathscr{H}_h(\Sigma) \coloneq \mbV_h(\Sigma)\times \mbV_h(\Sigma)^*$ that provides a
characterization of both Dirichlet and Neumann transmission conditions (note that they are satisfied by functions of the space $\Vh(\Omega,\Gamma)$ 
in the initial discrete problem~\eqref{IntialBVP3}). The next result follows from Lemma~\ref{OrthogonalDecomposition} and 
is a further characterization of this subspace, established e.g.~in~\cite[Lemma
3.5 and 3.7]{MR4648528}. It reformulates Dirichlet/Neumann transmission conditions as \emph{generalized impedance transmission conditions}.

\begin{lem}\label{CaracExchangeOp}\quad\\
  Define $\Pi\colon \mbV_h(\Sigma)^*\to \mbV_h(\Sigma)^*$ by $\Pi\coloneq 2\mT\mR
  (\mR^*\mT\mR)^{-1}\mR^*-\Id$. Then $\Pi$ is an involutive isometry, that is, 
  $\Pi^{2} = \Id$ and $\Vert \Pi(\bq)\Vert_{\mT^{-1}} = \Vert
  \bq\Vert_{\mT^{-1}}, \ \forall \bq\in \mbV_h(\Sigma)^*$. Moreover, for
  any pair $(\bv,\bp)\in \mathscr{H}_h(\Sigma) \coloneq \mbV_h(\Sigma)\times\mbV_h(\Sigma)^*$, we have
  \begin{equation}\label{CaracAvecEchangeOperator}
    (\bv,\bp)\in \mathscr{X}_h(\Sigma) \coloneq \mbX_h(\Sigma)\times\mbX_h(\Sigma)^{\circ}\quad
    \iff\quad -\bp+i\mT(\bv) = \Pi(\bp+i\mT(\bv)).
  \end{equation}
\end{lem}
\noindent 
This characterization of transmission conditions is non-trivial,
even in the case of a single subdomain $J = 1$, i.e., if the
computational domain $\Omega$ is not subdivided. We name $\Pi$ the
\textit{exchange operator}. The exact form taken by this operator
intrinsically depends on the choice of the impedance operator $\mT$. The
operator $\Pi$ a priori couples distant subdomains, although, for
certain choices of $\mT$, the exchange operator becomes local, that is, it couples only adjacent subdomain, see the
discussion in~\cite[Section 9]{MR4648528}.

\section{Substructured reformulation of the problem}\label{sec:problem_reformulation}

We wish to reformulate Problem~\eqref{IntialBVP3} considering the domain decomposition introduced in the previous section. Define the Galerkin discretization
$\mA_{\Gamma}\colon \Vh^{\textsc{b}}(\Gamma) \to \Vh^{\textsc{b}}(\Gamma)^*$
of the boundary operator $\mathcal{A}_{\Gamma}$, that is, $\langle \mA_{\Gamma}(\bu),\bv\rangle\coloneq
\langle \mathcal{A}_{\Gamma}(\bu),\bv\rangle$ for all $\bu,\bv\in
\Vh^{\textsc{b}}(\Gamma)$, and define the block-diagonal operator
$\mA\colon \mbV_h(\Omega)\to \mbV_h(\Omega)^*$ by
\begin{equation*}
  \begin{aligned}
    & \mA\coloneq \mrm{diag}(\mA_{\Gamma},\mA_{\Omega_1},\dots,\mA_{\Omega_J})\\
    & \langle \mA_{\Omega_j}(u),v\rangle\coloneq\int_{\Omega_j}(\nabla u\cdot\nabla v - \kappa^{2}u v)\,d\bx
    \quad \forall u,v\in \Vh(\Omega_j).
  \end{aligned}
\end{equation*}
Due to Assumption~\eqref{Assumption1} on the sign of the imaginary part of the problem under study,
if Assumption~\eqref{Assumption2} is also satisfied, then the block diagonal operator
$\mA$ itself satisfies a similar property
\begin{equation}\label{Assumption3}
  \eqref{Assumption1}, \eqref{Assumption2}\quad\Longrightarrow\quad \Im\{\langle
  \mA(\bv),\overline{\bv}\rangle\}\leq 0,\ \forall \bv\in
  \mbV_h(\Omega).
\end{equation}
The source term of~\eqref{IntialBVP3} can be decomposed similarly to \(\mA\). We
define $\ell_{\Omega_j}\in \Vh(\Omega_j)^*$ by $\langle
\ell_{\Omega_j}, v\rangle\coloneq\int_{\Omega_j}f v\,d\bx$ for all $v \in \Vh(\Omega_j)$, and then
$\bell\coloneq(\ell_{\Gamma},\ell_{\Omega_1},\dots,\ell_{\Omega_J})\in
\mbV_h(\Omega)^*$. 
By definition of $\mathcal{R}$ in~\eqref{DomainRestrictionOperators}, we have the factorizations 
\begin{equation}\label{DecompositionMainOperator}
  \mA_{\Omega\times\Gamma} = \mathcal{R}^{*}\mA\mathcal{R}
  \quad\text{and}\quad
  \ell_{\Omega\times\Gamma} = \mathcal{R}^{*}\boldsymbol{\ell}.
\end{equation}
The block-diagonal operator $\mA$ can then be used to decompose Problem~\eqref{IntialBVP3} subdomain-wise, and express the following first reformulation of Problem~\eqref{IntialBVP3}.

\begin{lem}\quad\\
  If $(u,p)\in\Vh(\Omega,\Gamma)$ satisfies Problem~\eqref{IntialBVP3}, then
  there exists $\bp\in \mbV_h(\Sigma)^*$ such that, for $\bu =
  \mathcal{R}(u,p)$, the following equations hold
  \begin{equation}\label{IntialBVP4}
    \begin{aligned}
      & \bu \in \mbV_h(\Omega), \bp\in \mbV_h(\Sigma)^*\\
      & \mA\bu - \mB^{*}\bp = \bell\\
      & -\bp+i\mT\mB\bu = \Pi(\bp+i\mT\mB\bu).
    \end{aligned}
  \end{equation}
  Reciprocally if~\eqref{IntialBVP4} holds for a pair
  $(\bu,\bp)\in \mbV_h(\Omega)\times \mbV_h(\Sigma)^{*}$, then
  there exists $(u,p)\in \Vh(\Omega,\Gamma)$ solution to Problem~\eqref{IntialBVP3} such that $\bu = \mathcal{R}(u,p)$.
\end{lem}

\begin{proof}
  Since by definition \(\mbX_h(\Omega) \coloneq \mathcal{R}(\Vh(\Omega,\Gamma))\), Equation~\eqref{IntialBVP3} can be rewritten \(\lbrack \bu \coloneq \mathcal{R}(u,p)\in \mbX_h(\Omega)\) and \(\langle \mA(\bu)-\bell,\bv\rangle = 0,\ \forall \bv\in\mbX_h(\Omega) \rbrack\), which is equivalent to \(\lbrack \bu\in \mbX_h(\Omega)\) and
  \(\mA(\bu)-\bell\in \mbX_h(\Omega)^{\circ} \rbrack\). Next we know from \textit{i)} of Lemma~\ref{ReformulationTransmission} that \(\lbrack \bu\in \mbX_h(\Omega)\rbrack\iff \lbrack \bu\in
  \mbV_h(\Omega) \;\text{and}\; \mB(\bu)\in \mbX_h(\Sigma) \rbrack\), and \textit{ii)} of Lemma~\ref{ReformulationTransmission} yields \(
  \lbrack \mA(\bu)-\bell\in \mbX_h(\Omega)^{\circ}\rbrack \iff\lbrack
  \exists\bp\in \mbX_h(\Sigma)^{\circ}\;\text{such that}\;\mA(\bu)-\bell = \mB^{*}\bp \rbrack
  \). Thanks to the characterization of \(\mathscr{X}_h(\Sigma) = \mbX_h(\Sigma)\times \mbX_h(\Sigma)^{\circ}\) provided by Lemma~\ref{CaracExchangeOp}, we finally obtain~\eqref{IntialBVP4}.
\end{proof}

We wish to further rearrange Problem~\eqref{IntialBVP4}, so as to
eliminate the volume unknown $\bu$, and reduce the problem to an
equation with a substructured unknown on the skeleton $\Sigma$ only. Following the theory in~\cite{MR4648528,MR4433119}, we rewrite the first equation of~\eqref{IntialBVP4} as $(\mA -i\mB^*\mT\mB)\bu
=\mB^{*}(\bp-i\mT\mB\bu) + \bell$. However, in the present case the
operator $\mA -i\mB^*\mT\mB$ may be singular, contrary to
the situation in~\cite{MR4648528,MR4433119}. 
Here are two fundamental 
properties that we will use to circumvent this state of affairs.

\begin{lem}\label{ElementaryPropScattering}\quad\\
  If either \(\mA-i\mB^*\mT\mB\) is invertible, or Assumption~\eqref{Assumption2} holds, then\\[-5pt]
  \begin{equation*}
    \hspace{-8.65cm}
    \begin{aligned}
     & \textit{i)}  & \mrm{Ker}(\mA-i\mB^*\mT\mB) & = \mrm{Ker}(\mA)\cap\mrm{Ker}(\mB)\\
     & \textit{ii)} & \mrm{Im}(\mA-i\mB^*\mT\mB)  & = \mrm{Im}(\mA)+\mrm{Im}(\mB^*)\\[5pt]
    \end{aligned}
  \end{equation*}
\end{lem}
\begin{proof}
  The result is clear if \(\mA-i\mB^*\mT\mB\) is invertible. Indeed, \(\mrm{Ker}(\mA)\cap\mrm{Ker}(\mB)\subset \mrm{Ker}(\mA-i\mB^*\mT\mB)=\{0\}\) implies \(\mrm{Ker}(\mA)\cap\mrm{Ker}(\mB)=\{0\}\), and \(\mbV_h(\Omega)^* = \mrm{Im}(\mA-i\mB^*\mT\mB) \subset \mrm{Im}(\mA)+\mrm{Im}(\mB^*)\) implies \(\mrm{Im}(\mA)+\mrm{Im}(\mB^*)=\mbV_h(\Omega)^*\).  
  
  Let us assume now that \(\mA-i\mB^*\mT\mB\) has a non-trivial kernel, but Assumption~\eqref{Assumption2} holds, which implies~\eqref{Assumption3}: \(\Im\{\langle \mA(\bv),\overline{\bv}\rangle\}\leq 0,\ \forall \bv\in \mbV_h(\Omega)\). 
On the one hand, $\mrm{Ker}(\mA)\cap\mrm{Ker}(\mB)\subset
\mrm{Ker}(\mA-i\mB^*\mT\mB)$. On the other hand, $\bu \in
\mrm{Ker}(\mA-i\mB^*\mT\mB)\Rightarrow 0 = \Im\{ \langle
\mA(\bu),\overline{\bu}\rangle - i\langle
\mB^*\mT\mB(\bu),\overline{\bu}\rangle\} = \Im\{ \langle
\mA(\bu),\overline{\bu}\rangle\} - \Vert \mB(\bu)\Vert_{\mT}^{2}$, so
that $0\leq \Vert \mB(\bu)\Vert_{\mT}^{2} = \Im\{ \langle
\mA(\bu),\overline{\bu}\rangle\}\leq 0\Rightarrow \mB(\bu) =
0\Rightarrow \bu\in \mrm{Ker}(\mB)$. Finally, $\bu\in
\mrm{Ker}(\mA-i\mB^*\mT\mB)\cap\mrm{Ker}(\mB) = \mrm{Ker}(\mA)\cap
\mrm{Ker}(\mB)$. This establishes \textit{i)}.

Since $\Im\{ \langle \mA^*(\overline{\bv}),{\bv}\rangle\} = \Im\{
\langle \mA(\bv),\overline{\bv}\rangle\} \leq 0,\ \forall \bv\in
\mbV_h(\Omega)$, we prove in the same manner that
$\mrm{Ker}(\mA^*-i\mB^*\mT\mB) =  \mrm{Ker}(\mA^*)\cap\mrm{Ker}(\mB)$. 
Hence, using~\eqref{ClosedRangeTheorem} and~\eqref{PropSumOfAnnihilator},
$\mrm{Im}(\mA)+\mrm{Im}(\mB^*) = \mrm{Ker}(\mA^*)^{\circ}+\mrm{Ker}(\mB)^{\circ} = ( \mrm{Ker}(\mA^*)\cap\mrm{Ker}(\mB))^{\circ} 
= \mrm{Ker}(\mA^*-i\mB^*\mT\mB)^{\circ} = \mrm{Im}(\mA-i\mB^*\mT\mB)$.
\end{proof}

\begin{rem}\label{RemarkCompatibility}\quad\\
  Under one of the hypotheses of Lemma~\ref{ElementaryPropScattering}, if~\eqref{IntialBVP4} holds, then
  $\bell\in \mrm{Im}(\mA-i\mB^*\mT\mB)$ necessarily.
\end{rem}

\begin{rem}\quad\\
  In the following, we will use the relations i) and ii) of Lemma~\ref{ElementaryPropScattering}. Thus, our discrete theory will cover both the case when \(\mA-i\mB^*\mT\mB\) has a non-trivial kernel, and when it is invertible. In the latter case, proofs could be simplified using the stronger result that \(\mrm{Ker}(\mA-i\mB^*\mT\mB)=\{0\}\) and \(\mrm{Im}(\mA-i\mB^*\mT\mB)\) is the whole discrete space.
\end{rem}

Like in~\cite{MR4665035,MR4648528, MR4433119}, the rearrangement of Problem~\eqref{IntialBVP4}
will rely on pairs of
Dirichlet-Neumann traces.
This leads us to introduce the multi-trace space
\begin{equation*}
  \begin{aligned}
    & \mathscr{H}_h(\Sigma)\coloneq \mbV_h(\Sigma)\times \mbV_h(\Sigma)^*\\ &
    \Vert (\bv,\bp)\Vert_{\mT\times\mT^{-1}}^{2} \coloneq \Vert
    \bv\Vert_{\mT}^{2} + \Vert \bp\Vert_{\mT^{-1}}^{2}. 
  \end{aligned}
\end{equation*}
Two subspaces of $\mathscr{H}_h(\Sigma)$ will play a pivotal
role. They are related to local problems and transmission conditions,
and defined as follows
\begin{equation}\label{RemarkableSubspaces}
  \begin{aligned}
    \mathscr{C}_h(\mA) &\coloneq \{(\mB(\bu),\bp)\in \mathscr{H}_h(\Sigma),\; \mA\bu = \mB^{*}\bp\},\\
    \mathscr{X}_h(\Sigma) &\coloneq \mbX_h(\Sigma)\times \mbX_h(\Sigma)^{\circ}, 
  \end{aligned}
\end{equation}
where \(\mathscr{C}_h(\mA)\) is often called the \emph{Cauchy data} set. The single-trace space $\mathscr{X}_h(\Sigma)$ was already introduced in Lemma~\ref{CaracExchangeOp} and characterized using the exchange operator $\Pi$. 
We have the following fundamental results involving the space $\mathscr{C}_h(\mA)$.

\begin{prop}\label{IsomorphismCauchyTraces}\quad\\
  Under one of the hypotheses of Lemma~\ref{ElementaryPropScattering}, the application
  $(\bv,\bp)\mapsto \bp-i\mT(\bv)$ isomorphically maps
  $\mathscr{C}_h(\mA)$ onto $\mbV_h(\Sigma)^*$, with the estimate
  $\Vert (\bv,\bp)\Vert_{\mT\times\mT^{-1}}^{2}\leq \Vert
  \bp-i\mT\bv\Vert_{\mT^{-1}}^{2} \leq 2\,\Vert
  (\bv,\bp)\Vert_{\mT\times\mT^{-1}}^{2}$ for all
  $(\bv,\bp)\in \mathscr{C}_h(\mA)$.
\end{prop}
\begin{proof}
Pick $\bq\in \mbV_h(\Sigma)^{*}$ arbitrary. According to
\textit{ii)} of Lemma~\ref{ElementaryPropScattering}, there exists
$\bu\in \mbV_h(\Omega)$ such that $(\mA-i\mB^*\mT \mB)\bu =
\mB^{*}\bq$, which rewrites $\mA\bu = \mB^{*}(\bq+i\mT\mB\bu)$.
As a consequence $(\bv,\bp) \coloneq(\mB\bu, \bq+i\mT\mB\bu)\in
\mathscr{C}_h(\mA)$, and $\bp - i\mT\bv = \bq$ by
construction, hence the surjectivity.

Next, for any $(\bv,\bp)\in \mathscr{C}_h(\mA)$, there exists
$\bu\in \mbV_h(\Omega)$ such that $\bv = \mB\bu$ and $\mB^*\bp =
\mA\bu$. Thus, $\langle \bp,\overline{\bv}\rangle =
\langle\bp,\mB(\overline{\bu})\rangle =
\langle\mB^*(\bp),\overline{\bu}\rangle =
\langle\mA(\bu),\overline{\bu}\rangle$. 
Then,~\eqref{Assumption3} yields $0\leq -2\Im\{\langle
\bp,\overline{\bv}\rangle\}\leq \Vert
(\bv,\bp)\Vert_{\mT\times\mT^{-1}}^{2}$, hence $0\leq \Vert
\bp-i\mT\bv\Vert_{\mT^{-1}}^{2} - \Vert
(\bv,\bp)\Vert_{\mT\times\mT^{-1}}^{2}\leq \Vert
(\bv,\bp)\Vert_{\mT\times\mT^{-1}}^{2}$, which establishes both
injectivity and the estimates. 
\end{proof}

\begin{prop}\label{CauchySetDecomposition}\quad\\
  Assume one of the hypotheses of Lemma~\ref{ElementaryPropScattering}. Define the graph space of \(\,i \mT\colon \mbV_h(\Sigma)\to \mbV_h(\Sigma)^* \): \(\mathscr{G}(i\mT)\coloneq \{(\bv,i\mT(\bv)),\bv \in \mbV_h(\Sigma)\}\). Then
  \begin{align*}
    \mathscr{H}_h(\Sigma) = \mathscr{C}_h(\mA) \oplus \mathscr{G}(i\mT).
  \end{align*}
\end{prop}

\begin{proof}
  Let \((\bv,\bp)\in \mathscr{C}_h(\mA) \cap \mathscr{G}(i\mT)\). There exists \(\bu \in \mbV_h(\Omega)\) such that \(\bv = \mB ( \bu)\), \(\mA \bu = \mB^* \bp\), and \(\bp = i \mT \bv\). Combining these relations with Lemma~\ref{ElementaryPropScattering}, we deduce that \(\bu \in \mrm{Ker}(\mA - i \mB^* \mT \mB)= \mrm{Ker}(\mA) \cap \mrm{Ker}(\mB)\), and then \((\bv,\bp)=(0,0)\).

  Now take an arbitrary \((\bv,\bp)\in\mathscr{H}_h(\Sigma)\). By surjectivity of \(\mB\), there exists \(\tilde{\bu}\in \mbV_h(\Omega)\) such that \(\mB(\tilde{\bu})=\bv\). Using \textit{ii)} of Lemma~\ref{ElementaryPropScattering}, we deduce that there exists \(\bw \in \mbV_h(\Omega)\) such that \((\mA-i\mB^*\mT\mB)\bw = \mA \tilde{\bu}-\mB^*\bp\). Set
  \begin{align*}
    &\bv_1 \coloneq \mB(\bw), && \bp_1 \coloneq i \mT \bv_1 = i \mT \mB(\bw),\\
    &\bv_2 \coloneq \mB(\tilde{\bu}-\bw), && \bp_2 \coloneq \bp - i \mT \mB(\bw).
  \end{align*}
  By construction, \((\bv_1,\bp_1)\in \mathscr{G}(i \mT)\). Besides, \(\mA ( \tilde{\bu}-\bw) = \mB^*(\bp - i\mT \mB \bw)= \mB^* \bp_2\), so that \((\bv_2,\bp_2)\in \mathscr{C}_h(\mA)\). To end the proof, note that \( \bv_1+\bv_2 = \mB(\tilde{\bu}) = \bv\) and \(\bp_1+\bp_2=\bp\) so that \(\mathscr{C}_h(\mA) + \mathscr{G}(i\mT)=\mathscr{H}_h(\Sigma)\).
\end{proof}

The space $\mathscr{C}_h(\mA)$ introduced above ``models'' the wave
equation local to each subdomain. In the next Proposition, we introduce the so-called
\emph{scattering operator}, which will play the role of a local solver
in our domain decomposition method. Recall the
definition of the Moore-Penrose pseudo-inverse mentioned in
Section~\ref{NotationConventions} and denoted with \(^{\dagger}\).

\begin{prop}\label{ProposistionContractivity}\quad\\
  Assume one of the hypotheses of Lemma~\ref{ElementaryPropScattering}. There exists a unique
  linear map $\mS\colon \mbV_h(\Sigma)^{*}\to \mbV_h(\Sigma)^{*}$, later
  referred to as \emph{scattering operator}, that satisfies
  \begin{equation}\label{ScatteringOperator}
    \bp+i\mT\bv = \mS(\bp-i\mT\bv),\quad \forall (\bv,\bp)\in
    \mathscr{C}_h(\mA).
  \end{equation}
  This operator is also given by the formula $\mS = \Id
  +2i\mT\mB(\mA-i\mB^*\mT\mB)^{\dagger}\mB^*$. 
  
  Besides, assuming~\eqref{Assumption2}, \(\mS\) is
  $\mT^{-1}$-contractive\textsuperscript{\ref{note:contraction}}, i.e.~for any $\bq\in \mbV_h(\Sigma)^{*}$, $\Vert \mS(\bq)\Vert_{\mT^{-1}} \le \Vert \bq \Vert_{\mT^{-1}}$. 
  More precisely,
  \begin{equation}\label{Contractivity}
    \begin{aligned}
      & \Vert \mS(\bq)\Vert_{\mT^{-1}}^{2} - 
      4\Im\{ \langle \mA(\bu),\overline{\bu}\rangle\} = \Vert \bq\Vert_{\mT^{-1}}^{2}\\
      & \text{where}\;\;\bu = (\mA-i\mB^*\mT\mB)^{\dagger}\mB^*\bq. 
    \end{aligned}
  \end{equation}
\end{prop}
\begin{proof}
Proposition~\ref{IsomorphismCauchyTraces} shows that Equation
\eqref{ScatteringOperator} unequivocally defines the operator
$\mS$. Next, pick $\bq\in\mbV_h(\Sigma)^{*}$ arbitrary. According to
\textit{ii)} of Lemma~\ref{ElementaryPropScattering}, there exists
$\tilde{\bu}\in \mbV_h(\Omega)$ such that
$(\mA-i\mB^*\mT\mB)\tilde{\bu} = \mB^{*}\bq$. Let us set $\bu\coloneq
(\mA-i\mB^*\mT\mB)^{\dagger}(\mA-i\mB^*\mT\mB)\tilde{\bu}$. Then,
according to the standard properties of Moore-Penrose pseudo-inverse
$\mM\cdot\mM^{\dagger}\cdot\mM = \mM$, with $\mM = \mA-i\mB^*\mT\mB$.
So, we conclude $(\mA-i\mB^*\mT\mB)\bu = (\mA-i\mB^*\mT\mB)\tilde{\bu} = \mB^*\bq$.

Now set $\bp = \bq + i\mT\mB\bu$. We have $\mA\bu - \mB^*\bp
= 0$ hence $(\mB\bu,\bp)\in\mathscr{C}_h(\mA)$. Since $\bq
= \bp - i\mT\mB\bu$, Equation~\eqref{ScatteringOperator} gives $\mS(\bq) = \bp  +i\mT\mB\bu = \bq + 2i\mT\mB\bu
= (\Id + 2i\mT\mB(\mA-i\mB^*\mT\mB)^{\dagger}\mB^*)\bq$ for all
$\bq\in \mbV_h(\Sigma)^*$, which confirms the alternative
expression of $\mS$ stated above. The contractivity identity~\eqref{Contractivity} is established as follows
\begin{equation*}
  \begin{aligned}
    \Vert \mS(\bq)\Vert_{\mT^{-1}}^{2} 
    & = \Vert \bq + 2i\mT\mB\bu \Vert_{\mT^{-1}}^{2}\\
    & = \Vert \bq\Vert_{\mT^{-1}}^{2} + 4\Vert \mB\bu\Vert_{\mT}^{2} +
    4\Im\{\langle \bq,\mB\overline{\bu}\rangle\}\\
    & = \Vert \bq\Vert_{\mT^{-1}}^{2} + 4\Vert \mB\bu\Vert_{\mT}^{2} +
    4\Im\{\langle \mA\bu,\overline{\bu}\rangle\} - 4\Im\{i\langle \mB^*\mT\mB\bu,\overline{\bu}\rangle \} \\
    & = \Vert \bq\Vert_{\mT^{-1}}^{2}+4\Im\{\langle \mA\bu,\overline{\bu}\rangle\}.
  \end{aligned}
\end{equation*}
\end{proof}

\noindent 
The scattering operator actually gives a characterization of the space of Cauchy data \(\mathscr{C}_h(\mA)\).

\begin{lem}\label{ScatteringOperatorChara}\quad\\
  Assume one of the hypotheses of Lemma~\ref{ElementaryPropScattering}. For any \((\bv,\bp)\in \mathscr{H}_h(\Sigma)\)
  \begin{align*}
    (\bv,\bp)\in \mathscr{C}_h(\mA) \Longleftrightarrow \bp+i\mT\bv = \mS(\bp-i\mT\bv). 
  \end{align*}
\end{lem}
\begin{proof}
  The proof relies on~\eqref{ScatteringOperator} and Proposition~\ref{CauchySetDecomposition}, and is similar to the one in~\cite[Lemma 7.3]{MR4665035}.
\end{proof}

We can now continue with the reformulation of Problem~\eqref{IntialBVP4}.

\begin{prop}\label{SkeletonFormulation1}\quad\\
  Assume one of the hypotheses of Lemma~\ref{ElementaryPropScattering}. If $(\bu,\bp)\in
  \mbV_h(\Omega)\times\mbV_h(\Sigma)^*$ satisfies~\eqref{IntialBVP4}
  for some $\bell \in\mrm{Im}(\mA)+\mrm{Im}(\mB^*)$, then the pair
  $(\bq_+,\bq_-) \coloneq (\bp+i\mT\mB\bu,\bp-i\mT\mB\bu )\in
  \mbV_h(\Sigma)^*\times \mbV_h(\Sigma)^*$ satisfies
  \begin{equation}\label{IntialBVP5}
    \begin{aligned}
      \bq_+-\mS(\bq_-) & = 2i\mT\mB(\mA-i\mB^*\mT\mB)^{\dagger}\bell\\
      \Pi(\bq_+)+\bq_- & = 0.
    \end{aligned}
  \end{equation}
  Conversely, if $(\bq_+,\bq_-)\in \mbV_h(\Sigma)^*\times
  \mbV_h(\Sigma)^*$ satisfies~\eqref{IntialBVP5} for some $\bell\in
  \mrm{Im}(\mA)+\mrm{Im}(\mB^*)$, then the pair $(\bu,\bp)\in
  \mbV_h(\Omega)\times\mbV_h(\Sigma)^{*}$, given by $\bu =
  (\mA-i\mB^*\mT\mB)^{\dagger}(\mB^*\bq_-+\bell)$ and $\bp =
  (\bq_++\bq_-)/2$, solves~\eqref{IntialBVP4}.
\end{prop}
\begin{proof}
If $(\bu,\bp)\in \mbV_h(\Omega)\times\mbV_h(\Sigma)^*$ satisfies~\eqref{IntialBVP4}, let us set $\bq_{\pm} = \bp\pm
i\mT\mB(\bu)$ and $\bfg = 2i\mT\mB(\mA-i\mB^*\mT\mB)^{\dagger}\bell$.
Then, the second equation of~\eqref{IntialBVP4} readily implies
$\Pi(\bq_+)+\bq_- = 0$. On the other hand, rearranging the terms
in the first equation of~\eqref{IntialBVP4} yields
$(\mA-i\mB^*\mT\mB)\bu = \mB^{*}\bq_- + \bell$, which can be
multiplied by $(\mA-i\mB^*\mT\mB)^{\dagger}$. Since $\mrm{Ker}(\mA-i\mB^*\mT\mB)\subset
\mrm{Ker}(\mB)$ according to Lemma~\ref{ElementaryPropScattering}, using (\textit{iii}) from~\eqref{MoorePenroseProperties} and~\eqref{MoorePenroseKernel} we
obtain $\mB(\bu)= \mB(\mA-i\mB^*\mT\mB)^{\dagger}(\mA-i\mB^*\mT\mB)\bu
= \mB(\mA-i\mB^*\mT\mB)^{\dagger}(\mB^{*}\bq_-+\bell)$. Then,
using the expression of the scattering operator provided by
Proposition~\ref{ProposistionContractivity}, this yields
$\bq_-+2i\mT\mB(\bu) = \mS(\bq_-)+\bfg$, which rewrites $\bq_+
= \mS(\bq_-)+\bfg$.

Now assume that $(\bq_+,\bq_-)\in \mbV_h(\Sigma)^*\times
\mbV_h(\Sigma)^*$ satisfies~\eqref{IntialBVP5}. Setting $\bu \coloneq
(\mA-i\mB^*\mT\mB)^{\dagger}(\mB^*\bq_-+\bell)$, the first line of~\eqref{IntialBVP5} indicates $\bq_+-\bq_- = 2i\mT\mB(\bu)$ and,
with $\bp \coloneq (\bq_++\bq_-)/2$, this yields $\bq_\pm =
\bp\pm i\mT\mB(\bu)$. Hence, from the second line of~\eqref{IntialBVP5} we deduce
$\Pi(\bp+i\mT\mB\bu)+\bp-i\mT\mB\bu = 0$. Next, since
$\mB^{*}\bq_-+\bell\in \mrm{Im}(\mA)+\mrm{Im}(\mB^*) =
\mrm{Im}(\mA-i\mB^*\mT\mB)$ by Lemma~\ref{ElementaryPropScattering}, using (\textit{iv}) from~\eqref{MoorePenroseProperties} and~\eqref{MoorePenroseRange} we obtain $(\mA-i\mB^*\mT\mB)\bu =
(\mA-i\mB^*\mT\mB)(\mA-i\mB^*\mT\mB)^{\dagger}(\mB^*\bq_-+\bell) =
\mB^*\bq_-+\bell$, which rewrites $\mA\bu =
\mB^*(\bq_-+i\mT\mB\bu)+\bell = \mB^*(\bp)+\bell$.
So $(\bu,\bp)$ solves Problem~\eqref{IntialBVP4}. 
\end{proof}

\begin{rem}\quad\\
  The assumption $\bell \in\mrm{Im}(\mA)+\mrm{Im}(\mB^*)$ should not
  be regarded as a limitation of the present analysis since Problem~\eqref{IntialBVP4} can only hold under such a condition.
\end{rem}

Proposition~\ref{SkeletonFormulation1} above has exhibited a
reformulation of~\eqref{IntialBVP4} as a $2\times 2$ system of
equations posed only on the skeleton $\Sigma$ of the subdomain
partition. Left multiplying the first equation of~\eqref{IntialBVP5}
by $\Pi$, and using the second equation to write $\bq_- = -\Pi(\bq_+)$,
we obtain the final substructured discrete formulation of our Helmholtz problem: 

\begin{cor}[Generalized Optimized Schwarz Method]\label{cor:skeleton_problem}\quad\\
  Assume one of the hypotheses of Lemma~\ref{ElementaryPropScattering}. If $(\bu,\bp)\in
  \mbV_h(\Omega)\times\mbV_h(\Sigma)^*$ satisfies~\eqref{IntialBVP4}
  for some $\bell \in\mrm{Im}(\mA)+\mrm{Im}(\mB^*)$, then $\bq =
  \bp-i\mT\mB\bu$ solves the problem
  \definecolor{grisclair}{gray}{0.9}
  \begin{empheq}[box={\setlength{\fboxsep}{10pt}\colorbox{grisclair}}]{equation}\label{IntialBVP6}
    \begin{aligned}
      & \bq\in \mbV_h(\Sigma)^*\quad \text{and}\\
      & (\Id+\Pi\mS)\bq = -2i\Pi\mT\mB(\mA-i\mB^*\mT\mB)^{\dagger}\bell
    \end{aligned}
  \end{empheq}
  Conversely, if $\bq$ satisfies~\eqref{IntialBVP6},
  then the pair $(\bu,\bp)\in
  \mbV_h(\Omega)\times\mbV_h(\Sigma)^{*}$ given by $\bu =
  (\mA-i\mB^*\mT\mB)^{\dagger}(\mB^*\bq+\bell)$ and $\bp =
  (\bq-\Pi(\bq))/2$ solves~\eqref{IntialBVP4}.
\end{cor}

\noindent 
Problem~\eqref{IntialBVP6} is the GOSM formulation, with an unknown $\bq$ on the skeleton $\Sigma$ of the domain partition.
Let us underline the structure of the operator $\Id+\Pi\mS$ appearing
in the left-hand side of~\eqref{IntialBVP6}. Since both $\Pi$
and $\mS$ are contractions\textsuperscript{\ref{note:contraction}} (see Lemma~\ref{CaracExchangeOp} and
Proposition~\ref{ProposistionContractivity}) when Assumption~\eqref{Assumption2} holds, the operator $\Pi\mS$ is
itself a contraction and the operator $\Id+\Pi\mS$ takes the form of 
``identity + contraction''. This guarantees the positivity for
this operator, a desirable property in the perspective of the
iterative solver defined in Algorithm~\ref{alg:iterative_procedure}.

\section{Analysis of resonances}
\label{sec:resonances}

As already pointed out at the end of Section~\ref{sec:discr}, the operator $\mA_{\Omega\times \Gamma}$ of Formulation~\eqref{IntialBVP3}, that
describes our boundary value problem, is not necessarily
invertible. This may occur when modelling physical resonance
phenomena, but may also occur in the case of spurious resonances
stemming from a FEM-BEM coupling, as discussed at the end of
Section~\ref{FEMBEMCoupling}. In the present section, we shall
investigate how this situation transfers to our domain decomposition skeleton formulation~\eqref{IntialBVP6}. Our analysis is based on an intermediary map \(\Phi\) that we study in detail.

The next result involves $\mB^{\dagger}\colon \mbV_h(\Sigma)\to \mbV_h(\Omega)$, which is called \emph{harmonic lifting map} in the literature on domain decomposition (see e.g.~\cite[\S 4.4]{MR2104179} or~\cite[\S 1.2.6]{MR3013465}). We remind that $\mB^{\dagger}$ is the Moore-Penrose pseudo-inverse of the (surjective) trace map
$\mB\colon \mbV_h(\Omega)\to \mbV_h(\Sigma)$, see Section~\ref{NotationConventions} and in particular~\eqref{CaracPseudoInverse}
for a characterization.

\begin{prop}\label{IsomorphismKernels1}\quad\\
  Assume one of the hypotheses of Lemma~\ref{ElementaryPropScattering}. Define the operator $\Phi\colon 
  \mrm{Ker}(\mA_{\Omega\times\Gamma})\to \mbV_h(\Sigma)^*$ by the formula
  $\Phi(\bu) \coloneq (\mB^{\dagger})^*\cdot(\mA-i\mB^*\mT\mB)\cdot\mathcal{R}(\bu)$.
  Then $\bu\in \mrm{Ker}(\Phi)\iff\mathcal{R}(\bu)\in \mrm{Ker}(\mA)\cap \mrm{Ker}(\mB)$, 
  and $\mrm{dim}\,\mrm{Ker}(\Phi) = \mrm{dim}(\mrm{Ker}(\mA)\cap \mrm{Ker}(\mB))$.
\end{prop}
\begin{proof}
Suppose first that $\bu = (u,p)\in \mrm{Ker}(\mA_{\Omega\times\Gamma})$ is arbitrary.
We have $\mathcal{R}(\bu)\in \mrm{Im}(\mathcal{R}) = \mbX_h(\Omega)$ by~\eqref{DefinitionSingleTraceSpace} and $\mA\mathcal{R}(\bu)\in 
\mrm{Ker}(\mathcal{R}^*) = \mbX_h(\Omega)^{\circ}$ by~\eqref{DecompositionMainOperator} and~\eqref{eq:polarOfImage}. According to \textit{ii)} of Lemma~\ref{ReformulationTransmission}, there exists $\bp\in
\mbX_h(\Sigma)^{\circ}$ such that $\mA\mathcal{R}(\bu) = \mB^{*}(\bp)$. 
Then, setting $\bv = \mB\mathcal{R}(\bu)\in \mbX_h(\Sigma) =
\mB(\mbX_h(\Omega))$ by \textit{i)} of Lemma~\ref{ReformulationTransmission}, $(\bv,\bp)\in
\mathscr{C}_h(\mA)\cap \mathscr{X}_h(\Sigma)$~---~see~\eqref{RemarkableSubspaces}. Besides, 
$(\mB^{\dagger})^*\mA\mathcal{R}(\bu) = (\mB^{\dagger})^*\mB^{*}(\bp)
= (\mB\cdot\mB^{\dagger})^{*}(\bp) = \bp$, since \(\mB\) is surjective. This leads to
\begin{equation}\label{IdentitePhiIntermediaire}
  \Phi(\bu) = \bp-i\mT(\bv).
\end{equation}
Since $(\bv,\bp)\in\mathscr{C}_h(\mA)$ and Proposition~\ref{IsomorphismCauchyTraces} holds, we derive $\bu\in \mrm{Ker}(\Phi)\iff \bp-i\mT(\bv) = 0\iff (\bv,\bp) =
0$. By the very definition of $\bv$ and
$\bp$ above, and since $\mB^*\colon \mbV_h(\Sigma)^*\to \mbV_h(\Omega)^*$ is
one-to-one, we have $ (\bv,\bp) = 0 \iff \lbrack\;\bv = 0 =
\mB\mathcal{R}(\bu)$ and $ \mB^{*}(\bp ) = 0 =
\mA\mathcal{R}(\bu)\;\rbrack \iff \mathcal{R}(\bu)\in
\mrm{Ker}(\mA)\cap \mrm{Ker}(\mB)$.

We finish the proof by noting that $\mathcal{R}\colon \Vh(\Omega,\Gamma)\to
\mbV_h(\Omega)$ is one-to-one according to its definition~\eqref{DomainRestrictionOperators}, hence $\mrm{dim}\,\mrm{Ker}(\Phi)
= \mrm{dim}\,\mathcal{R}(\,\mrm{Ker}(\Phi)\,) = \mrm{dim}(\mrm{Ker}(\mA)\cap
\mrm{Ker}(\mB))$. 
\end{proof}

We underline that, in the previous result, the map $\Phi$ is taken restricted to $\mrm{Ker}(\mA_{\Omega\times\Gamma})$. After having studied the kernel
of this map, we now study its range.

\begin{prop}\label{IsomorphismKernels2}\quad\\
  Assume one of the hypotheses of Lemma~\ref{ElementaryPropScattering}. 
  Let $\Phi$ be the operator defined in Proposition~\ref{IsomorphismKernels1}. 
  Then $\mrm{Im}(\Phi) = \Phi(\,\mrm{Ker}(\mA_{\Omega\times\Gamma})\,) = \mrm{Ker}(\Id+\Pi\mS)$.
\end{prop}
\begin{proof}
Recall the description of the map $\Phi$ from the beginning of the
proof of Proposition~\ref{IsomorphismKernels1} and~\eqref{IdentitePhiIntermediaire}. What precedes shows that, if $\bu\in
\mrm{Ker}(\mA_{\Omega\times\Gamma})$, then $\Phi(\bu) = \bp-i\mT(\bv)$
for some $(\bv,\bp)\in \mathscr{C}_h(\mA)\cap\mathscr{X}_h(\Sigma)$.
Applying the definition of the scattering operator $\mS$, see~\eqref{ScatteringOperator}, $\mS\cdot\Phi(\bu) =
\bp+i\mT(\bv)$. Then, applying the characterization of $\mathscr{X}_h(\Sigma)$ with the exchange operator
$\Pi$, see~\eqref{CaracAvecEchangeOperator}, we obtain
$\Pi\cdot\mS\cdot\Phi(\bu) = \Pi(\bp+i\mT(\bv)) = -\bp+i\mT(\bv) =
-\Phi(\bu)$. Hence, $(\Id + \Pi\mS)\Phi(\bu) = 0$ for all \(\bu\in
\mrm{Ker}(\mA_{\Omega\times\Gamma})\), which rewrites
\begin{equation*}
  \Phi(\,\mrm{Ker}(\mA_{\Omega\times\Gamma})\,)\subset \mrm{Ker}(\Id + \Pi\mS).
\end{equation*}

There only remains to show that
$\Phi(\,\mrm{Ker}(\mA_{\Omega\times\Gamma})\,) \supset \mrm{Ker}(\Id +
\Pi\mS)$. Pick an arbitrary $\bq_-\in \mrm{Ker}(\Id + \Pi\mS)$ and
set $\bq_+ = \mS(\bq_-)$. Since $(\Id + \Pi\mS)\bq_- = 0$ we
conclude that $\bq_-+\Pi(\bq_+) = 0$. Next define $(\bv,\bp)\in
\mbV_h(\Sigma)\times \mbV_h(\Sigma)^*$ by
\begin{equation*}
  \left\{
  \begin{array}{l}
    \bp-i\mT\bv = \bq_-\\
    \bp+i\mT\bv = \bq_+
  \end{array}
  \right.
  \quad\iff\quad
  \left\{
  \begin{array}{l}
    \bv = \mT^{-1}(\bq_+ -\bq_-)/(2i)\\
    \bp = (\bq_+ + \bq_-)/2
  \end{array}
  \right.
  .
\end{equation*}
According to the equations satisfied by $(\bq_-,\bq_+)$, we
have $\bp+i\mT\bv = \mS(\bp-i\mT\bv)$ and $-\bp+i\mT\bv =
\Pi(\bp+i\mT\bv)$, which rewrites
$(\bv,\bp)\in\mathscr{C}_h(\mA)\cap\mathscr{X}_h(\Sigma)$ due to~\eqref{CaracAvecEchangeOperator} and Lemma~\ref{ScatteringOperatorChara}.
Since $(\bv,\bp)\in\mathscr{C}_h(\mA)$, there exists $\bw\in
\mbV_h(\Omega)$ such that $\mB(\bw) = \bv$ and $\mA\bw =
\mB^*(\bp)$. On the other hand, as $\bv\in\mbX_h(\Sigma)$, we derive
$\bw\in \mB^{-1}(\mbX_h(\Sigma)) = \mbX_h(\Omega) =
\mathcal{R}(\Vh(\Omega,\Gamma))$, see \textit{i)} of Lemma~\ref{ReformulationTransmission} and~\eqref{DefinitionSingleTraceSpace}. So $\bw = \mathcal{R}(\bu)$ for
some $\bu\in \Vh(\Omega,\Gamma)$, and
$\mA_{\Omega\times\Gamma}(\bu) = \mathcal{R}^*\mA\mathcal{R}(\bu) =
\mathcal{R}^*\mA\bw = \mathcal{R}^*\mB^*\bp = \mathcal{B}^*R^*\bp = 0$ since $\bp\in
\mbX_h(\Sigma)^{\circ} = (\mR(\Vh(\Sigma)))^\circ = \mrm{Ker}\,\mR^*$ by~\eqref{eq:polarOfImage}. As a consequence, $\bu \in
\mrm{Ker}(\mA_{\Omega\times\Gamma})$ and $\Phi(\bu) = \bp-i\mT(\bv) =
\bq_-$. 
\end{proof}

\begin{thm}\label{EquivalenceDimensionResonance}\quad\\
  Assume one of the hypotheses of Lemma~\ref{ElementaryPropScattering}. Then 
  $\mrm{dim}\,\mrm{Ker}(\mA_{\Omega\times\Gamma})
  =\mrm{dim}(\mrm{Ker}(\mA)\cap\mrm{Ker}(\mB) )
  +\mrm{dim}\,\mrm{Ker}(\mrm{Id}+\Pi\mS)$.
\end{thm}
\begin{proof}
Using the rank-nullity theorem, $\mrm{dim}\,\mrm{Ker}(\mA_{\Omega\times\Gamma}) =
\mrm{dim}\,\mrm{Ker}(\Phi) + \mrm{dim}\,\mrm{Im}(\Phi)$. Then
$\mrm{dim}\,\mrm{Ker}(\Phi) =
\mrm{dim}(\mrm{Ker}(\mA)\cap\mrm{Ker}(\mB) )$ by Proposition~\ref{IsomorphismKernels1}, and $\mrm{dim}\,\mrm{Im}(\Phi) =
\mrm{dim}\,\mrm{Ker}(\mrm{Id}+\Pi\mS)$ by Proposition~\ref{IsomorphismKernels2}. This concludes the proof. 
\end{proof}

Note that Theorem~\ref{EquivalenceDimensionResonance} holds
no matter the specific choice of the impedance operator $\mT$, as soon as it induces a scalar
product, i.e.~it complies with~\eqref{ChoiceImpedance}. Theorem~\ref{EquivalenceDimensionResonance} adapts~\cite[Propositions 8.2 and 8.4]{MR4665035} to the discrete setting and, at
the same time, provides an extension since in the present case we do
not assume that $\mrm{Ker}(\mA)\cap \mrm{Ker}(\mB) = \{0\}$. The
case $\mrm{Ker}(\mA)\cap \mrm{Ker}(\mB)\neq \{0\}$ covers situations
where FEM-BEM couplings satisfying Assumption~\eqref{Assumption2} suffer from a spurious resonance phenomenon. 

An interesting consequence of the formula of Theorem~\ref{EquivalenceDimensionResonance} is that, if
$\mA_{\Omega\times\Gamma}$ is invertible, then, systematically, both
$\mA - i\mB^*\mT\mB$ (see Lemma~\ref{ElementaryPropScattering}) and
$\Id+\Pi\mS$ are invertible. On the other hand, if $\mA -
i\mB^*\mT\mB$ is invertible, then $\mrm{Ker}(\mA)\cap\mrm{Ker}(\mB)
= \{0\}$, see Lemma~\ref{ElementaryPropScattering}. This occurs for
example when imposing local boundary conditions such as Dirichlet,
Neumann or Robin conditions on $\Gamma$. In that case the map $\Phi$
establishes an isomorphism between
$\mrm{Ker}(\mA_{\Omega\times\Gamma})$ and
$\mrm{Ker}(\mrm{Id}+\Pi\mS)$.
Finally,~\cite{BoisneaultBonazzoliEtAl2025SRS} establishes that \(\mrm{Ker}(\mA_\Gamma -i\mB_\Gamma^*\mT_\Gamma\mB_\Gamma)\) and \(\mrm{Ker}(\mA_{\Omega\times\Gamma})\) are either simultaneously trivial or non-trivial for the Costabel coupling, i.e., the dimensions are equal. Thus, we deduce from Theorem~\ref{EquivalenceDimensionResonance} that the substructured problem~\eqref{IntialBVP6} is well-posed even at spurious resonances for the Costabel FEM-BEM coupling! 

\section{Inf-sup estimates}\label{sec:infsup_estimates}

We have just seen that, under the hypothesis
$\mrm{Ker}(\mA)\cap\mrm{Ker}(\mB) = \{0\}$, the operators
$\mA_{\Omega\times\Gamma}$ and $\mrm{Id}+\Pi\mS$ are simultaneously
both invertible or both non-invertible. Keeping the hypothesis
$\mrm{Ker}(\mA)\cap\mrm{Ker}(\mB) = \{0\}$, we wish to discuss how the
inf-sup constants of these two operators compare. For this we need to
introduce a few constants. First, we define two constants related
to the boundary trace operator
\begin{equation*}
  \begin{aligned}
    & t_{h}^{-}\coloneq \inf_{\bp\in\mbV_h(\Sigma)^*\setminus\{0\}} \Vert
    \mB^*(\bp)\Vert_{\mbV_h(\Omega)^{*}}/ \Vert
    \bp\Vert_{\mT^{-1}}\\
    & t_{h}^{+}\coloneq
    \sup_{\bp\in\mbV_h(\Sigma)^*\setminus\{0\}} \Vert
    \mB^*(\bp)\Vert_{\mbV_h(\Omega)^{*}}/ \Vert
    \bp\Vert_{\mT^{-1}}
  \end{aligned}
\end{equation*}
These are the extremal singular values of the operator $\mB^{*}$.
These two constants were thoroughly discussed in~\cite[Section
  7]{MR4648528} and~\cite{MR4433119}, where it was shown that an
appropriate choice of impedance operator $\mT$ leads to $t_h^{\pm}$
being $h$-uniformly bounded from below and above. We also need to
consider the continuity modulus
\begin{equation*}
  \Vert \mA\Vert\coloneq \sup_{\bu\in \mbV_h(\Omega)\setminus\{0\}}
  \sup_{\bv\in \mbV_h(\Omega)\setminus\{0\}} \frac{\vert\langle
    \bv,\mA(\bu)\rangle\vert}{\Vert
    \bv\Vert_{\mbV_h(\Omega)}\Vert
    \bu\Vert_{\mbV_h(\Omega)}}.
\end{equation*}
The next theorem compares the inf-sup constants (see definition~\eqref{DefInfSupCst}) of the
undecomposed problem~\eqref{IntialBVP3} and of the skeleton formulation~\eqref{IntialBVP6}. We do not provide the proof of this result: for
$\mrm{Ker}(\mA_{\Omega\times\Gamma})\neq \{0\}$, this inequality is a
direct consequence of Theorem~\ref{EquivalenceDimensionResonance}, and,
for $\mrm{Ker}(\mA_{\Omega\times\Gamma})= \{0\}$, the proof closely
parallels the one of~\cite[Propositions 8.1 and 8.4]{MR4648528}.

\begin{thm}\quad\\
  Assume $\mrm{Ker}(\mA)\cap\mrm{Ker}(\mB) = \{0\}$.
  Then the following estimate holds
  \begin{equation*}
    \infsup_{\Vh(\Omega,\Gamma)\to
      \Vh(\Omega,\Gamma)^*}(\mA_{\Omega\times\Gamma}) \leq (
    (t_h^+)^2+(2\Vert \mA\Vert/t_h^-)^2 )
    \infsup_{\mbV_h(\Sigma)^*\to
      \mbV_h(\Sigma)^*}(\Id+\Pi\mS).
  \end{equation*}
\end{thm}

Recall that both $\Pi$ and $\mS$ are contractive\textsuperscript{\ref{note:contraction}} in the norm $\Vert
\cdot\Vert_{\mT^{-1}}$ (see Lemma~\ref{CaracExchangeOp} and Proposition~\ref{ProposistionContractivity}), so following again the same
argumentation as in~\cite[Corollary 6.2]{MR4648528}, the previous
estimate combined with the contractivity of $\Pi\mS$ implies the
strong coercivity of the skeleton operator $\Id+\Pi\mS$.

\begin{cor}\label{CoercivityProperty}\quad\\
  Consider that $\mrm{Ker}(\mA)\cap\mrm{Ker}(\mB) = \{0\}$ and the imaginary sign assumption~\eqref{Assumption3} holds.
  Then we have the estimate
  \begin{equation*}
    \begin{aligned}
      \inf_{\bp\in \mbV_h(\Sigma)^*\setminus\{0\}}\frac{\Re \{\langle
        (\Id+\Pi\mS)\bp, \mT^{-1}\overline{\bp}\rangle\}}{\Vert
        \bp\Vert_{\mT^{-1}}^{2}}\geq \frac{1}{2}
      \Big(\;\frac{\alpha_h}{ (t_h^+)^2+(2\Vert
        \mA\Vert/t_h^-)^2}\;\Big)^{2} &\\[10pt]
      \text{with}\quad\alpha_h\coloneq \infsup_{\Vh(\Omega,\Gamma)\to
        \Vh(\Omega,\Gamma)^*}(\mA_{\Omega\times\Gamma}). &
    \end{aligned}
  \end{equation*}
\end{cor}

\begin{rem}\label{rem:hrobustness}\quad\\
  Assume that~\eqref{IntialBVP2} admits a unique solution and no
  physical resonance phenomenon is occurring. The impedance operator
  $\mT$ can be chosen appropriately so as to guarantee that the
  constants $t_h^{\pm}$ are $h$-uniformly bounded from below and above,
  see e.g.~\cite{MR4433119}. With such a choice of impedance, since
  the discretization can reasonably be assumed stable,
  i.e. $\liminf_{h\to 0}\alpha_h>0$, the inequality in Corollary~\ref{CoercivityProperty} provides $h$-uniform coercivity
  estimate. Such an estimate guarantees $h$-robust convergence of classical
  linear iterative solvers such as GMRes.
  
  On the other hand, the present analysis covers FEM-BEM
  coupling schemes such as~\eqref{CostabelCoupling} for which we are
  able to prove a convergence bound, see Theorem~\ref{thm:richardson_upper_bound}. To our knowledge, this yields the
  first $h$-robust iterative algorithm for FEM-BEM coupling with
  guaranteed convergence bounds.
\end{rem}

\section{Extension to more general problems}\label{sec:general_problems}

The theory developed so far can be generalized to boundary value problems that combine both Problems~\eqref{IntialBVPa} and~\eqref{IntialBVPb}. In addition to $\Omega$, let $\Omega_{\text{O}} \subset \RR^d$ be a polyhedral open set (not necessarily bounded nor connected) with bounded boundary and $\Omega_{\text{O}} \cap \Omega = \emptyset$. The wavenumber is constant in the BEM subdomain $\Omega_{\text{B}} = \RR^d \setminus (\overline{\Omega \cup \Omega_{\text{O}}})$ and can vary arbitrarily in the FEM subdomain $\Omega$, while $\Omega_{\text{O}}$ represents for instance an impenetrable obstacle (or a cavity) with a sound-soft or sound-hard boundary. Note that this configuration can include cross-points, that is, points where the three domains $\Omega$, $\Omega_{\text{O}}$ and $\Omega_{\text{B}}$ are adjacent, like in Figure~\ref{fig:cross_points-config}. Acoustic wave propagation is then modelled by the following general boundary value problem: 
\begin{equation}\label{eq:generalBVP}
  \begin{aligned}
    &u\in \mH^{1}_{\text{loc}}(\RR^d\setminus \overline{\Omega_{\text{O}}})\;\;\text{and}\;\; -\Delta u-\kappa^{2}u = f \text{ in }\RR^d\setminus \overline{\Omega_{\text{O}}}\\
    &+\;\text{boundary condition on } \partial\Omega_{\text{O}}\\
    &+\;\text{Sommerfeld's radiation condition if } \Omega_{\text{B}} \ne \emptyset.
  \end{aligned}
\end{equation}
Note that Equation~\eqref{eq:generalBVP} also covers both the configurations considered previously: it reduces to Problem~\eqref{IntialBVPa} when $\Omega_{\text{O}} = \RR^d \setminus \overline{\Omega}$, while it becomes Problem~\eqref{IntialBVPb} when $\Omega_{\text{O}} = \emptyset$. 

We wish to give a generalization of variational formulation~\eqref{IntialBVP2}, which we remind is the starting point of our theory. 
Let ${\Gamma}_{\text{B}} = \partial\Omega_{\text{B}}$ and${\Gamma}_{\text{O}} = \partial\Omega_{\text{O}}$, and define the space
\begin{equation*}
\begin{split}
  \mathscr{X}(\Omega, {\Gamma}_{\text{B}}, {\Gamma}_{\text{O}}) \coloneq 
  \{ & ( v, (\psi_{\text{B}}, q_{\text{B}}), (\psi_{\text{O}}, q_{\text{O}}) ) \in \\
  & \mH^{1}(\Omega) \times (\mH^{1/2}(\Gamma_{\text{B}})\times \mH^{-1/2}(\Gamma_{\text{B}})) \times (\mH^{1/2}(\Gamma_{\text{O}})\times \mH^{-1/2}(\Gamma_{\text{O}})) \mid \\
  &\exists\, \tilde{v} \in H^1_{\text{loc}}(\RR^d\setminus \overline{\Omega_{\text{O}}}) \text{ such that } \tilde{v} |_{\Omega}=v, \,\tilde{v} |_{\Gamma_{\text{B}}} = \psi_{\text{B}}, \,\tilde{v} |_{\Gamma_{\text{O}}, c} = \psi_{\text{O}} 
  \}.
\end{split}
\end{equation*}
As an analogue to~\eqref{IntialBVP2}, it can be shown (see~\cite{noteCRAS}) that Problem~\eqref{eq:generalBVP} can be reformulated as: 
\begin{equation}\label{eq:varfgeneralBVP}
  \begin{aligned}
    & \text{Find}\; ( u, (\phi_{\text{B}}, p_{\text{B}}), (\phi_{\text{O}}, p_{\text{O}}) ) \in \mathscr{X}(\Omega, {\Gamma}_{\text{B}}, {\Gamma}_{\text{O}}) \; \text{such that}\\
    & \langle \mathcal{A}_{\Omega}(u),v \rangle 
    + \langle \mathcal{A}_{\Gamma_{\text{B}}}(\phi_{\text{B}},p_{\text{B}}), (\psi_{\text{B}},q_{\text{B}}) \rangle
    + \langle \mathcal{A}_{\Gamma_{\text{O}}}(\phi_{\text{O}},p_{\text{O}}), (\psi_{\text{O}},q_{\text{O}}) \rangle \\
    & = \langle \ell_{\Omega}, v\rangle 
      + \langle \ell_{\Gamma_{\text{O}}}, (\psi_{\text{O}},q_{\text{O}}) \rangle, 
    \quad \forall \, ( v, (\psi_{\text{B}}, q_{\text{B}}), (\psi_{\text{O}}, q_{\text{O}}) ) \in \mathscr{X}(\Omega, {\Gamma}_{\text{B}}, {\Gamma}_{\text{O}}), 
  \end{aligned}
\end{equation}
where $\mathcal{A}_{\Gamma_{\text{B}}}$ represents a FEM-BEM coupling (see~\eqref{CostabelCoupling},~\eqref{JohnsonNedelecCoupling} or~\eqref{BielakMacCamyCoupling}), $\mathcal{A}_{\Omega}, \ell_{\Omega}$ are defined in~\eqref{HelmholtzOperator}, and $\mathcal{A}_{\Gamma_{\text{O}}}, \ell_{\Gamma_{\text{O}}}$ represent a boundary condition (see~\eqref{DirichletBC} or~\eqref{NeumannBC}, together with~\cite[Examples 3.1--3.4]{MR4665035}). 
Then, the whole theory for the domain decomposition method can be seamlessly extended by replacing, in the definition of \(\Sigma\), in the functional spaces and in the operators defined in the previous sections, the component on $\Gamma$ with two components on ${\Gamma}_{\text{B}}$ and ${\Gamma}_{\text{O}}$. 

\begin{rem}[Strongly and weakly imposed boundary conditions]\label{rem:weakstrongBC}\quad\\
When the boundary condition on \({\Gamma}_{\text{O}}\) is enforced thanks to an operator \(\mathcal{A}_{\Gamma_{\text{O}}}\) as above, we say that the boundary condition is \emph{weakly} imposed.
Alternatively, the boundary condition on \(\Gamma_{\text{O}}\) can be \emph{strongly} imposed as in classical variational formulations, by modifying the expressions of \(\mathcal{A}_{\Omega}\), \(\ell_{\Omega}\) in \eqref{HelmholtzOperator} and possibly the variational space \(\mH^{1}(\Omega)\). The GOSM theory holds also in the latter case. We refer to~\cite{MR4433119} for more details, in which the theory is established for Helmholtz equation in a bounded domain with strongly imposed impedance boundary condition.
\end{rem}

\section{Numerical experiments}\label{sec:numerical_experiments}

As pointed out in the last paragraph of
Section~\ref{sec:problem_reformulation}, if Assumption~\eqref{Assumption2}
is satisfied, the operator \(\Id+\Pi\mS\) from the substructured
problem~\eqref{IntialBVP6} takes the form ``identity + contraction''. We
choose to apply a Richardson iterative scheme (see
e.g.~\cite{Hackbusch2016ISL}), for which geometric convergence is guaranteed,
see Theorem~\ref{thm:richardson_upper_bound}.

In the numerical tests, Problem~\eqref{IntialBVP2}, and hence its discretization~\eqref{IntialBVP3}, are assumed well-posed, i.e., the wavenumber \(\kappa\) is chosen such that the solution is unique. Thus, as discussed at the end of Section~\ref{sec:resonances}, \(\mA-i\mB^*\mT\mB\) is invertible and problem~\eqref{IntialBVP6} is well-posed. Moreover, \(\mA\) and \(\mB\) being subdomain-wise block-diagonal matrices, \(\mT\) is also chosen to be subdomain-wise block-diagonal, \textit{i.e.}\ \(\mT:= \diag(\mT_{\Gamma},\mT_{\Omega_1},\dots,\mT_{\Omega_J})\). Note that if each block of \(\mT\) is real and defines a scalar product on the local trace space, then~\eqref{ChoiceImpedance} is satisfied.

We give more details on how to solve efficiently~\eqref{IntialBVP6} in Section~\ref{subsec:main_algo}, we discuss the choice of transmission operators \(\mT_j\) in \cref{sec:transmission_operator}, and present applications for various geometries and material configurations in \cref{ssec:fem-bem-strong} and \cref{ssec:fem-bem-weak}.

\subsection{Richardson solver and geometric convergence}\label{subsec:main_algo}

\paragraph{Notation} We emphasize that bold quantities, such as \(\bq\) in Problem~\eqref{IntialBVP6}, correspond to vectors of vectors, where each subvector contains the degrees of freedom (dofs) associated with the discretization of the corresponding subdomain. We denote the subvector associated with subdomains \(\Omega_j\) with a non-bold font and subscript \(j=1,\dots,J\), while \(j=0\) refers to \(\Gamma\). 

\paragraph{Considered unknowns} For implementation simplicity reasons\footnote{It is easier to consider discrete unknowns in \(\mbV_h(\Sigma)\), rather than in its dual \(\mbV_h(\Sigma)^*\).}, we solve~\eqref{IntialBVP6} after performing the change of variable \( \tilde{\bq} = \mT^{-1} \bq\). Hence, we solve a new problem with \(\tilde{\Pi} \coloneqq \mT^{-1} \Pi \mT\), \(\tilde{\mS} \coloneqq \mT^{-1} \mS \mT\) and \(\tilde{\bb} \coloneqq -2i\tilde{\Pi}\mB(\mA-i\mB^*\mT\mB)^{\dagger}\bell\), instead of \(\Pi\), \(\mS\) and \(\bb \coloneqq -2i\Pi \mT \mB(\mA-i\mB^*\mT\mB)^{\dagger}\bell\) respectively.
Note that \(\tilde{\Pi}\) and \(\tilde{\mS}\) have the same properties as \(\Pi\) and \(\mS\). In particular, \(\Id + \tilde{\Pi} \tilde{\mS}\) takes the form ``identity + contraction'', with respect to \(\Vert \cdot \Vert_{\mT}\).
Moreover, recalling the formulas \(\Pi = 2\mT\mR (\mR^*\mT\mR)^{-1}\mR^*-\Id\) and \(\mS = \Id
  +2i\mT\mB(\mA-i\mB^*\mT\mB)^{\dagger}\mB^*\), we see that, when computing a matrix-vector product with \(\tilde{\Pi}\) (resp.~\(\tilde{\mS}\)), the same numbers of products as with \(\Pi\) (resp.~\(\mS\)) are done. 
It is also of interest, even though the right-hand side is computed only once, that one less matrix-vector product is needed to compute \(\tilde{\bb}\) than to compute \(\bb\). Finally, remark that \(\tilde{\Pi} = \Pi^*\), but \(\tilde{\mS}\) might be different from \(\mS^*\) since we do not assume that \(\mA = \mA^*\) (which does not hold, for instance, with \(\mA_{\Gamma}\) defined as in~\eqref{JohnsonNedelecCoupling} or~\eqref{BielakMacCamyCoupling}).

\paragraph{Iterative scheme} To solve the modified equation resulting from \eqref{IntialBVP6}, that is, \((\Id+\tilde{\Pi}\tilde{\mS})\tilde{\bq} = \tilde{\bb}\), we use the Richardson iteration with relaxation parameter \(\beta \in (0,1)\): 
\begin{equation}\label{eq:richardson_iterate}
  \begin{aligned}
    \tilde{\bq}^{n+1} 
    &= \tilde{\bq}^n + \beta \left(\tilde{\bb} - (\Id + \tilde{\Pi} \tilde{\mS}) \tilde{\bq}^n \right)\\
    &= (1-\beta) \tilde{\bq}^n - \beta \, \tilde{\Pi} \tilde{\mS} \tilde{\bq}^n + \beta \, \tilde{\bb}. 
  \end{aligned}
\end{equation}
The detailed procedure for the resulting Generalized Optimized Schwarz Method (GOSM) is presented in Algorithm~\ref{alg:iterative_procedure}. 
The core of the Richardson iteration is at lines~\ref{line:core1} and~\ref{line:core2}. We also introduce two numerical parameters: \(N=30000\) the maximum number of iterations allowed, and \(\varepsilon = 10^{-6}\) the tolerance threshold for the relative residual
\(
  \lVert \tilde{\bb} - (\Id + \tilde{\Pi} \tilde{\mS}) \tilde{\bq}^n \rVert_2 / \lVert \tilde{\bb} - (\Id + \tilde{\Pi} \tilde{\mS}) \tilde{\bq}^0 \rVert_2
\), under which we consider that the algorithm has converged.

\begin{algorithm}[H]
  \caption{GOSM with Richardson iteration}\label{alg:iterative_procedure}
  \begin{algorithmic}[1]
    \State Choose an initial \(\tilde{\bq}\), \(n=0\)
    \For{j = \(0\) to \(J\)}
        \State Compute and store \(\left(\mA_j -i \mB_j^* \mT_j \mB_j\right)^{-1}\) \Comment{LU factorization}\label{line:LU}
        \State \(s_j = \tilde{q}_j + 2 i \mB_j \left(\mA_j -i \mB_j^* \mT_j \mB_j\right)^{-1} \mB_j^* \mT_j \tilde{q}_j\) \Comment{Local scattering}
    \EndFor
    \State \(\bg = \tilde{\Pi} \bs\) \Comment{Global exchange}\label{line:exchangeOp1}
    \For{j = \(0\) to \(J\)}
        \State Compute the residual \(r_j = \tilde{b}_j - (\tilde{q}_j + g_j)\)
        \State \(r_j^0 = r_j\)
    \EndFor
    \While{ \(\lVert \br \rVert_2 > \varepsilon \, \lVert \br^0 \rVert_2\) and \( n < N\)}  \Comment{Stopping criterion with \(2\)-norm}\label{line:stopping}
      \For{j = \(0\) to \(J\)}
        \State \(\tilde{q}_j = \tilde{q}_j + \beta \, r_j\)\label{line:core1}
        \State \(s_j = \tilde{q}_j + 2 i \mB_j \left(\mA_j -i \mB_j^* \mT_j \mB_j\right)^{-1} \mB_j^* \mT_j \tilde{q}_j\) \Comment{Local scattering}
      \EndFor
      \State \(\bg = \tilde{\Pi} \bs\) \Comment{Global exchange}\label{line:exchangeOp2} 

      \For{j = \(0\) to \(J\)}
        \State \(r_j = \tilde{b}_j - (\tilde{q}_j + g_j)\)\label{line:core2}
      \EndFor
      \State \(n=n+1\)
    \EndWhile
    \State \Return \(\tilde{\bq}\)
  \end{algorithmic}
\end{algorithm}

\paragraph{Richardson iteration properties}
The Richardson iteration~\eqref{eq:richardson_iterate} converges if and only if \(\rho((1-\beta)\Id - \beta \, \tilde{\Pi} \tilde{\mS}) < 1\), where \(\rho(\cdot)\) denotes the spectral radius (see e.g.~\cite[Section~3.5]{Hackbusch2016ISL}). The spectral radius being independent of a specific norm, the stopping criterion in line~\ref{line:stopping} of Algorithm~\ref{alg:iterative_procedure} is computed for the discrete \(2\)-norm. 
For our numerical experiments we choose \(\beta=1/2\). Since \(\tilde{\Pi} \tilde{\mS}\) is \(\mT\)-contractive, we derive from~\cite[Corollary~3.33]{Hackbusch2016ISL} \(\rho((1-\beta)\Id - \beta \, \tilde{\Pi} \tilde{\mS}) \leq \tau_h\), with \(\tau_h \coloneqq \sqrt{1- \gamma_h^2/4}\) and
\(
  \gamma_h \coloneqq \inf_{\tilde{\bq} \in \mbV_h(\Sigma)^*\setminus\{0\}} \lVert (\Id + \tilde{\Pi} \tilde{\mS}) \tilde{\bq} \rVert_{\mT}/ \lVert \tilde{\bq} \rVert_{\mT}
\). 
We point out that depending on which impedance operator \(\mT\) is chosen, \(\tau_h \) may, or may not, be independent of \(h\) (see also Remark~\ref{rem:hrobustness} and Section~\ref{sec:transmission_operator}). In particular, a \(h\)-independent upper bound for the convergence rate has been proved in~\cite{MR4433119}, when \(\mT\) is block diagonal and each of its block \(\mT_j\) is \emph{\(h\)-uniformly stable} in the following sense:

\begin{equation}\label{Assumption4}
  \begin{aligned}
    & \textbf{Assumption:} \qquad
    t_*^- \coloneqq \liminf_{h \to 0} t_h^- > 0,\qquad t_*^+ \coloneqq \limsup_{h \to 0} t_h^+ < +\infty.
  \end{aligned}
\end{equation}

\begin{thm}[{\cite[Corollary~10.7]{MR4433119}}]\label{thm:richardson_upper_bound}\quad\\
   Assume~\eqref{Assumption2},~\eqref{Assumption4} and \(\alpha_h > 0\).
  Let \(\tilde{\bfq}^{\infty} \in \mbV_h(\Sigma)\) refers to the unique solution of modified Equation~\eqref{IntialBVP6}.
  Then, \(\liminf_{h \to 0} \gamma_h > 0\), and for any \(0 < \gamma_* < \liminf_{h \to 0} \gamma_h\), there exists \(h_* > 0\) such that the iterate \(\tilde{\bfq}^n\) computed by means of~\eqref{eq:richardson_iterate} satisfies, with \(\tau_* \coloneqq 1 - \gamma_*^2/4\), the estimate
  \begin{align*}
    \dfrac{\lVert \tilde{\bfq}^n - \tilde{\bfq}^{\infty} \rVert}{\lVert \tilde{\bfq}^0 - \tilde{\bfq}^{\infty} \rVert} \leq \tau_*^{n/2} \quad \forall h \in (0, h_*), \forall n \geq 0.
  \end{align*}
\end{thm}

\paragraph{Global exchange}

The coupling between subdomains is performed by applying the exchange operator \(\tilde{\Pi}\), at lines~\ref{line:exchangeOp1} and~\ref{line:exchangeOp2} in Algorithm~\ref{alg:iterative_procedure}, that is, at initialization and then at each iteration. This process is described in detail in Algorithm~\ref{alg:exchange}, where \((\mR_j)_{0\le j \le J}\) denote the block components of the restriction operator \(\mR\) defined in~\eqref{DomainRestrictionOperators}. 
Its core is at line~\ref{line:coreExchange}, where subdomains share data between each other. This operation is the only one that might be global~---~see the last paragraph of Section~\ref{sec:ddm}. We highlight that the involved matrix \(\mR^* \mT \mR\) is symmetric and positive definite, see~\eqref{ChoiceImpedance}. Then, line~\ref{line:coreExchange} can be performed with the Conjugate Gradient (CG) method~---~we plan to apply, in a future work, efficient preconditioners available in the literature.
All the other operations of Algorithm~\ref{alg:iterative_procedure} are local to a given subdomain \(j\), so the GOSM can be implemented in parallel. 

\begin{algorithm}[H]
  \caption{Apply \(\tilde{\Pi}\) to a vector \(\bs\)}\label{alg:exchange}
  \begin{algorithmic}[1]
    \For{j = \(0\) to \(J\)}
      \State \(c_j = \mR_j^* \mT_j s_j \) \Comment{Local products}
    \EndFor
    \State \(\bu = \left(\mR^* \mT \mR\right)^{-1} \bc\) \Comment{Global product: Solve with CG}\label{line:coreExchange}
    \For{j = \(0\) to \(J\)}
      \State \(g_j = 2 \mR_j u_j - s_j\) \Comment{Local products}
    \EndFor
    \State \Return \(\bg\)
  \end{algorithmic}
\end{algorithm}

\paragraph{Discretization and software}
Meshes have been generated using \texttt{Gmsh}~\cite{GeuzaineRemacle2009G3F}, version \(4.13.1\). We have implemented the GOSM in a (sequential) library written in C++. 
We use $\mathbb{P}_1$-Lagrange finite and boundary elements. Finite element (resp.~boundary element) matrices are stored in sparse (resp.~dense) format. 
To compute the LU factorization of the blocks \(\mA_j -i \mB_j^* \mT_j \mB_j\), we use \texttt{UMFPACK}~\cite{Davis2004A8U} (resp.~\texttt{LAPACK}~\cite{Anderson1999LUG}) for sparse (resp.~dense) matrices. In future work, we plan to parallelize our GOSM code and to use a hierarchical format for the boundary element matrices, in order to solve large-scale problems.

\subsection{Choosing transmission operators}\label{sec:transmission_operator}

We now describe the three types of transmission operators used in our
numerical tests for a block of \(\mT\) associated with a subdomain \(\omega
\subset \RR^d\). The transmission conditions are applied on
\(\Gamma_{\omega} \coloneqq \partial \omega\).

\paragraph{Transmission operator 1: local Despr\'es operator}
The first one is the \emph{local} Després operator \(\mT_{D,\Gamma_{\omega}}\), given by
\(\langle\mT_{D,\Gamma_{\omega}} (u), v \rangle_{\Gamma_{\omega}} = \kappa \int_{\Gamma_{\omega}} u v \diff s\), of classical impedance
transmission conditions, so called because it was used
in~\cite{Despres1991DDM} for the first optimized Schwarz method for the
Helmholtz equation. It defines a scalar product on the local trace space,
and thus can be used to define an operator \(\mT\)
satisfying~\eqref{ChoiceImpedance} on \(\Sigma\). But it does not satisfy
Assumption~\eqref{Assumption4}, i.e., it is \emph{not} \(h\)-uniformly
stable (see for instance~\cite[Section~5]{MR4433119}).

\paragraph{Transmission operator 2: Yukawa hypersingular operator}
The second one, already used for instance
in~\cite[§3.1.2]{CollinoJolyEtAl2020ECN}, is \(\mT_{Y,\Gamma_{\omega}}
\coloneqq \mathcal{W}_{i \kappa, \Gamma_{\omega}}\), the hypersingular
operator, defined in~\eqref{BIOP}, this time taking the Green kernel of a
``positive Helmholtz'' operator, also known as the \emph{Yukawa operator},
that is, \(-\Delta + \kappa^2\,\Id\).  By definition
\(\mT_{Y,\Gamma_{\omega}}\) is \emph{non-local}, can be used to define a
transmission operator~\eqref{ChoiceImpedance} and satisfies
Assumption~\eqref{Assumption4}, see
again~\cite[Section~5]{MR4433119}. Since an explicit expression of the
Green kernel, available only for constant $\kappa$, is needed to implement
such operators, in practice \(\mT_{Y,\Gamma_{\omega}}\) can only be defined
for homogeneous subdomains.
\begin{figure}[!h]
  \centering
  \includegraphics[scale=1]{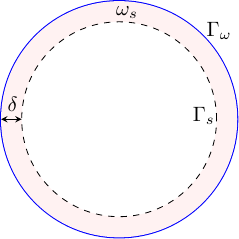} 
  \hspace*{1.0cm}
  \includegraphics[scale=1]{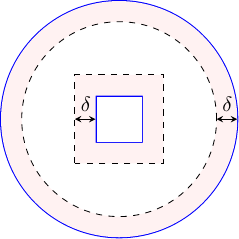} 
  \caption{Two examples of the layer \(\omega_s \subset \omega\) to compute the Schur complement based transmission operator \(\mT_{S}\) for a subdomain $\omega$. The layer $\omega_s$ is the shaded red region, and its boundary is the union of \(\Gamma_{\omega} \coloneq \partial \omega\) (solid blue line) and of  \(\Gamma_s \coloneqq \partial \omega_s \setminus \Gamma_{\omega}\) (dashed black line)}\label{fig:dtn_configuration}
\end{figure}
\paragraph{Transmission operator 3: Schur complement}
The last transmission operator \(\mT_{S,\omega}\) is a Schur complement
based operator, relying on the discretization of a positive
Dirichlet-to-Neumann (\(\DtN\)) map for the Yukawa operator. This
\emph{non-local} operator has been first introduced
in~\cite[Chapter~8]{Parolin2020NOD}, and then used
in~\cite{CollinoJolyEtAl2022NLI},~\cite[Section~4.2]{ClaeysCollinoEtAl2022NOS}
and~\cite[Section~6]{AtchekzaiClaeys2024ANL}. To derive it we consider
\(\omega_s \subset \omega\), a thin layer along \(\Gamma_{\omega}\), as
shown in Figure~\ref{fig:dtn_configuration}. The boundary of \(\omega_s\)
is the union of two disjoint parts, namely \(\Gamma_{\omega}\) and
\(\Gamma_s \coloneqq \partial \omega_s \setminus \Gamma_{\omega}\).  As
in~\cite[Section~6]{AtchekzaiClaeys2024ANL}, the map \(\mT_{S,\omega}:
\Vh(\Gamma_{\omega}) \rightarrow \Vh(\Gamma_{\omega})^* \) is then defined
as the unique Hermitian positive definite linear map satisfying the
minimization property
\begin{align*}
  \langle \mT_{S,\omega}(v), \overline{v} \rangle \coloneqq \min \lbrace
  \lVert \nabla v_s \rVert_{\mL^2(\omega_s)}^2 + \kappa^2 \lVert v_s
  \rVert_{\mL^2(\omega_s)}^2 + \kappa \lVert v_s \rVert_{\mL^2(\Gamma_s)}^2
  \mid v_s \in \mbV_h(\omega_s), v_s|_{\Gamma_{\omega}}=v\rbrace.
\end{align*}
The third term in the sum above expresses the choice of imposing
homogeneous Robin boundary conditions on \(\Gamma_s\), as
in~\cite[Section~8.3]{Parolin2020NOD}. The thickness \(\delta\) of the
layer \(\omega_s\) is a numerical parameter, which is chosen
\(h\)-independent in order for the operator to satisfy
Assumption~\eqref{Assumption4}
(see~\cite[Section~8.3.3]{Parolin2020NOD}). From now on we shall take
\(\delta = \lambda/10\) as a rule of thumb, i.e.~ten percent of the
wavelength.

In practice, to evaluate the action $\mT_{S,\omega}$, we work 
with the following characterization,
suitable for matrix-vector products: for \(g_D \in \Vh(\Gamma_{\omega})\)
\begin{equation}\label{eq:schur_def}
\begin{aligned}
  & \mT_{S,\omega}(g_D) = q  \; \text{ with } (v, q) \in \Vh(\omega_s) \times \Vh(\Gamma_{\omega})^* \text{ such that }\\
  & \begin{bmatrix}
    \mA_{\omega_s} & - (\mB^{'})^*\\
    \mB^{'} & 0
  \end{bmatrix}
  \begin{pmatrix}
    v \\
    q
  \end{pmatrix}
  = 
  \begin{pmatrix}
    0 \\
    g_D
  \end{pmatrix},
\end{aligned}
\end{equation}
where \(\mA_{\omega_s}\) is the matrix arising from the discretization of
the Yukawa operator in \(\omega_s\), and \(\mB^{'}\) is the restriction
matrix from \(\omega_s\) to \(\Gamma_{\omega}\). 

\begin{rem}[Focus on LU factorization when \(\mT_{S,\omega}\) is involved]\label{rq:lu_schur}\quad\\
  When \(\mT_{S,\omega}\) is chosen for the \(j\)-th subdomain, we may need to solve a local problem written as
  \begin{align*}
    (\mA_{j} -\imath\mB_{j}^* \mT_{S,\omega} \mB_{j})u_{j} = \mB_{j}^{*} \mT_{S,\omega} \tilde{\alpha}_j + l_{j}
  \end{align*}
  that arises from (modified) Equations~\eqref{Contractivity}
  or~\eqref{IntialBVP5}. Since \cref{eq:schur_def} is more suited to
  compute matrix-vector products than for deriving a matrix expression, we
  use it to build, as in~\cite[§4.2]{ClaeysCollinoEtAl2022NOS}, the
  enriched problem of the above equation:
  \begin{align*}
    \begin{bmatrix}
      \mA_{j} & 0 & \mB_{j}^* \\
      0 & -\imath \mA_{\omega_s} & - (\mB^{'})^* \\
      \mB_{j} & - \mB^{'} & 0
    \end{bmatrix}
    \begin{pmatrix}
      u_{j} \\
      v \\
      p
    \end{pmatrix}
    =
    \begin{pmatrix}
      l_{j} \\
      0 \\
      \imath \tilde{\alpha}_j
    \end{pmatrix},
  \end{align*}
  where \(p \coloneqq - \mT_{S,\omega} (\imath \mB_{j} u_{j} + \tilde{\alpha}_j)\) is an auxiliary variable.
  It is on this block matrix that is applied the LU factorization of line~\ref{line:LU} of Algorithm~\ref{alg:iterative_procedure}.
\end{rem}

\subsection{FEM-BEM coupling}\label{ssec:fem-bem-strong}
From now on, we only consider Helmholtz problems in dimension \(d=2\). In this section, the geometry of the numerical experiments consists in a bounded domain \(\Omega\) and an unbounded homogeneous domain \(\Omega_{\text{B}}\) whose bounded boundary, included in \(\partial\Omega\), is the surface denoted \(\Gamma\), see Figures~\ref{fig:homogen} and \ref{fig:heterogeneous_lens}. Note that, for this specific geometry with \(J=1\), \(\Gamma\) is also the skeleton \(\Sigma\) of the partition, and the (modified) exchange operator reads:
\begin{equation*}
  \tilde{\Pi} =
  \begin{bmatrix}
      \left(\mT_{\Gamma} + \mT_{\Omega}\right)^{-1} & 0 \\
      0 & \left(\mT_{\Gamma} + \mT_{\Omega}\right)^{-1}
  \end{bmatrix}
  \cdot
  \begin{bmatrix}
      \mT_{\Gamma} - \mT_{\Omega} & 2 \mT_{\Omega} \\
      2 \mT_{\Gamma} & \mT_{\Omega} - \mT_{\Gamma}
  \end{bmatrix}.
\end{equation*}


We choose \(\mA_{\Gamma}\) as one of the FEM-BEM couplings given in Section~\ref{FEMBEMCoupling}. Invertibility of \(\mA_{\Gamma} -i \mB_{\Gamma}^* \mT_{\Gamma} \mB_{\Gamma}\) is ensured because we have assumed that problem~\eqref{IntialBVP2} is well-posed;~so \(\kappa^2\) is not an eigenvalue of the Laplacian problem in \(\mathbb{R}^d \setminus \overline{\Omega_{\text{B}}}\) with homogeneous Dirichlet boundary conditions (see~\cite{BoisneaultBonazzoliEtAl2025SRS} for a proof and the expressions of the kernel of \(\mA_{\Gamma} -i \mB_{\Gamma}^* \mT_{\Gamma} \mB_{\Gamma}\) when \(\kappa\) is a spurious resonance).


\begin{defn}\label{dfn:transmission_choices}\quad\\
  The simulations are run for several configurations of transmission operators, namely:
  \begin{itemize}
    \item Configuration \DespresDespres, for which \(\mT_{\Omega} = \mT_{\Gamma} = \mT_{D,\Gamma}\),
    \item Configuration \SchurSchur, for which \(\mT_{\Omega} = \mT_{\Gamma} = \mT_{S,\Omega}\),
    \item Configuration \YukawaYukawa, for which \(\mT_{\Omega} = \mT_{\Gamma} = \mT_{Y,\Gamma}\),
    \item Configuration \SchurYukawa, for which \(\mT_{\Gamma} = \mT_{Y,\Gamma}\) and \(\mT_{\Omega} = \mT_{S,\Omega}\).
  \end{itemize}
  The configuration names refer to the subdomains, from the outside to the inside, that is, the first letter refers to the BEM subdomain, and the second one to the FEM subdomain.
  The thin layer needed to define \(\mT_{S,\Omega}\) is considered inside \(\Omega\). The operator \(\mT_{Y,\Gamma}\) is defined by taking \(\Omega_{\text{B}}\) as the \emph{interior} domain (see Section~\ref{FEMBEMCoupling}).
  Finally, note that the last three configurations only involve \emph{non-local} transmission operators.
\end{defn}

For the \DespresDespres, \SchurSchur{} and \YukawaYukawa{} configurations,
the same transmission operator is chosen for both domains.
Then, \(\tilde{\Pi}\) becomes the simple swap operator
$(p,q)\mapsto (q,p)$ and we recover the setup of classical OSMs.
For the \SchurSchur, \YukawaYukawa{} and \SchurYukawa{} configurations, all transmission operators satisfy the \(h\)-uniform stability Assumption~\eqref{Assumption4}. Then, according to \cref{thm:richardson_upper_bound}, the convergence rate \(\tau_h\) is \(h\)-independent.
This result does not hold for the \DespresDespres{} configuration, for which an estimate similar to the one of Theorem~\ref{thm:richardson_upper_bound} exists, but with \(\tau_h\) asymptotically behaving like \(1-\mathcal{O}(h^2)(h \to 0)\), as shown in~\cite[Remark~12.3]{MR4433119}. We emphasize that the geometric convergence arises for the \DespresDespres{} configuration because we study a discrete problem, whereas it is known that at the continuous level the convergence is not geometric, see~\cite{CollinoGhanemiEtAl2000DDM}.

\paragraph{Matrix properties}
Once discretized, classical FEM-BEM couplings~\eqref{IntialBVP2} for the Helmholtz equation lead to block matrices where some blocks are sparse and others are dense. Mixing sparse and dense blocks can be particularly difficult to deal with for numerical solver, the choice of a software to perform an efficient LU factorization being not natural.
In what follows, we highlight that depending on the choice of the transmission operators, the GOSM can also suffer from this issue, or not, when solving local problems.
\begin{itemize}
  \item[\SchurSchur{}] The local BEM problem \(\mA_{\Gamma} -i \mB_{\Gamma}^* \mT_{\Gamma} \mB_{\Gamma}\) mixes a sparse block and dense blocks.
  \item[\YukawaYukawa{}] The local FEM problem mixes a dense block from the transmission operator and sparse blocks.
  \item[\SchurYukawa{}] Transmission operators associated with each subdomain are non-local, and of the ``type of their local subproblem'', \textit{i.e.}~\(\mA_{j}\) and \(\mathcal\mT_{j}\) are both either sparse or dense. Thus, local problems do not mix dense and sparse blocks. Yet, transmission operators are distinct, meaning that even for a single interface the exchange operator \(\tilde{\Pi}\) is not just the swap operator, and mixes dense and sparse blocks. The difference with the previous configurations is that applying \(\tilde{\Pi}\) means solving a positive definite problem. So we can use a CG solver and only use matrix vector product of the \emph{block diagonal} operator \(\mT\), without mixing in practice dense and sparse blocks.
\end{itemize}

\begin{rem}\quad\\
  Avoiding mixing dense and sparse blocks also mean that in practice, one can use different optimized code for local FEM and BEM problems. For instance, one could use one-level or two-level preconditioners well-suited for FEM problems (see, e.g.,~\cite{LahayeTangEtAl2017MSH}), while using analytic preconditioners well-suited for BEM problems (see, e.g.,~\cite{AntoineBoubendir2008IPS}).
  This is one of the motivations in using ``weak''/substructured FEM-BEM couplings, see~\cite{Caudron2018CFB,CaudronAntoineEtAl2020OWC}.
\end{rem}

\subsubsection{Homogeneous problem}\label{subsubsec:homogenous_fem_bem}

For the first experiment, the whole space \(\mathbb{R}^2\) is assumed
homogeneous, \textit{i.e.} \(\kappa=k > 0\) constant, except for an
impenetrable disk \(\mathcal{D}\) of radius \(1\).  We consider the problem
of the scattering of an incoming plane wave \(u_{i}(r,\theta) = \exp(\imath
  \kappa \, r \cos(\theta))\), where \((r,\theta)\) are the polar
coordinates, by the obstacle \(\mathcal{D}\). Therefore, we solve the
Dirichlet problem
\begin{equation}\label{pbm:fem-bem_dirichlet}
  \begin{cases}
    \Delta u_D + \kappa^2 u_D = 0, 
    & \text{in } \mathbb{R}^2 \setminus \mathcal{D}, \\
    u_D(r,\theta) = - u_i(r, \theta), & \text{on \(\partial \mathcal{D}\)}, \\
    +  \text{ Sommerfeld's condition} &\text{\eqref{SommerfeldRadiationCondition}}, 
  \end{cases}
\end{equation}
whose unique solution is \(u_D(r, \theta) \coloneqq -\sum_{p \in
  \mathbb{Z}} \imath^{\lvert p \rvert}\exp(\imath p \theta)
J_{\lvert p \rvert}(\kappa) H^1_{\lvert p \rvert}(\kappa
r)/H^1_{\lvert p \rvert}(\kappa ) \), with \(J_{\nu}\) the Bessel function of first kind
of order \(\nu\) and \(H^1_{\nu}\) the Hankel function of first kind of
order \(\nu\), see e.g.~\cite[Chap.10]{MR2723248}.

In order for the problem to fit the framework of Problem~\eqref{IntialBVPb} and Equation~\eqref{IntialBVP2}, we introduce, as an artificial interface, a circle \(\Gamma\) of radius \(r=2\), and denote by \(\Omega\) the bounded domain whose boundary is composed of \(\Gamma\) and \(\partial \mathcal{D}\), see Figure~\ref{fig:fem_bem_homogeneous_configuration}.
Even though \(\Omega\) is homogeneous, we discretize it using finite elements.
In this experiment, boundary conditions on \(\partial \mathcal{D}\) are imposed strongly, that is, directly in the \(\Omega\) local problem, see Remark~\ref{rem:weakstrongBC}. In particular, for this geometry, and only for this geometry, \(\Gamma_{\Omega} = \Gamma \subsetneq \partial \Omega\), instead of \(\partial \Omega\) as explained in Section~\ref{sec:transmission_operator}. Thus, the transmission conditions are only set on \(\Sigma = \Gamma\), and so is the skeleton problem~\eqref{IntialBVP6}.


\begin{figure}[ht]
  \centering
  \begin{subfigure}{0.48\textwidth}
    \centering
    \includegraphics[scale=1]{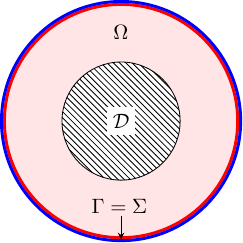} 
    \caption{Geometry considered for problem~\eqref{pbm:fem-bem_dirichlet}. \(\Sigma\)~is the skeleton}\label{fig:fem_bem_homogeneous_configuration}
  \end{subfigure}
  ~
  \begin{subfigure}{0.48\textwidth}
    \centering
    \includegraphics[scale=0.15]{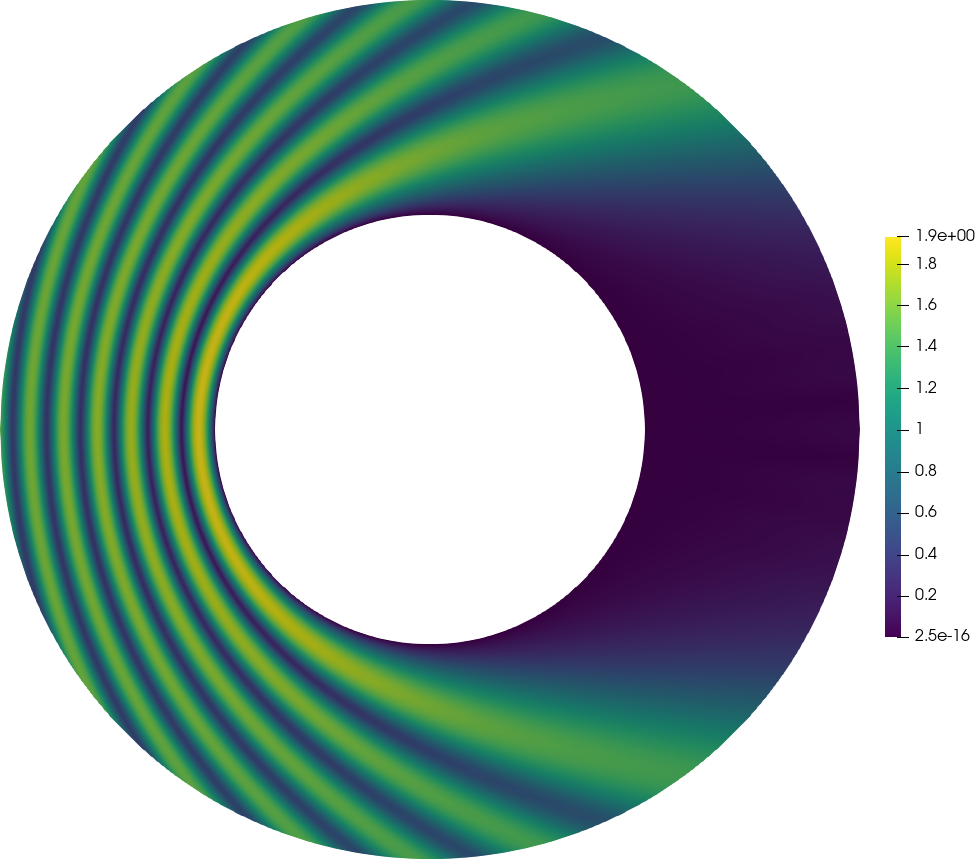} 
    \caption{Amplitude of the total field, for \SchurYukawa{} configuration, Costabel coupling and \(\kappa = 20\), \(N_{\lambda} = 40\)}\label{fig:strong_fem_bem_visual}
  \end{subfigure}
  \caption{Homogeneous test case, which has an explicit expression of the exact solution}\label{fig:homogen}
\end{figure}

\paragraph{Numerical results}
We denote by \(N_{\lambda} \coloneqq (2 \pi)/(\kappa h)\) the number of points per wavelength. 
Figure~\ref{fig:strong_fem_bem_visual} shows the modulus of the total field for \(\kappa = 20\) and \(N_{\lambda} = 40\).
We report in Tables~\ref{table:fem-bem-heterogeneous-Npwl_nb_it} and~\ref{table:fem-bem-heterogeneous-kappa_nb_it} the number of iterations \(N_{it}\) needed to reach convergence, respectively for varying \(N_{\lambda}\) and \(\kappa\).
For both studies, we observe that the number of iterations is roughly the same for Costabel, Johnson-Nédélec or Bielak-MacCamy couplings, even though \cref{thm:richardson_upper_bound} only guarantees convergence when the first one is used.
Thus, in what follows, we choose to focus on the Costabel coupling, but emphasize that all the numerical conclusions we draw are also valid for the two other FEM-BEM couplings.
The number of degrees of freedom of each simulation can be found in Table~\ref{table:dof_homogeneous}.

\begin{table}[t]
  \resizebox{\textwidth}{!}{
    \begin{tabular}{clcccc}
      \toprule
      \({N_{\lambda}}\) & \bfseries{FEM-BEM coupling} &  \bfseries{\DespresDespres} &  \bfseries{\YukawaYukawa} &  \bfseries{\SchurSchur} &  \bfseries{\SchurYukawa} \\ \midrule
      \multirow{3}{*}{30} & Costabel        & 5680 & 138 & 134 & 83 \\
                          & Johnson-Nédélec & 5857 & 137 & 132 & 83 \\ 
                          & Bielak-MacCamy  & 5858 & 137 & 131 & 85 \\ \midrule
      \multirow{3}{*}{40} & Costabel        & 9393 & 137 & 132 & 81 \\
                          & Johnson-Nédélec & 9917 & 136 & 130 & 86 \\ 
                          & Bielak-MacCamy  & 9991 & 136 & 127 & 82 \\ 
      \bottomrule
      \end{tabular}
  \begin{tabular}{clcccc}
    \toprule
    \({N_{\lambda}}\) & \bfseries{FEM-BEM coupling} &  \bfseries{\DespresDespres} &  \bfseries{\YukawaYukawa} &  \bfseries{\SchurSchur} &  \bfseries{\SchurYukawa} \\ \midrule
    \multirow{3}{*}{50} & Costabel        & 13830 & 134 & 130 & 80 \\
                        & Johnson-Nédélec & 14946 & 134 & 134 & 82 \\ 
                        & Bielak-MacCamy  & 14944 & 134 & 127 & 79 \\ \midrule
    \multirow{3}{*}{60} & Costabel        & 18957 & 135 & 128 & 78 \\
                        & Johnson-Nédélec & 20994 & 133 & 128 & 80 \\ 
                        & Bielak-MacCamy  & 21052 & 133 & 126 & 81 \\ 
    \bottomrule
    \end{tabular}
  }
  \caption{Several \(N_{it}\) used to draw~\figpart{fig:fem_bem_strong-Dirichlet}{left}, for \(\kappa = 10\) and \(h\kappa = 2\pi/N_{\lambda}\). See Definition~\ref{dfn:transmission_choices} for transmission operators choices}\label{table:fem-bem-heterogeneous-Npwl_nb_it}
\end{table}

\begin{table}[t]
  \resizebox{\textwidth}{!}{
    \begin{tabular}{@{}clcccc@{}}
      \toprule
      \({\kappa}\) & \bfseries{FEM-BEM coupling} & \bfseries{\DespresDespres} &  \bfseries{\YukawaYukawa} &  \bfseries{\SchurSchur} &  \bfseries{\SchurYukawa} \\ \midrule
      \multirow{3}{*}{6}  & Costabel        & 4343 & 146 & 79 & 69 \\
                          & Johnson-Nédélec & 4389 & 147 & 79 & 68 \\ 
                          & Bielak-MacCamy  & 4344 & 146 & 79 & 68 \\ \midrule
      \multirow{3}{*}{12} & Costabel        & 7558 & 141 & 104 & 72 \\
                          & Johnson-Nédélec & 7959 & 140 & 102 & 71 \\ 
                          & Bielak-MacCamy  & 7937 & 140 & 102 & 68 \\ 
      \bottomrule
      \end{tabular}
  \begin{tabular}{@{}clcccc@{}}
    \toprule
    \({\kappa}\) & \bfseries{FEM-BEM coupling} &  \bfseries{\DespresDespres} &  \bfseries{\YukawaYukawa} &  \bfseries{\SchurSchur} &  \bfseries{\SchurYukawa} \\ \midrule
    \multirow{3}{*}{18} & Costabel        & 10567 & 135 & 121 & 78 \\
                        & Johnson-Nédélec & 11071 & 135 & 124 & 79 \\ 
                        & Bielak-MacCamy  & 11152 & 134 & 125 & 78 \\ \midrule
    \multirow{3}{*}{24} & Costabel        & 13221 & 131 & 140 & 85 \\
                        & Johnson-Nédélec & 14107 & 131 & 144 & 89 \\ 
                        & Bielak-MacCamy  & 14262 & 132 & 140 & 87 \\ 
    \bottomrule
    \end{tabular}
  }
  \caption{Several \(N_{it}\) used to draw~\figpart{fig:fem_bem_strong-Dirichlet}{right}, for \(h^2 \kappa^3 = (2 \pi/10)^2\). See Definition~\ref{dfn:transmission_choices} for transmission operators choices}\label{table:fem-bem-heterogeneous-kappa_nb_it}
\end{table}

\begin{rem}\quad\\
  Recovering almost the same number of iterations for the three considered
  FEM-BEM couplings is not a surprise. Indeed, computing explicitly, at the
  continuous level, the operator \(\mS_{\Gamma}\) for all three couplings
  leads to the common expression \(\mS_{\Gamma} =
  (\Id/2+\tilde{\mathcal{K}}_{\Gamma}+ \imath
  \mT_{\Gamma}\mathcal{V}_{\Gamma})^{-1}
  (\Id/2+\tilde{\mathcal{K}}_{\Gamma} - \imath
  \mT_{\Gamma}\mathcal{V}_{\Gamma})\). This is nothing else than the
  outgoing-ingoing impedance-to-impedance operator of the local BEM
  problem. The proof relies on Calderón identities~\cite[Proposition
    3.6.4.]{MR2743235}. Although, in a discretized
  setting, scattering operators do not coincide anymore (because Calderón identities no more hold),
  this remark explains the similarities in
  Tables~\ref{table:fem-bem-heterogeneous-Npwl_nb_it}-\ref{table:fem-bem-heterogeneous-kappa_nb_it}.
\end{rem}

We begin by studying the \(h\)-dependence of each configuration of Definition~\ref{dfn:transmission_choices} with \(\kappa = 20\) and increasing the value of \(N_{\lambda}\), which consists in refining the mesh.

The evolution of the number of iterations is presented in \figpart{fig:fem_bem_strong-Dirichlet}{left}.
As expected from Theorem~\ref{thm:richardson_upper_bound}, the number of iteration does not grow for the configurations involving non-local operators, that is to say \SchurSchur, \YukawaYukawa{} and \SchurYukawa.
The curve associated with the local \DespresDespres{} configuration evolves asymptotically as \(\mathcal{O}(h^{1.69})\), which is quasi-quadratic, as expected by the discussion led after Definition~\ref{dfn:transmission_choices}.

We now turn to \figpart{fig:fem_bem_strong-Dirichlet}{right} that reports how the number of iterations evolves when \(\kappa\) grows. To avoid numerical pollution~---~see~\cite{IhlenburgFEA,IhlenburgBabuska1995FES,BabuskaSauter1997IPE,DeraemaekerBabuskaEtAl1999MDL}~---, we take \(N_{\lambda}=10\) at \(\kappa = 1\), and then refine the mesh with \(h^2 \kappa^3 = (2\pi/10)^2\). We point out that the operator \(\mT_{S,\Omega}\) involved in the \SchurSchur{} and \SchurYukawa{} configurations is different for each simulation, since the layer thickness \(\delta\) is inversely proportional to \(\kappa\). The relative errors for each problem are around one percent. 
It appears that \(N_{it}\) grows sublinearly for \DespresDespres{} and \SchurSchur{} configurations, while it is quasi-constant for the \YukawaYukawa{} one. For the \SchurYukawa{} configuration, it seems that the behavior is a mix of those of \SchurSchur{} and \YukawaYukawa. Indeed, while \(\kappa \lesssim 13\) the number of iterations is quasi-constant, and once this value is reached, \(N_{it}\) grows sublinearly.

\begin{figure}[!ht]
  \includegraphics[scale=1]{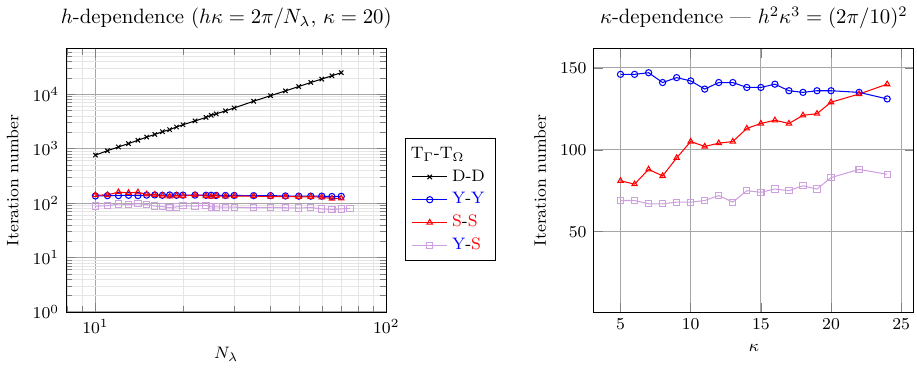}
  \caption{Homogeneous problem with strongly imposed Dirichlet boundary conditions and the Costabel coupling. Left, \(N_{it}\) with respect to \(N_{\lambda}\) for \(\kappa=20\) and \(h\kappa = 2\pi/N_{\lambda}\). Right, \(N_{it}\) with respect to the wavenumber \(\kappa\), for \(h^2\kappa^3 = (2\pi/10)^2\)}\label{fig:fem_bem_strong-Dirichlet}
\end{figure}


Finally, we have also led both experiences when the Dirichlet (sound-soft) boundary condition in Equation~\eqref{pbm:fem-bem_dirichlet} is replaced by the Neumann (sound-hard) boundary condition. Then, we solve
\begin{equation}\label{pbm:fem-bem_neumann}
  \begin{cases}
    \Delta u_N + \kappa^2 u_N = 0, 
    & \text{in } \mathbb{R}^2 \setminus \mathcal{D}, \\
     \gamma_N(u_N)(r,\theta) = - \gamma_N(u_i)(r,\theta), & \text{on \(\partial \mathcal{D}\)}, \\
    +  \text{ Sommerfeld's condition} &\text{\eqref{SommerfeldRadiationCondition}}, 
  \end{cases}
\end{equation}
whose unique solution is
\begin{equation*}
u_N(r, \theta) \coloneqq -\sum_{p \in \mathbb{Z}} \imath^{\lvert p \rvert}
\frac{\lvert p \rvert J_{\lvert p \rvert}(\kappa) - \kappa \, J_{\lvert p \rvert+1}(\kappa)}{\lvert p \rvert H^1_{\lvert p \rvert}(\kappa) - \kappa \, H^1_{\lvert p \rvert+1}(\kappa)} H^1_{\lvert p \rvert}(\kappa r)\exp(i p \theta).
\end{equation*}
The conclusions are the same as the ones drawn with Dirichlet boundary conditions, see~\cref{fig:fem_bem_strong-Neumann}.



\begin{figure}[!ht]
  \includegraphics[scale=1]{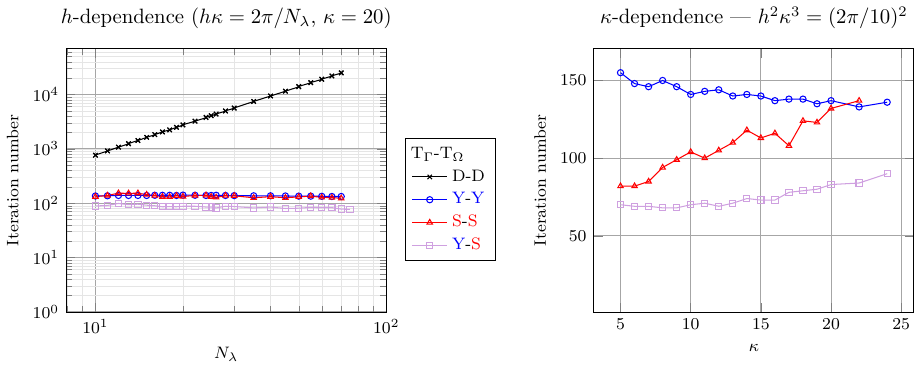}
  \caption{Homogeneous problem with strongly imposed Neumann boundary conditions and the Costabel coupling. Left, \(N_{it}\) with respect to \(N_{\lambda}\) for \(\kappa=20\) and \(h\kappa = 2\pi/N_{\lambda}\). Right, \(N_{it}\) with respect to the wavenumber \(\kappa\), for \(h^2\kappa^3 = (2\pi/10)^2\)}\label{fig:fem_bem_strong-Neumann}
\end{figure}

\subsubsection{Heterogeneous problem}\label{sssec:hom_toy_pb}

Keeping the same incident wave \(u_{i}(r,\theta) = \exp(i k \, r
  \cos(\theta))\), we now consider a heterogeneous acoustic lens: the
wavenumber is defined by \(\kappa(r, \theta)^2 = k^2 \ \eta(r, \theta)\),
with \(k > 0\) and
\begin{align*}
  \eta(r, \theta) \coloneqq 
  \begin{cases}
    1 &\text{if } r \geq 1\\
    \dfrac{2}{1+r^2} &\text{otherwise}
  \end{cases},
\end{align*}
meaning that the lens is a disk of radius \(r=1\).
We solve Problem~\eqref{IntialBVPb}, for which the FEM subdomain \(\Omega\) is a disk of radius \(r=2\), whose boundary is \(\Gamma\), and the BEM subdomain is its complement \(\mathbb{R}^d \setminus \overline{\Omega}\), see Figure~\ref{fig:heterogeneous_lens}. Following the observations made for the previous experiment, we focus on the Costabel coupling.

\begin{figure}[!ht]
  \begin{subfigure}{0.48\textwidth}
    \centering
    \includegraphics[scale=1]{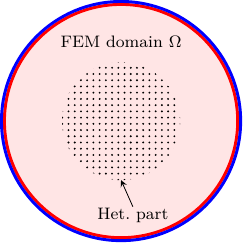}
  \end{subfigure}
  ~
  \begin{subfigure}{0.48\textwidth}
    \centering
    \includegraphics[scale=0.15]{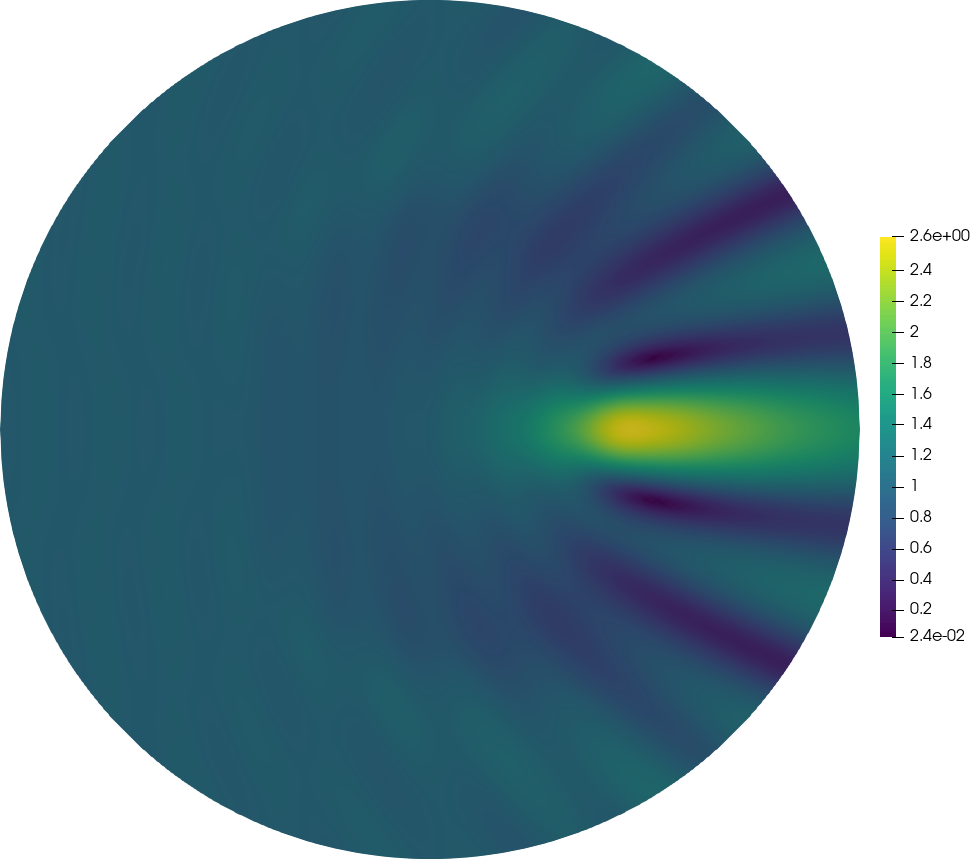}
  \end{subfigure}
  \caption{FEM subdomain \(\Omega\) with a heterogeneous pocket (left). Amplitude of the total field, for \SchurYukawa{} configuration, Costabel coupling and \(k = 10\), \(N_{\lambda} = 20\) (right)}\label{fig:heterogeneous_lens}
\end{figure}

Letting \(N_{\lambda}\) vary while taking \(k = 10\), the results are presented in \figpart{fig:fem-bem_heterogeneous}{left}. The observations are similar to the previous experiment of Section~\ref{subsubsec:homogenous_fem_bem}, and coherent with Theorem~\ref{thm:richardson_upper_bound}. Indeed, the number of iterations for \DespresDespres{} configuration evolves asymptotically as \(\mathcal{O}(h^{1.78})\), which is quasi-quadratic, while for the three non-local configurations \YukawaYukawa, \SchurYukawa{} and \SchurSchur, \(N_{it}\) is quasi-constant.

We now turn to \figpart{fig:fem-bem_heterogeneous}{right}, where it is shown how \(N_{it}\) evolves when \(k\) grows (numerical pollution is avoided in the same way as for the homogeneous problem).
We observe asymptotically a sublinear growth for \DespresDespres{} and \SchurSchur{}, respectively \(\mathcal{O}(k^{0.86})\) and \(\mathcal{O}(k^{0.4})\), while for \YukawaYukawa{} configuration, \(N_{it}\) is quasi-constant. 
For the \SchurYukawa{} configuration, the evolution of the number of iterations looks like having two states, as for the previous geometry: for \(\kappa \leq 15\) it is quasi-constant, and then it grows sublinearly. We interpret this behavior as a mix of those of \YukawaYukawa{} and \SchurSchur{} configurations. Additionally, we observe that for \(k=15\) and \(22\) the number of iterations of \SchurSchur{} has `bumps', which are also found for the \SchurYukawa{} configuration.



\begin{figure}[!ht]
    \includegraphics[scale=1]{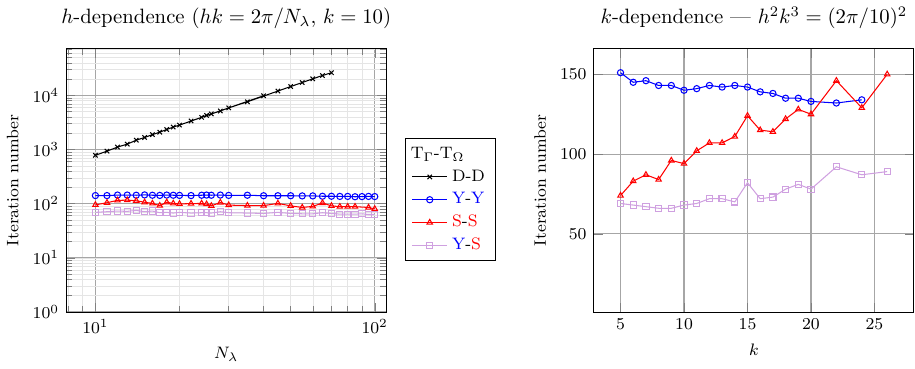}
  \caption{Heterogeneous problem with the Costabel coupling. Left, \(N_{it}\) with respect to \(N_{\lambda}\) for \(k=10\) and \(hk = 2\pi/N_{\lambda}\). Right, \(N_{it}\) with respect to \(k\) with \(h^2k^3 = (2\pi/10)^2\). The number of degrees of freedom of each simulation can be found in Table~\ref{table:dof_heterogeneous}}\label{fig:fem-bem_heterogeneous}
\end{figure}

\subsubsection{Conclusions}
\label{ssec:conclusionNum1}

Overall, for the two material settings the worst number of iterations is obtained with \DespresDespres{} configuration, whether when refining the mesh or for increasing \(\kappa\), and the \SchurYukawa{} configuration is the one for which \(N_{it}\) is the smallest. We remind that, in the latter configuration, each transmission operator is of the same type as its respective subproblem, see the discussion in the introduction of Section~\ref{ssec:fem-bem-strong}. 

Besides, note that configurations such as \SchurSchur{} and \YukawaYukawa{}, which involve the same transmission operator on a boundary shared by two subdomains \(\Omega_j\) and \(\Omega_k\), can only be considered when either \(\Gamma_j \subset \Gamma_k\) or \(\Gamma_k \subset \Gamma_j\). This property is not satisfied for partitions with cross-points, which are typically obtained with automatic graph partitioners. Hence, it is more realistic to use configurations such as the \SchurYukawa{} one, for which the transmission operator \(\mT_j\) associated with a domain \(\Omega_j\) is defined independently of its neighbouring domains.

\subsection{FEM-BEM coupling and weakly imposed boundary conditions}\label{ssec:fem-bem-weak}
Here, the space \(\RR^2\) is decomposed into three domains: \(\Omega_{\text{B}}\) homogeneous and unbounded, \(\Omega\) bounded and possibly heterogeneous, and \(\Omega_{\text{O}}\) impenetrable, see Figure~\ref{fig:fem-bem_weak_config}, and
we solve problems of the form~\eqref{eq:generalBVP}. Their reformulation~\eqref{eq:varfgeneralBVP} involves the operators \(\mathcal{A}_{\Gamma_{\text{B}}}\), \(\mathcal{A}_{\Omega}\) and \(\mathcal{A}_{\Gamma_{\text{O}}}\), where \(\mathcal{A}_{\Gamma_{\text{B}}}\) is one of the FEM-BEM couplings introduced in Section~\ref{FEMBEMCoupling}, while \(\mathcal{A}_{\Gamma_{\text{O}}}\) corresponds to one of the boundary condition operators of Section~\ref{sec:pbs}.

The goal of this section is to study the impact of the transmission operator chosen for \(\Gamma_2\) on the number of iterations. Since the skeleton \(\Sigma\) is now composed of the union of several boundaries, we add a subscript to each transmission operator in order to indicate the transmission boundary on which it is defined. Following the conclusions drawn in Section~\ref{ssec:conclusionNum1}, we choose \(\mT_{\text{B}} = \mT_{Y,\Gamma_{\text{B}}}\) and \(\mT = \mT_{S, \Omega}\). 
The transmission operator \(\mT_{\text{O}}\) is chosen among \(\mT_{Y,\Gamma_{\text{O}}}\), \(\mT_{S, \Omega_{\text{O}}}\) and \(\mT_{D, \Gamma_{\text{O}}}\). 
Using the same conventions as Definition~\ref{dfn:transmission_choices}, we denote the three possible configurations respectively \YukawaSchurYukawa, \YukawaSchurSchur{} and \YukawaSchurDespres.
We also consider a fourth choice named \DespresDespresDespres{} for which \(\mT_{\cdot} = \kappa\, \Id\) on each boundary, mainly to compare it with \YukawaSchurDespres.
Note that \(\mT_{Y,\Gamma_{\text{B}}}\) and \(\mT_{Y,\Gamma_{\text{O}}}\) are respectively defined by considering \(\Omega_{\text{B}}\) and \(\Omega_{\text{O}}\) as \emph{interior} domain (see Section~\ref{FEMBEMCoupling}), while the thin layer needed to define \(\mT_{S, \Omega}\) and \(\mT_{S, \Omega_{\text{O}}}\) are respectively taken inside \(\Omega\) and \(\Omega_{\text{O}}\).

\subsubsection{Homogeneous problem with weakly imposed boundary conditions}\label{sssec:fem-bem-weak-toy}

We consider the same problems~\eqref{pbm:fem-bem_dirichlet} and~\eqref{pbm:fem-bem_neumann} as in Section~\ref{subsubsec:homogenous_fem_bem}, but now the Dirichlet or Neumann boundary conditions are weakly imposed, see Remark~\ref{rem:weakstrongBC}. Thus, \(\Gamma_{\text{O}}\) has to be understood as a subdomain too, and is part of the skeleton \(\Sigma\), see Figure~\ref{fig:fem-bem_weak_config}. We emphasize that transmission conditions now also arise on \(\Gamma_{\text{O}}\), meaning that the boundary of \(\Omega\) involved in the GOSM is larger than the one of Section~\ref{ssec:fem-bem-strong}.



\begin{figure}[!ht]
  \centering
      \includegraphics[scale=1]{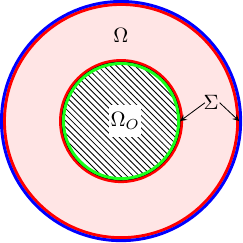}
  \caption{Skeleton \(\Sigma\) for FEM-BEM simulations with weakly imposed boundary conditions on \(\Gamma_2\). \(\Gamma_{\text{B}}\) is the blue circle, \(\Gamma_{\text{O}}\) is the green circle, and \(\partial \Omega = \Sigma = \Gamma_{\text{B}} \cup \Gamma_{\text{O}}\) is the union of the two red circles}\label{fig:fem-bem_weak_config}
\end{figure}

We start by studying the Dirichlet problem for a fixed wavenumber while \(N_{\lambda}\) varies;~results are in \figpart{fig:fem-bem_weak-Npwl_dependence}{left}. As expected, when \(\mT_{D,\cdot}\) is used for at least one domain (\YukawaSchurDespres{} and \DespresDespresDespres), we observe a quasi-quadratic growth of the number of iterations. Moreover, \(N_{it}\) is smaller with \YukawaSchurDespres{} configuration, in which only \(\mT_{D,\Gamma_{\text{O}}}\) is considered, than with \DespresDespresDespres{} configuration.
Nonetheless, the number of iterations is greater than the configurations involving only non-local operators: for \YukawaSchurYukawa{} and \YukawaSchurSchur{}, \(N_{it}\) remains quasi-constant, which is coherent with Theorem~\ref{thm:richardson_upper_bound}.
The same behaviors are observed when considering Equation~\eqref{pbm:fem-bem_neumann} instead, except for the \YukawaSchurDespres{} configuration where \(N_{it}\) is also quasi-constant.
One possible explanation is the following remark: the unknown \(p\) associated with \(\Gamma_{\text{O}}\) is theoretically equal to \(0\) (see~\cite[Example~3.2]{MR4665035}), and numerically has converged to \(0\) after one iteration. This might explain why the operator chosen on \(\Gamma_{\text{O}}\) has an impact on the number of iterations, but not on its behavior when the mesh is refined.

Turning to the \(\kappa\)-dependence results shown in \figpart{fig:fem-bem_weak-Npwl_dependence}{right}, we only study the configurations for which we have not observed any growth for \(N_{it}\) in the previous paragraph. For the Dirichlet problem, \(N_{it}\) seems quasi-constant, while for the Neumann problem, it grows sublinearly, as it was the case when considering strongly imposed boundary conditions, see \figpart{fig:fem_bem_strong-Neumann}{right}. The growth is nevertheless faster. Recovering for the Neumann problem the same behavior as when the boundary condition is strongly imposed might also be due to the artificial unknown \(p\) we introduced.


\subsubsection{Conclusions}\label{sssec:fem-bem-weak-toy-conclusions}

We conclude this study by pointing that, without considering the configurations involving the Després transmission operator, \(N_{it}\) is smaller for \YukawaSchurSchur{} than for \YukawaSchurYukawa{} with a Dirichlet boundary condition, while the opposite is true with a Neumann boundary condition. This indicates that the characterization of what is a good choice of transmission operator for weakly imposed boundary condition might be a matter of interest.


\begin{figure}[!ht]
  \includegraphics[scale=1]{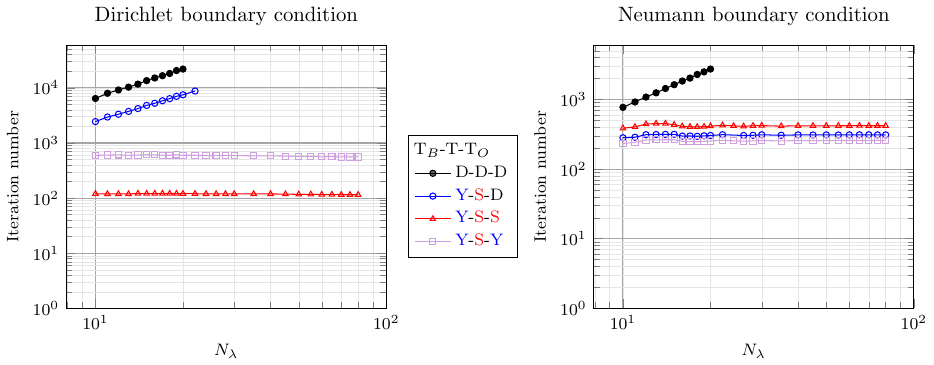} 
  \caption{\(N_{it}\) with respect to \(N_{\lambda}\). Homogeneous problem with weakly imposed boundary conditions (Dirichlet on the left and Neumann on the right), for \(\kappa = 20\), \(h\kappa = 2\pi/N_{\lambda}\), and the Costabel coupling
  }\label{fig:fem-bem_weak-Npwl_dependence}
\end{figure}


\begin{figure}[!ht]
  \includegraphics[scale=1]{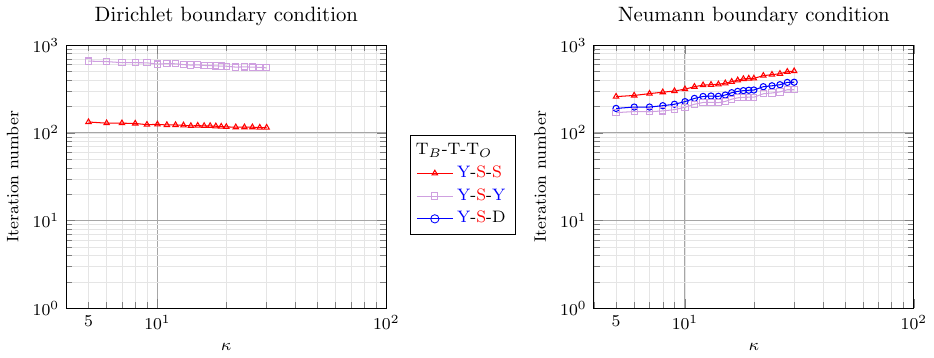}
  \caption{\(N_{it}\) with respect to the wavenumber \(\kappa\). Homogeneous problem with weakly imposed boundary conditions (Dirichlet on the left and Neumann on the right), for \(h^2\kappa^3 = (2\pi/10)^2\) and the Costabel coupling}\label{fig:fem-bem_weak-k_dependence}
\end{figure}

\subsubsection{Rectangular cavity problem: a configuration with cross-points}

We end our numerical studies by considering in \(\mathbb{R}^2\) the
scattering of a plane wave \(u_{i}(x,y) = \exp(i \kappa (x\, \cos(\theta) + y
  \, \sin(\theta)))\) by an open sound-soft rectangular cavity, with an
incident angle \(\theta=4\pi/10\). This is modelled by
Problem~\eqref{eq:generalBVP}, where \(\Omega_{\text{O}}\) is the
rectangular cavity, the FEM subdomain \(\Omega\) is the space inside the
cavity, \(\Omega_{\text{B}} = \mathbb{R}^d \setminus (\overline{\Omega}
\cup \overline{\Omega_{\text{O}}})\). The wavenumber is assumed constant in
\(\mathbb{R}^d \setminus \overline{\Omega_{\text{O}}}\), a Dirichlet
boundary condition \(u = -u_i\) is applied on \(\Gamma_{\text{O}}\), and
the volume source term \(f\) is null.  Note that the domain decomposition
involves cross-points as shown in
Figure~\ref{fig:cross_points-config}. This last numerical experiment aims
to show the robustness to cross-points of the GOSM, even when FEM-BEM
coupling is considered. More precisely, the cavity \(\Omega_{\text{O}}\) is
an \(L_{\text{O}}\times l_{\text{O}}=1.5\times0.6\) open rectangle, while
\(\Omega\) is an \(L\times l=1.3 \times 0.4\) rectangle.  Note that similar
geometries have already been considered
in~\cite{ChandlerWildeGrahamEtAl2012NAB,DoleanMarchandEtAl2025CAG}. We also
emphasize that the idea of ``closing'' a cavity in order to uncouple it
from the rest of the space has been studied in~\cite{Bourguignon2011PMD}
for electromagnetic wave scattering problems.

\begin{figure}[!ht]
  \centering
      \includegraphics[scale=2]{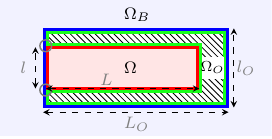}
  \caption{Domain decomposition for the rectangular cavity problem. Cross-points are circled in gray}\label{fig:cross_points-config}
\end{figure}

We emphasize that Problem~\eqref{eq:generalBVP} admits a unique solution for any wavenumber. Nonetheless, due to the open cavity, the problem admits \emph{quasimodes}, i.e., functions that approximately satisfy the homogeneous problem and are spatially localized, see~\cite[Section~5]{DoleanMarchandEtAl2025CAG}.
For a Dirichlet boundary condition, it has been shown in~\cite[Section~5.6.2]{ChandlerWildeGrahamEtAl2012NAB} that a family of quasimodes exists for \(\kappa=k_n = \pi n / l\), \(n \in \mathbb{N}^*\), which are eigenvalues of the Laplacian problem with homogeneous Dirichlet boundary conditions set in \(\Omega\). We will see that the presence of quasimodes induces an increase of the number of iterations.

The subdomains \(\Omega\) and \(\Omega_{\text{O}}\) being of small size, the transmission operators \(\mT_{S, \Omega}\) and \(\mT_{S, \Omega_{\text{O}}}\) are derived from the discretization of a positive \(\DtN\) problem set in the whole domains \(\Omega\) and \(\Omega_{\text{O}}\) respectively, and not in a thin layer along their boundary, contrary to Section~\ref{sec:transmission_operator}.

\begin{figure}[!ht]
  \includegraphics[scale=1]{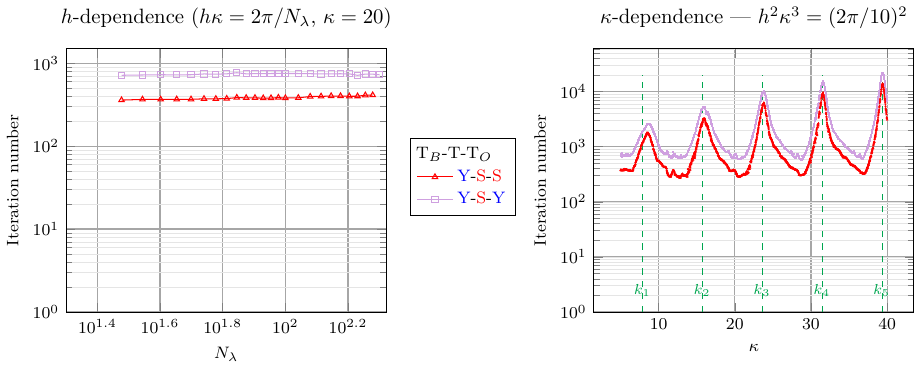}
  \caption{Open rectangular cavity problem with the Costabel coupling. Left, \(N_{it}\) with respect to \(N_{\lambda}\) for \(\kappa=20\) and \(h \kappa = 2\pi/N_{\lambda}\). Right, \(N_{it}\) with respect to \(\kappa\) with \(h^2k^3 = (2\pi/10)^2\). The dashed green lines corresponds to \((k_n)_{n \in \mathbb{N}}\). The number of degrees of freedom of each simulation can be found in Table~\ref{table:dof_cavity}}\label{fig:fem-bem-rectangular_cavity-Dirichlet}
\end{figure}
As in the previous sections, we start by studying the problem for a fixed wavenumber while \(N_{\lambda}\) varies;~results are in \figpart{fig:fem-bem-rectangular_cavity-Dirichlet}{left} for \(\kappa = 20\). In view of the observations drawn in Section~\ref{sssec:fem-bem-weak-toy}, we only study \YukawaSchurSchur{} and \YukawaSchurYukawa{} configurations. We recover the same result, namely \(N_{it}\) is nearly \(h\)-independent for both configurations. Thus, the method appears to be robust to cross-points.

The behavior of the \(\kappa\)-dependence, which can be seen in
\figpart{fig:fem-bem-rectangular_cavity-Dirichlet}{right}, is different
from the one analyzed for the homogeneous problem and its simpler geometry
in Section~\ref{sssec:fem-bem-weak-toy}.  Indeed, several peaks arise
around the wavenumbers \(k_n\), represented by the vertical green dashed
lines. We underline that the peaks are \emph{not} due to some spurious
resonances of BIOs, which are the \(k_{\text{B},n \, m} = \pi
\sqrt{m^2/L_{\text{B}}^2 + n^2/l_{\text{B}}^2}\) for \(m,n\in
\mathbb{N}^*\). In such geometry, this phenomenon is not specific to
our method, and has already been observed for FEM
in~\cite{DoleanMarchandEtAl2025CAG}.  Once again, when Dirichlet boundary
conditions are imposed, we observe that \(N_{it}\) is smaller for
\YukawaSchurSchur{} than for \YukawaSchurYukawa, as highlighted in
\cref{sssec:fem-bem-weak-toy-conclusions} for a geometry without
cross-points.

\begin{appendix}
  \section{Problem sizes for the numerical experiments of Section~\ref{ssec:fem-bem-strong}}
\begin{table}[ht]
  \resizebox{\textwidth}{!}{
  \begin{tabular}{l|cc}
    \multicolumn{1}{c|}{\(N_{\lambda}\)} & \(\Omega\) & \(\Gamma\) \\ \hline
    10                                   & 11,458      & 400        \\
    11                                   & 13,799      & 440        \\
    12                                   & 16,395      & 480        \\
    13                                   & 19,184      & 520        \\
    14                                   & 22,208      & 560        \\
    15                                   & 25,436      & 600        \\
    16                                   & 28,913      & 640        \\
    17                                   & 32,576      & 680        \\
    18                                   & 36,479      & 720        \\
    19                                   & 40,620      & 760        \\
    20                                   & 44,957      & 800        \\
    22                                   & 54,308      & 880        
  \end{tabular}
  \hspace{0.2cm}
  \begin{tabular}{l|cc}
    \multicolumn{1}{c|}{\(N_{\lambda}\)} & \(\Omega\) & \(\Gamma\) \\ \hline
    24                                   & 64,527      & 960         \\
    26                                   & 75,638      & 1,040       \\
    28                                   & 87,645      & 1,120       \\
    30                                   & 100,493     & 1,200       \\
    35                                   & 136,548     & 1,400       \\
    40                                   & 178,104     & 1,600       \\
    45                                   & 225,197     & 1,800       \\
    50                                   & 277,787     & 2,000       \\
    55                                   & 335,858     & 2,200       \\
    60                                   & 399,529     & 2,400       \\
    65                                   & 468,653     & 2,600       \\
    70                                   & 543,254     & 2,800       
  \end{tabular}
  \hspace{1.0cm}
  \begin{tabular}{l|cc}
    \multicolumn{1}{c|}{\(\kappa\)} & \(\Omega\) & \(\Gamma\) \\ \hline
    5                               & 3,688      & 224        \\
    6                               & 6,272      & 294        \\
    7                               & 9,891      & 371        \\
    8                               & 14,636     & 453        \\
    9                               & 20,664     & 540        \\
    10                              & 28,316     & 633        \\
    11                              & 37,479     & 730        \\
    12                              & 48,579     & 832        \\
    13                              & 61,604     & 938        \\
    14                              & 76,798     & 1,048      \\
    15                              & 94,286     & 1,162      \\
    16                              & 114,258    & 1,280      
  \end{tabular}
  \hspace{0.2cm}
  \begin{tabular}{lcc}
    \multicolumn{1}{c|}{\(\kappa\)} & \(\Omega\) & \(\Gamma\) \\ \hline
    \multicolumn{1}{l|}{17}         & 136,924    & 1,402      \\
    \multicolumn{1}{l|}{18}         & 162,515    & 1,528      \\
    \multicolumn{1}{l|}{19}         & 190,949    & 1,657      \\
    \multicolumn{1}{l|}{20}         & 222,450    & 1,789      \\
    \multicolumn{1}{l|}{22}         & 295,758    & 2,064      \\
    \multicolumn{1}{l|}{24}         & 383,709    & 2,352      \\
    &&\\
    &&\\
    &&\\
    &&\\
    &&\\
    &&
  \end{tabular}
  }
  \caption{Number of degrees of freedom in each computational domain of the homogeneous geometry, for the \(h\)-dependence (left) and \(\kappa\)-dependence (right) experiments related to Figures~\ref{fig:fem_bem_strong-Dirichlet} and~\ref{fig:fem_bem_strong-Neumann}}
  \label{table:dof_homogeneous}
\end{table}

\begin{table}[ht]
  \resizebox{\textwidth}{!}{
  \begin{tabular}{l|cc}
    \multicolumn{1}{c|}{\(N_{\lambda}\)} & \(\Omega\) & \(\Gamma\) \\ \hline
    10                                   & 3,814      & 200        \\
    11                                   & 4,608      & 220        \\
    12                                   & 5,458      & 240        \\
    13                                   & 6,387      & 260        \\
    14                                   & 7,410      & 280        \\
    15                                   & 8,479      & 300        \\
    16                                   & 9,637      & 320        \\
    17                                   & 10,875     & 340        \\
    18                                   & 12,168     & 360        \\
    19                                   & 13,545     & 380        \\
    20                                   & 14,986     & 400        \\
    22                                   & 18,115     & 440        \\
    24                                   & 21,508     & 480        \\
    26                                   & 25,215     & 520        \\
    28                                   & 29,207     & 560        
  \end{tabular}
  \hspace{0.2cm}
  \begin{tabular}{l|cc}
    \multicolumn{1}{c|}{\(N_{\lambda}\)} & \(\Omega\) & \(\Gamma\) \\ \hline
    30                                   & 33,499     & 600        \\
    35                                   & 45,517     & 700        \\
    40                                   & 59,379     & 800        \\
    45                                   & 75,065     & 900        \\
    50                                   & 92,592     & 1,000      \\
    55                                   & 111,974    & 1,100      \\
    60                                   & 133,162    & 1,200      \\
    65                                   & 156,221    & 1,300      \\
    70                                   & 181,095    & 1,400       \\
    75                                   & 207,804    & 1,500      \\
    80                                   & 236,372    & 1,600      \\
    85                                   & 266,755    & 1,700      \\
    90                                   & 298,998    & 1,800      \\
    95                                   & 333,058    & 1,900      \\
    100                                  & 368,944    & 2,000     
  \end{tabular}
  \hspace{1.0cm}
  \begin{tabular}{l|cc}
    \multicolumn{1}{c|}{\(k\)} & \(\Omega\) & \(\Gamma\) \\ \hline
    5                               & 4,774      & 224        \\
    6                               & 8,145      & 294        \\
    7                               & 12,904     & 371        \\
    8                               & 19,177     & 453        \\
    9                               & 27,175     & 540        \\
    10                              & 37,270     & 633        \\
    11                              & 49,496     & 730        \\
    12                              & 64,200     & 832        \\
    13                              & 81,516     & 938        \\
    14                              & 101,663    & 1,048      \\
    15                              & 124,892    & 1,162      \\
    16                              & 151,473    & 1,280      \\
    17                              & 181,617    & 1,402      \\
    18                              & 215,631    & 1,528      \\
    19                              & 253,461    & 1,657      \\
  \end{tabular}
  \hspace{0.2cm}
  \begin{tabular}{lcc}
    \multicolumn{1}{c|}{\(k\)} & \(\Omega\) & \(\Gamma\) \\ \hline
    \multicolumn{1}{l|}{20}         & 295,364    & 1,789      \\
    \multicolumn{1}{l|}{22}         & 392,903    & 2,064      \\
    \multicolumn{1}{l|}{24}         & 509,959    & 2,352      \\
    && \\
    && \\
    && \\
    && \\
    && \\
    && \\
    && \\
    && \\
    && \\
    && \\
    && \\
    &&
  \end{tabular}
  }
  \caption{Number of degrees of freedom in each computational domain of the heterogeneous geometry, for the \(h\)-dependence (left) and \(k\)-dependence (right) experiments related to Figure~\ref{fig:fem-bem_heterogeneous}}
  \label{table:dof_heterogeneous}
\end{table}

\begin{table}[ht]
  \resizebox{\textwidth}{!}{
  \begin{tabular}{l|ccc}
    \multicolumn{1}{c|}{\(N_{\lambda}\)} & \(\Gamma_{\text{B}}\) & \(\Omega\) & \(\Gamma_{\text{O}}\) \\ \hline
    30                                   & 405          & 5,818      & 655          \\
    35                                   & 472          & 7,762      & 762         \\
    40                                   & 536          & 10,083     & 868          \\
    45                                   & 604          & 12,818     & 978          \\
    50                                   & 670          & 15,623     & 1,084         \\
    55                                   & 739          & 18,894     & 1,195         \\
    60                                   & 806          & 22,545     & 1,304         \\
    65                                   & 872          & 26,216     & 1,410         \\
    70                                   & 940          & 30,361     & 1,520         \\
    75                                   & 1,006        & 34,993     & 1,628         \\
    80                                   & 1,071        & 39,694     & 1,735         \\
    85                                   & 1,140        & 44,657     & 1,844       
  \end{tabular}
  \hspace{0.2cm}
  \begin{tabular}{l|ccc}
    \multicolumn{1}{c|}{\(N_{\lambda}\)} & \(\Gamma_{\text{B}}\) & \(\Omega\) & \(\Gamma_{\text{O}}\) \\ \hline
    90                                   & 1,205        & 50,170     & 1,951         \\
    95                                   & 1,273        & 55,805     & 2,061         \\
    100                                  & 1,339        & 61,602     & 2,167         \\
    110                                  & 1,476        & 74,847     & 2,388         \\
    120                                  & 1,607        & 88,687     & 2,601         \\
    130                                  & 1,741        & 103,819    & 2,817         \\
    140                                  & 1,875        & 120,519    & 3,035         \\
    150                                  & 2,008        & 137,954    & 3,250         \\
    160                                  & 2,140        & 157,375    & 3,466         \\
    170                                  & 2,276        & 177,208    & 3,684         \\
    180                                  & 2,410        & 199,139    & 3,900         \\
    190                                  & 2,543        & 221,410    & 4,117       
  \end{tabular}
  \hspace{1.0cm}
  \begin{tabular}{l|ccc}
    \multicolumn{1}{c|}{\(\kappa\)} & \(\Gamma_{\text{B}}\) & \(\Omega\) & \(\Gamma_{\text{O}}\) \\ \hline
    5.00                            & 77           & 256        & 125          \\
    6.50                            & 113          & 497        & 183          \\
    \textbf{8.47}                   & 166          & 997        &  268         \\
    9.50                            & 197          & 1,434      & 319          \\
    11.00                           & 247          & 2,181      & 399          \\
    12.50                           & 300          & 3,219      & 484          \\
    14.00                           & 355          & 4,427      & 573          \\
    15.50                           & 410          & 5,960      & 664          \\
    \textbf{16.05}                  & 433          & 6,608      & 701          \\
    18.50                           & 533          & 9,978      & 863          \\
    20.00                           & 601          & 12,604     & 973          \\
    21.50                           & 668          & 15,623     & 1,082     
  \end{tabular}
  \hspace{0.2cm}
  \begin{tabular}{l|ccc}
    \multicolumn{1}{c|}{\(\kappa\)} & \(\Gamma_{\text{B}}\) & \(\Omega\) & \(\Gamma_{\text{O}}\) \\ \hline
    \textbf{23.8}                   & 779          & 20,999     & 1,261        \\
    24.50                           & 814          & 22,949     & 1,316        \\
    26.00                           & 890          & 27,359     & 1,440        \\
    28.00                           & 993          & 33,992     & 1,607        \\
    30.00                           & 1,102        & 41,660     & 1,782        \\
    \textbf{31.59}                  & 1,190        & 48,873     & 1,926        \\
    33.50                           & 1,298        & 58,089     & 2,102        \\
    35.00                           & 1,386        & 66,085     & 2,244        \\
    36.50                           & 1,478        & 75,175     & 2,392        \\
    38.00                           & 1,570        & 84,770     & 2,540        \\
    \textbf{39.39}                  & 1,657        & 94,004     & 2,681        \\
    40.00                           & 1,694        & 98,795     & 2,742         
  \end{tabular}
  }
  \caption{Number of degrees of freedom in each computational domain of the rectangular cavity geometry, for the \(h\)-dependence (left) and \(k\)-dependence (right) experiments related to Figure~\ref{fig:fem-bem-rectangular_cavity-Dirichlet}. In bold, the wavenumbers associated with peaks tops in Figure~\ref{fig:fem-bem-rectangular_cavity-Dirichlet}}
  \label{table:dof_cavity}
\end{table}

\end{appendix}

\bibliography{manuscrit}
\bibliographystyle{plain}

\end{document}